\documentclass[oneside,english]{amsart}
\usepackage[T1]{fontenc}
\usepackage[latin9]{inputenc}
\usepackage{mathtools}
\usepackage{enumitem}
\usepackage{amstext}
\usepackage{amsthm}
\usepackage{amssymb}

\usepackage{xcolor}
\usepackage[normalem]{ulem}
\usepackage{hyperref}

\usepackage{verbatim}

\makeatletter
\numberwithin{equation}{section}
\numberwithin{figure}{section}
\theoremstyle{plain}
\newtheorem{thm}{\protect\theoremname}
\theoremstyle{definition}
\newtheorem{defn}[thm]{\protect\definitionname}
\theoremstyle{remark}
\newtheorem*{rem*}{\protect\remarkname}
\theoremstyle{plain}
\newtheorem{fact}[thm]{\protect\factname}
\theoremstyle{remark}
\newtheorem{rem}[thm]{\protect\remarkname}
\theoremstyle{plain}
\newtheorem{lem}[thm]{\protect\lemmaname}
\newlist{casenv}{enumerate}{4}
\setlist[casenv]{leftmargin=*,align=left,widest={iiii}}
\setlist[casenv,1]{label={{\itshape\ \casename} \arabic*.},ref=\arabic*}
\setlist[casenv,2]{label={{\itshape\ \casename} \roman*.},ref=\roman*}
\setlist[casenv,3]{label={{\itshape\ \casename\ \alph*.}},ref=\alph*}
\setlist[casenv,4]{label={{\itshape\ \casename} \arabic*.},ref=\arabic*}
\theoremstyle{plain}

\theoremstyle{plain}
\newtheorem{prop}[thm]{\protect\propositionname}
\theoremstyle{plain}
\newtheorem*{cor*}{\protect\corollaryname}
\theoremstyle{plain}
\newtheorem*{lem*}{\protect\lemmaname}

\newtheorem{maintheorem}{Theorem}

\def\N{\mathbb N}
\def\Z{\mathbb Z}
\def\R{\mathbb R}

\def\T{\mathbb T}
\def\C{\mathbb C}

\def\h{\mathit h}

\def\d{\delta}

\newcommand{\rr}{\color{red}}

\newcommand{\Meng}[2]{\left\{#1\mathrel{}\middle|\mathrel{}#2\right\}}
\newcommand{\Ordered}[2]{\left\langle#1\mathrel{}\middle|\mathrel{}#2\right\rangle}
\newcommand{\abs}[1]{\left\lvert#1\right\rvert}
\newcommand{\trees}{\mathcal{T}\kern-.5mm rees}
\newcommand{\MPT}{\mathcal{X}}
\newcommand{\freq}{\operatorname{freq}}
\newcommand{\sh}{\operatorname{sh}}

\usepackage{babel}
\providecommand{\lemmaname}{Lemma}
\providecommand{\propositionname}{Proposition}
\providecommand{\theoremname}{Theorem}

\makeatother

\usepackage{babel}
\providecommand{\casename}{Case}
\providecommand{\corollaryname}{Corollary}
\providecommand{\definitionname}{Definition}
\providecommand{\factname}{Fact}
\providecommand{\lemmaname}{Lemma}
\providecommand{\propositionname}{Proposition}
\providecommand{\remarkname}{Remark}
\providecommand{\theoremname}{Theorem}

\begin{document}
	\title{Anti-classification results for weakly mixing diffeomorphisms}
	\author{Philipp Kunde} 
	\thanks{Faculty of Mathematics and Computer Science, Jagiellonian University in Krakow, ul. Lojasiewicza 6, 30-348 Kraków, Poland. This research is part of the project No. 2021/43/P/ST1/02885 co-funded by the National Science Centre
		and the European Union's Horizon 2020 research and innovation programme under the Marie Sklodowska-Curie
		grant agreement no. 945339.}
	\begin{abstract}
		We extend anti-classification results in ergodic theory to the collection of weakly mixing systems by proving that the isomorphism relation as well as the Kakutani equivalence relation of weakly mixing invertible
		measure-preserving transformations are not Borel sets. This shows
		in a precise way that classification of weakly mixing systems up to isomorphism or Kakutani equivalence is
		impossible in terms of computable invariants, even with a very inclusive understanding
		of ``computability''. We even obtain these anti-classification results for weakly mixing
		area-preserving smooth diffeomorphisms on compact surfaces admitting
		a non-trivial circle action as well as real-analytic diffeomorphisms
		on the $2$-torus.
	\end{abstract}
	
	\maketitle
	
	\insert\footins{\footnotesize - 
		}
	
	\section{\label{sec:Introduction}Introduction}
	One of the oldest and most influential problems in ergodic theory is the classification of systems up to appropriate equivalence relations. Dating back to the foundational paper \cite{Ne} by J.~von Neumann the \emph{isomorphism problem} asks to classify measure-preserving transformations (MPT's) up to isomorphism. By an MPT we mean a measure-preserving automorphism of a standard non-atomic probability space and we let $\MPT$ denote the set of all MPT's of a fixed standard non-atomic probability space $(\Omega,\mathcal{M},\mu)$. We endow $\MPT$ with the weak topology (see Section~\ref{subsec:erg}). This topology is compatible with a complete separable metric and hence makes $\MPT$ into a Polish space. We also recall that two automorphisms $S,T\in\MPT$ are \emph{isomorphic} (written $S \cong T$) if there exists $\varphi\in\MPT$ such that $S\circ\varphi$ and $\varphi\circ T$ agree $\mu$-almost everywhere. 	
	The isomorphism problem has been a guiding light for directions of research within ergodic theory, and has been solved only for some special classes of transformations. Two great successes are the classification of ergodic MPT's with pure point spectrum by the spectrum of the associated Koopman operator \cite{HN42} and the classification of Bernoulli shifts by their measure-theoretic entropy \cite{Or70}. Many properties of transformations like mixing of various types or finite rank have been characterized and studied in connection with the isomorphism problem but the general problem remained intractable. 
	
	Starting in the late 1990's, so-called \textit{anti-classification results} have been established. We refer to the survey article \cite{Fsurvey} by M.~Foreman for an overview of complexity results of structure and classification of dynamical systems. These results rigorously demonstrated that von Neumann's isomorphism problem is impossible: In their landmark paper \cite{FRW} M.~Foreman, D.~Rudolph and B.~Weiss showed that the measure-isomorphism relation for ergodic MPT's is not a Borel set. Informally speaking, this result says that determining isomorphism between ergodic transformations is inaccessible to countable methods that use countable amounts of information. This result leads to several new questions.
	
	On the one hand, one can consider other equivalence relations instead of isomorphism. Since all ergodic MPT's are orbit equivalent according to Dye's Theorem, Kakutani equivalence is the best known and most natural equivalence relation for which the classification problem can be considered. Two ergodic automorphisms $T$ and $S$ of $(\Omega,\mathcal{M},\mu)$ are said to be \emph{Kakutani equivalent} (written $T \sim S$) if there exist positive measure subsets $A$ and $B$ of $\Omega$ such that the first return maps\footnote{For an automorphism $T$ of $(\Omega,\mathcal{M},\mu)$ and $A \in \mathcal{M}$ with $\mu(A)>0$, let $n_A(x)= \inf \Meng{i\in \mathbb{Z}^+}{T^i(x) \in A}$ be the \emph{first return time} to $A$. Then we define the \emph{first return map} (also called \emph{induced map}) $T_A : A \to A$ by $T^{n_A(x)}(x)$. Furthermore, the normalized induced measure on $A$ is defined by $\mu_A(E) = \frac{\mu(A\cap E)}{\mu(A)}$ for $E \in \mathcal{M}$.} $(T_A,\mu_A)$ and $(S_B,\mu_B)$ are isomorphic. This equivalence relation was introduced by S.~Kakutani in \cite{Kak}. In that work he also proved that two ergodic flows are isomorphic up to a time change if and only if they have Kakutani equivalent transformations as cross-sections. Since the Kakutani equivalence relation is weaker than isomorphism, one might expect classification to be simpler. However, in joint work with M.~Gerber we showed that the Kakutani equivalence relation for ergodic MPT's is not a Borel set \cite{GK3}. In fact, our result also holds for any equivalence relation between Kakutani equivalence and isomorphism.
	 
	On the other hand, the classification problem (with respect to isomorphism or Kakutani equivalence) can be restricted to the class of smooth ergodic diffeomorphisms of a compact manifold $M$ that preserve a smooth measure. In a recent series of papers \cite{FW1,FW2,FW3}, Foreman and Weiss extended the aforementioned anti-classification result from \cite{FRW} to the $C^{\infty}$ category by proving that the measure-isomorphism relation among pairs of volume-preserving ergodic $C^{\infty}$-diffeomorphisms  on compact surfaces admitting a non-trivial circle action is not a Borel set with respect to the $C^{\infty}$-topology. In joint work with S.~Banerjee we are even able to show this anti-classification result for real-analytic diffeomorphisms of the $2$-torus \cite{BK2}. Using the realization techniques from these papers, the anti-classification results for Kakutani equivalence can be obtained both in the smooth and real-analytic setting \cite[Theorems~25 and~26]{GK3}.
	
	The aforementioned anti-classification results emphasize the significance of restricting the classification problem to specific classes of dynamical systems (like in the successful classification of Bernoulli shifts via entropy). We want to investigate the classification problem for classes of transformations that are more random than ergodic systems but less random than Bernoulli shifts, such as weakly mixing transformations, mixing transformations, and $K$-automorphisms. 
	
	In the current paper we obtain anti-classification results for weakly mixing transformations (addressing Problem~4 in~\cite{Fsurvey}). We recall that $(\Omega,\mathcal{M},\mu, T)$ is said to be
	weakly mixing if there is no nonconstant function $h\in L^2(\Omega,\mu)$ such that $h(Tx)=\lambda \cdot h(x)$ for some $\lambda \in \C$. Equivalently, $(\Omega,\mathcal{M},\mu, T)$ is weakly mixing iff for every pair $A,B \in \mathcal{M}$ we have
	$
	\lim_{n\to \infty} \frac{1}{n}\sum^{n-1}_{k=1}\abs{\mu\left(T^{-k}(A)\cap B\right) - \mu(A)\mu(B)} =0.
	$
	The collection $\mathcal{WM}$ of weakly mixing transformations is a dense $G_{\delta}$ subset of $\MPT$ endowed with the weak topology \cite{Ha}. Hence, the topology induced on $\mathcal{WM}$ is Polish as well. By genericity of $\mathcal{WM}$ and the turbulence result from \cite{FW0} we know that there are no complete algebraic invariants for the isomorphism relation on $\mathcal{WM}$ (we refer to \cite{FW0} or \cite[section 5.5]{Fsurvey} for details on the concept of turbulence developed by G.~Hjorth to show that an equivalence relation is not reducible to an $S_{\infty}$-action). It is an open problem if the isomorphism relation on ergodic diffeomorphisms is turbulent (see \cite[Problem 1]{FW3}).
	
	In this paper we obtain stronger anti-classification results for weakly mixing transformations as well as diffeomorphisms. To state our result precisely we let
	\begin{align*}
		\mathcal{R}_{\text{iso}} \coloneqq \Meng{(S,T) \in \MPT \times \MPT}{S \cong T}, \ \ 
		\mathcal{R}_{\text{Kak}} \coloneqq \Meng{(S,T) \in \MPT \times \MPT}{S \sim T}.
	\end{align*}
    Then we show the unclassifiability of weakly mixing transformations with respect to isomorphism, Kakutani equivalence, and any equivalence relation
    between them.
    \begin{maintheorem}
    	\label{thm:mpt}Let $\mathcal{R}$ be any equivalence relation on $\MPT$ satisfying $\mathcal{R}_{\text{iso}}\subseteq \mathcal{R} \subseteq \mathcal{R}_{\text{Kak}}$. Then the collection
    	\[
    	\left\{ (S,T):S\text{ and }T\text{ are weakly mixing and $\mathcal{R}$-equivalent}\right\} \subset \MPT\times\MPT
    	\]
    	is a complete analytic set. In particular, it is not Borel.
    \end{maintheorem}
	
	For instance, our result holds for even equivalence and $\alpha$-equivalence. It is also worth to mention that our weakly mixing systems have measure-theoretic entropy zero.
	In a recent paper~\cite{GK4} with M.~Gerber we obtain analogous anti-classification results for the even more restricted class of $K$-automorphisms (that clearly have positive measure-theoretic entropy). Therein we even obtain anti-classification results for $K$-automorphisms that are smooth diffeomorphisms on the five-dimensional torus. The methods of the current paper allow us to obtain anti-classification results for weakly mixing diffeomorphisms on some two-dimensional manifolds. Since the $C^{\infty}$-topology refines the weak topology, the weakly mixing diffeomorphisms
	are still a $G_{\delta}$-set in the Polish space $\text{Diff}^{\,\infty}_{\,\lambda}(M)$ of measure-preserving diffeomorphisms.
	
		\begin{maintheorem}
		\label{thm:smooth}Let $M$ be the disk, annulus or torus with Lebesgue
		measure $\lambda$. Furthermore, let $\mathcal{R}$ be any equivalence relation on $\text{Diff}^{\,\infty}_{\,\lambda}(M)$ satisfying $$\mathcal{R}_{\text{iso}}\cap \left(\text{Diff}^{\,\infty}_{\,\lambda}(M) \times \text{Diff}^{\,\infty}_{\,\lambda}(M)\right) \subseteq \mathcal{R} \subseteq \mathcal{R}_{\text{Kak}}\cap \left(\text{Diff}^{\,\infty}_{\,\lambda}(M) \times \text{Diff}^{\,\infty}_{\,\lambda}(M)\right).$$ Then the collection
		\[
		\left\{ (S,T):S\text{ and }T\text{ are weakly mixing diffeomorphisms and $\mathcal{R}$-equivalent}\right\} 
		\]
		in $\text{Diff}^{\,\infty}_{\,\lambda}(M)\times\text{Diff}^{\,\infty}_{\,\lambda}(M)$
		is a complete analytic set and, hence, not a Borel set with respect to the $C^{\infty}$ topology.
	\end{maintheorem}
	
	Since $\mathcal{R}$-equivalence for measure-preserving diffeomorphisms is reducible to $\mathcal{R}$-equivalence on $\MPT$ (see Definition~\ref{defn:reduction} for the notion of a reduction), Theorem~\ref{thm:mpt} immediately follows from Theorem~\ref{thm:smooth}.
	
Theorem~\ref{thm:smooth} is related to another major question in ergodic theory dating back to the pioneering paper~\cite{Ne}: The \textit{smooth realization problem} asks whether there are smooth versions of the objects and concepts in abstract ergodic theory and whether every ergodic measure-preserving transformation has a smooth model. Here, a smooth model of an MPT $(\Omega,\mu, T)$ is a smooth diffeomorphism $f$ of a compact manifold $M$ preserving a measure $\lambda$ equivalent to the volume element such that the MPT $(M,\lambda,f)$ is isomorphic to the MPT $(\Omega,\mu, T)$. The only known general restriction is due to Kushnirenko who proved that such a diffeomorphism must have finite entropy. There are restrictions in low dimension: Any circle diffeomorphism with invariant smooth measure is conjugate to a rotation and any weakly mixing surface diffeomorphism of positive measure-theoretic entropy is Bernoulli by Pesin theory \cite{Pe1}. Thus, weakly mixing surface diffeomorphism of positive measure-theoretic entropy are classifiable by entropy. 

Apart from Kushnirenko's result excluding smooth models of infinite entropy MPT's, there is a lack of general results on the smooth realization problem. One of the most powerful tools of constructing smooth volume-preserving diffeomorphisms of entropy zero with prescribed ergodic or topological properties is the so-called \emph{approximation by conjugation} method (also known as the \emph{AbC} method or \emph{Anosov-Katok} method) developed by D.~Anosov and A.~Katok in a highly influential paper \cite{AK70}. We refer to the survey articles \cite{FK} and \cite{Ksurvey} for expositions of the AbC method and its wide range of applications in dynamics. In particular, it provided the first examples of weakly mixing $C^{\infty}$ diffeomorphisms on the disk $\mathbb{D}$ in \cite[section~5]{AK70}. The AbC method is also used to construct weakly mixing diffeomorphisms preserving additional properties like a measurable Riemannian metric \cite{GK} or a prescribed Liouville rotation number \cite{FS}. 

Furthermore, the AbC method plays a key role in transfering the anti-classification results for ergodic MPT's in \cite{FRW} and \cite{GK3} to the smooth setting. In \cite{FW1}, Foreman and Weiss found a class of symbolic systems (the so-called \emph{circular systems}) that are realizable as smooth diffeomorphisms using the \emph{untwisted} version of the AbC method (i.\,e. the conjugation map in the AbC construction maps its fundamental domain into itself). Then they showed in \cite{FW2} that there is a functor between the class of MPT's with an odometer factor and the class of circular systems that preserves factor and isomorphism structure. This functor allows them in \cite{FW3} to transform the odometer-based systems from \cite{FRW} to circular systems which are then realized as ergodic diffeomorphisms using the untwisted AbC method. Since untwisted AbC transformations cannot be weakly mixing \cite[Proposition 8.1]{Ka03}, we design a specific twisted version of the AbC method that allows us to produce weakly mixing systems with a manageable symbolic representation. This takes on a project proposed in \cite{FW1} to find symbolic representations for other versions of the AbC method. In our case, the associated \emph{twisted symbolic systems} will serve as counterpart of the circular systems in the Foreman-Weiss' series of papers. After some small modificatios to the constructions in \cite{GK3} we transform those systems to the twisted symbolic systems that we can realize as weakly mixing diffeomorphisms. This allows us to deduce Theorem~\ref{thm:smooth}. We refer to Section~\ref{subsec:outline} for a more detailed outline of the proof.
	
Beyond $C^{\infty}$, the next natural question is the setting of real-analytic diffeomorphisms. Our weakly mixing AbC constructions can also be realized as real-analytic diffeomorphisms on $\T^2$ using the concept of  \emph{block-slide type of maps} introduced in \cite{Ba17}. This allows us to obtain our anti-classification results in the real-analytic category (in fact, our diffeomorphisms are holomorphic on a band around $\T^2$ of prescribed width $\rho>0$ in imaginary directions and we refer to Subsection~\ref{subsec:DiffeomorphismSpaces} for the definition of the space $\text{Diff}_{\rho}^{\,\omega}(\mathbb{T}^{2},\lambda)$).
\begin{maintheorem}
	\label{thm:analytic}Let $\rho>0$ and $\lambda$ be the Lebesgue
	measure on $\mathbb{T}^{2}$. Furthermore, let $\mathcal{R}$ be any equivalence relation on $\text{Diff}_{\rho}^{\,\omega}(\mathbb{T}^{2},\lambda)$ satisfying $$\mathcal{R}_{\text{iso}}\cap \left(\text{Diff}_{\rho}^{\,\omega}(\mathbb{T}^{2},\lambda) \times \text{Diff}_{\rho}^{\,\omega}(\mathbb{T}^{2},\lambda)\right) \subseteq \mathcal{R} \subseteq \mathcal{R}_{\text{Kak}}\cap \left(\text{Diff}_{\rho}^{\,\omega}(\mathbb{T}^{2},\lambda) \times \text{Diff}_{\rho}^{\,\omega}(\mathbb{T}^{2},\lambda)\right).$$ Then the collection
	\[
	\left\{ (S,T):S\text{ and }T\text{ are weakly mixing diffeomorphisms and $\mathcal{R}$-equivalent}\right\} 
	\]
	in $\text{Diff}_{\rho}^{\,\omega}(\mathbb{T}^{2},\lambda)\times\text{Diff}_{\rho}^{\,\omega}(\mathbb{T}^{2},\lambda)$
	is a complete analytic set and, hence, not a Borel set with respect
	to the $\text{Diff}_{\rho}^{\,\omega}$ topology.
\end{maintheorem}

It is still an open problem whether the anti-classification results for $K$-systems from \cite{GK4} also hold in the real-analytic category. We emphasize that all real-analytic constructions in this article are done on the torus. It is a challenging problem to extend them to other real-analytic manifolds.
	
	\subsection{Strategy of proof}
Following the strategy from \cite{FRW} we want to reduce the complete analytic set of ill-founded trees to the collection of $\mathcal{R}$-equivalent weakly mixing transformations. To explain this further, we introduce some terminology and basic facts from Descriptive Set Theory (see \cite{Kechris} or \cite[section 2]{FRW}). The main tool is the idea of a
reduction.	
	\begin{defn}\label{defn:reduction}
		Let $X$ and $Y$ be Polish spaces and $A\subseteq X$, $B\subseteq Y$.
		A function $f:X\to Y$ \emph{reduces} $A$ to $B$ if and only if
		for all $x\in X$: $x\in A$ if and only if $f(x)\in B$. 
		Such a function $f$ is called a Borel (respectively, continuous) reduction
		if $f$ is a Borel (respectively, continuous) function.
	\end{defn}
	
	We note that if $f$ is a Borel reduction of $A$ to $B$ and $A$ is not Borel,
	then $B$ is not Borel.
	
	\begin{defn}
		If $X$ is a Polish space and $A\subseteq X$, then $A$ is \emph{analytic}
		if and only if it is the continuous image of a Borel subset of a Polish
		space. An analytic subset $A$ of a Polish space $X$ is called \emph{complete analytic} if every analytic set can be continuously reduced to $A$.
\end{defn}
	
	Since there are analytic sets that are not Borel, a complete analytic set is not Borel. The collection of ill-founded trees is an example of a complete analytic set.	
	Here, a \emph{tree} is a set $\mathcal{T}\subseteq\mathbb{N}^{<\mathbb{N}}$
	such that if $\tau=\left(\tau_{1},\dots,\tau_{n}\right)\in\mathcal{T}$
	and $\sigma=\left(\tau_{1},\dots,\tau_{m}\right)$ with $m\leq n$
	is an initial segment of $\tau$, then $\sigma\in\mathcal{T}$.	
		An \emph{infinite branch} through $\mathcal{T}$ is a function $f:\mathbb{N}\to\mathbb{N}$ such that for all $n\in\mathbb{N}$ we have $\left(f(0),\dots,f(n-1)\right)\in\mathcal{T}$. If a tree has an infinite branch, it is called \emph{ill-founded}.
		If it does not have an infinite branch, it is called \emph{well-founded}.
	
	In Section~\ref{subsec:trees} we describe a topology on the collection of trees. We will see that the space $\mathcal{T}\kern-.5mm rees$ of trees containing arbitrarily long finite sequences is a Polish space.
	As mentioned before, we have the classical fact that the collection of ill-founded trees is a complete
	analytic subset of $\mathcal{T}\kern-.5mm rees$ \cite[section~27]{Kechris}.
	
	To prove Theorem \ref{thm:smooth} we actually show the following stronger
	result.
	\begin{thm}
		\label{thm:criterion}Let $M$ be the disk, annulus or torus with Lebesgue
		measure $\lambda$. There is a continuous one-to-one map 
		\[
		\Phi:\mathcal{T}\kern-.5mm rees\to \text{Diff}^{\,\infty}_{\,\lambda}(M)
		\]
		such that for every $\mathcal{T}\in\mathcal{T}\kern-.5mm rees$ the diffeomorphism $T=\Phi(\mathcal{T})$ is weakly mixing and we have:
		\begin{enumerate}
			\item If $\mathcal{T}$ has an infinite branch, then $T$ and $T^{-1}$
			are isomorphic.
			\item If $T$ and $T^{-1}$ are Kakutani equivalent, then $\mathcal{T}$
			has an infinite branch.
		\end{enumerate}
	\end{thm}
We now show how Theorem \ref{thm:smooth} follows from Theorem \ref{thm:criterion}.
\begin{proof}[Proof of Theorem \ref{thm:smooth}]
	Since the equivalence relation $\mathcal{R}$ is finer than or equal to isomorphism and coarser than or equal to Kakutani equivalence, Theorem~\ref{thm:criterion} yields the existence of a continuous one-to-one map $\Phi:\mathcal{T}\kern-.5mm rees\to\text{Diff}^{\,\infty}_{\,\lambda}(M)$
	such that for $\mathcal{T}\in\mathcal{T}\kern-.5mm rees$ and weakly mixing $T=\Phi(\mathcal{T})$: $\mathcal{T}$ has an infinite branch if and only if $T$ and $T^{-1}$
	are $\mathcal{R}$-equivalent. Using the terminology from above, this says that $\Phi$ is a continuous reduction from the complete analytic set of ill-founded trees to the set 
	\[
	\left\{ T\in\text{Diff}^{\,\infty}_{\,\lambda}(M):T\text{ is weakly mixing and $\mathcal{R}$-equivalent to }T^{-1}\right\}. 
	\]
	We reduce this set to
	\[
	\left\{ (S,T)\in\text{Diff}^{\,\infty}_{\,\lambda}(M)\times\text{Diff}^{\,\infty}_{\,\lambda}(M):S\text{ and }T\text{ are weakly mixing and $\mathcal{R}$-equivalent}\right\} 
	\]
	by applying the continuous map $i(T)=\left(T,T^{-1}\right)$ from
	$\text{Diff}^{\,\infty}_{\,\lambda}(M)$ to $\text{Diff}^{\,\infty}_{\,\lambda}(M)\times\text{Diff}^{\,\infty}_{\,\lambda}(M)$.
\end{proof}

The proof of Theorem~\ref{thm:analytic} in Section~\ref{subsec:Non-Equiv} will follow the same strategy as in the smooth setting. 
	
	\subsection{Outline of the paper}\label{subsec:outline}

In their proof of the anti-classification result for ergodic measure-preserving transformations in \cite{FRW}, Foreman, Rudolph, and Weiss construct a continuous function $F$ from the space $\trees$ to the invertible measure-preserving transformations assigning to each tree $\mathcal{T}$ an ergodic transformation $T=F(\mathcal{T})$ of zero-entropy such that $T \cong T^{-1}$ just in case $\mathcal{T}$ has an infinite branch. The assembly of such a transformation $T$ can be viewed as cutting\&stacking construction or as the construction of a symbolic system. When taking the second viewpoint, then the symbolic system is built with a strongly uniform
and uniquely readable construction sequence $\left(\mathtt{W}_{n}\left(\mathcal{T}\right)\right)_{n\in\mathbb{N}}$ which implies its ergodicity (see Section~\ref{subsec:SymSpace} for terminology). Related to the structure of the tree, equivalence relations $\mathcal{Q}_{s}^{n}(\mathcal{T})$ on the collections $\mathtt{W}_{n}(\mathcal{T})$ of $n$-words and group actions on the equivalence classes in $\mathtt{W}_{n}(\mathcal{T})/\mathcal{Q}_{s}^{n}(\mathcal{T})$
are specified in \cite{FRW}. Then $(n+1)$-words are built by substituting finer equivalence classes of $n$-words into coarser classes using a probabilistic substitution lemma. The constructed transformations have an odometer as a non-trivial Kronecker factor and, hence, are \textit{Odometer-based Systems} in up-to-date terminology. Then a complete analysis of joinings over the odometer base is used in \cite{FRW} to find possible isomorphisms between transformations and their inverses. Since weakly mixing transformations have trivial Kronecker factors, we cannot use this method of joinings in our proof. Instead, we use a finite coding argument as in the proof of anti-classification results for Kakutani equivalence in \cite{GK3}. (As announced in \cite{FRW06}, the coding approach and $\overline{d}$-estimates could also be used to exclude an isomorphism between $T$ and $T^{-1}$ in case that the tree has no infinite branch.) The words in \cite{GK3} are built using a deterministic procedure by substituting so-called \emph{Feldman patterns} of finer classes into Feldman patterns of coarser classes. We refer to Section~\ref{subsec:Feldman} for a description of Feldman patterns and their properties. In particular, different Feldman patterns cannot be matched well in $\overline{f}$ even after a finite coding. We review further important properties of the construction from \cite{GK3} in Section~\ref{subsec:spec}. The resulting transformations are odometer-based systems and are not weakly mixing.

One also meets the \emph{odometer obstacle} when looking for smooth versions of the aforementioned anti-classification results, because it is a persistent open problem to find a smooth realization of transformations with an odometer-factor (see \cite[Problem 7.10]{FK}). Foreman and Weiss circumvent that obstacle by showing that the collection of odometer-based systems has the same global structure with respect to joinings as another collection of transformations, the so-called \emph{circular systems} that are extensions of particular circle rotations. For this purpose, they show in \cite{FW2} that there is a functor $\mathcal{F}$ between these classes that takes specific types of isomorphisms between odometer-based systems to isomorphisms between circular systems.
The definition of these circular systems is inspired by a symbolic representation for circle rotations by certain Liouville rotation numbers found in \cite{FW1}. Then Foreman and Weiss use the AbC method to show that these circular systems can be realized as area-preserving ergodic $C^{\infty}$-diffeomorphisms on torus or disk or  annulus (under some assumptions on the circular coefficients). 
In the AbC method one constructs diffeomorphisms as limits of conjugates $T_n = H_n \circ R_{\alpha_{n}} \circ H^{-1}_n$ with $\alpha_{n+1} = \frac{p_{n+1}}{q_{n+1}}=\alpha_n + \frac{1}{k_n \cdot l_n \cdot q^2_n} \in \mathbb{Q}$ and $H_n = H_{n-1} \circ h_n$, where the $h_n$'s are measure-preserving diffeomorphisms satisfying $R_{\frac{1}{q_n}} \circ h_{n+1} = h_{n+1} \circ R_{\frac{1}{q_n}}$. In each step the conjugation map $h_{n+1}$ and the parameter $k_n$ are chosen such that the diffeomorphism $T_{n+1}$ imitates the desired property with a certain precision. In a final step of the construction, the parameter $l_n$ is chosen large enough to guarantee closeness of $T_{n+1}$ to $T_{n}$ in the $C^{\infty}$-topology, and this way the convergence of the sequence $\left(T_n\right)_{n \in \mathbb{N}}$ to a limit diffeomorphism is provided. The resulting realization map $R$ of circular systems allows us to extend the aforementioned anti-classification results from \cite{FRW} and \cite{GK3} for measure-preserving transformations to the setting of smooth area-preserving diffeomorphisms via the reduction $R\circ \mathcal{F} \circ F$.

The AbC constructions in the series of papers \cite{FW1,FW2,FW3} by Foreman--Weiss and in \cite{GK3} were \emph{untwisted}, that is, the conjugation map $h_{n+1}$ maps the fundamental domain $[0,1/q_n]\times [0,1]$ into itself. By \cite[Proposition~8.1]{Ka03} an untwisted AbC transformation has a factor isomorphic to a circle rotation. Hence, an untwisted AbC transformation cannot be weakly mixing. Accordingly, we have to design a specific twisted version of the AbC method in Sections~\ref{subsec:abstract}--\ref{subsec:constr} to produce weakly mixing transformations. 

On the one hand, our AbC constructions are complicated enough to produce weak mixing behaviour. On the other hand, they are sufficiently manageable to still allow a simple symbolic representation. This symbolic representation is described in Section~\ref{subsec:symbolic}. It motivates the introduction of so-called \emph{twisted symbolic systems}. We present their definition and basic properties in Section~\ref{sec:twist}. The concept of twisted systems is our counterpart of the circular systems in \cite{FW1,FW2,FW3}. It is an interesting task to explore properties of twisted systems in parallel to the analysis of circular systems which culminated in a global structure theory in \cite{FW2}.

In Section \ref{subsec:smooth}, we obtain realizations of our weakly mixing abstract AbC transformations as $C^{\infty}$ diffeomorphisms on $\mathbb{T}^2$, $\mathbb{D}$, and $\mathbb{A}=\mathbb{S}^1 \times [0,1]$: If the parameter sequence $(l_n)_{n \in \N}$ growths sufficiently fast, then there is an area-preserving $C^{\infty}$ diffeomorphism measure-theoretically isomorphic to a given twisted symbolic system. 

In order to prove our Theorem \ref{thm:criterion} we undertake some small modifications to the inductive construction of words in \cite{GK3} such that the resulting odometer-based construction sequence allows the creation of a weakly mixing system via our twisting operator. We present the general substitution step in Section~\ref{subsec:Substitution} with emphasis on the small modifications. In Section~\ref{subsec:Construction} we execute the inductive construction process of the odometer-based construction sequence and the associated twisted construction sequence. In this way we build the continuous reduction $\Phi:\mathcal{T}\kern-.5mm rees\to \text{Diff}^{\,\infty}_{\,\lambda}(M)$. Finally, we verify in Section~\ref{sec:Proof} that $\Phi$ satisfies the properties stated in Theorem~\ref{thm:criterion}.

	\section{\label{sec:Preliminaries}Preliminaries}
	We review some terminology from ergodic theory and symbolic dynamics. In Sections~\ref{subsec:trees} and \ref{subsec:DiffeomorphismSpaces} we introduce spaces of trees and real-analytic diffeomorphisms of $\T^2$, respectively. 
	
	\subsection{Some basics in Ergodic Theory}\label{subsec:erg}
	We fix a \emph{standard measure space} $(\Omega,\mathcal{M},\mu)$, that is, a separable non-atomic probability space.  
	We let $\mathcal{X}$ denote the group of MPT's on $(\Omega,\mathcal{M},\mu)$
	endowed with the weak topology, where two MPT's are identified if
	they are equal on sets of full measure. Recall that $T_{n}\to T$ in the weak topology if and only
	if $\mu(T_{n}(A)\triangle T(A))\to0$ for every $A\in\mathcal{M}$.
	In order to give a concrete description of the weak topology we also recall that a (finite) \emph{partition} $\mathcal{P}$ of $(\Omega,\mathcal{M},\mu)$ is a collection $\mathcal{P}=\left\{ c_{\sigma}\right\} _{\sigma\in\Sigma}$
	of subsets $c_{\sigma}\in\mathcal{M}$ with $\mu(c_{\sigma}\cap c_{\sigma'})=0$
	for all $\sigma\neq\sigma'$ and $\mu\left(\bigcup_{\sigma\in\Sigma}c_{\sigma}\right)=1$,
	where $\Sigma$ is a finite set of indices. Each $c_{\sigma}$
	is called an \emph{atom} of the partition $\mathcal{P}$. 
	For two partitions $\mathcal{P}$ and $\mathcal{Q}$ we define the join of $\mathcal{P}$ and $\mathcal{Q}$ to be the partition $\mathcal{P}\vee \mathcal{Q}=\Meng{ c\cap d}{c\in \mathcal{P},d\in \mathcal{P}}$,
	and for a sequence of partitions $\{\mathcal{P}_{n}\}_{n=1}^{\infty}$ we let
	$\vee_{n=1}^{\infty}\mathcal{P}_{n}$ be the smallest $\sigma$-algebra containing $\cup_{n=1}^{\infty}\mathcal{P}_{n}$.
	We say that a sequence of partitions $\{\mathcal{P}_{n}\}_{n=1}^{\infty}$ is a \emph{generating sequence} if 
	$\vee_{n=1}^{\infty}\mathcal{P}_{n} = \mathcal{M}$. A sequence of partitions $\{\mathcal{P}_n\}_{n\in \N}$ is called \emph{decreasing} if $\mathcal{P}_{n+1}$ is a refinement of $\mathcal{P}_n$ for any $n\in \N$.
	We also have a standard
	notion of a \emph{distance between two ordered partitions}: If $\mathcal{P}=\{c_{i}\}_{i=1}^{N}$
	and $\mathcal{Q}=\{d_{i}\}_{i=1}^{N}$ are two ordered partitions with the same number of atoms, then
	we define $D_{\mu}(\mathcal{P},\mathcal{Q})=\sum_{i=1}^{N}\mu(c_{i}\triangle d_{i})$,
	where $\triangle$ denotes the symmetric difference.  
	Now we are ready to give a concrete description of the weak topology as follows: For $T \in \MPT$, a finite partition $\mathcal{P}$, and $\varepsilon >0$ we define 
	\[
	N(T,\mathcal{P},\varepsilon)=\Meng{S\in \MPT}{\sum_{A\in \mathcal{P}}\mu(TA\triangle SA)< \varepsilon}.
	\]
	If $\{\mathcal{P}_n\}_{n\in \mathbb{N}}$ is generating, then
	$
	\Meng{N(T,\mathcal{P}_n,\varepsilon)}{T\in \MPT, \, n \in \mathbb{N},\, \varepsilon>0}
$
	generates the weak topology on $\MPT$.
	
	The following criterion proves useful to check convergence in the weak topology.
	\begin{fact}[\cite{FW1}, Lemma 5]\label{fact:ConvWT}
		Let $\{T_n\}_{n\in \N}$ be a sequence of MPTs and $\{\mathcal{P}_n\}_{n\in \mathbb{N}}$ be a generating sequence of partitions. Then the following statements
		are equivalent:
		\begin{enumerate}
			\item The sequence $\{T_n\}_{n\in \N}$ converges to a MPT in the weak topology.
			\item For all $\epsilon > 0$, $p\in \N$ there is $N\in \N$ such that for all $n,m > N$ we have
			\[
			\sum_{A \in \mathcal{P}_p, i=\pm1}\mu\left(T^i_m(A)\triangle T^i_n(A)\right)<\epsilon.
			\]
		\end{enumerate}
	\end{fact}
	
	We will also use the following fact to construct an isomorphism between limits of sequences of MPT's.	
	\begin{fact}[\cite{FW1}, Lemma 30] \label{fact:isomorphism}
		Fix a sequence $\{\varepsilon_n\}_{n\in \N}$ such that $\sum_{n=1}^\infty\varepsilon_n<\infty$. Let $(\Omega,\mathcal{M},\mu)$ and $(\Omega',\mathcal{M}',\mu')$ be two standard measure spaces and $\{T_n\}_{n\in \N}$ and $\{T_n'\}_{n\in \N}$ be MPT's of $\Omega$ and $\Omega'$ converging weakly to $T$ and $T'$, respectively. Suppose $\{\mathcal{P}_n\}_{n\in \N}$ is a decreasing sequence of partitions and $\{K_n\}_{n\in \N}$ is a sequence of measure-preserving transformations such that 
		\begin{enumerate}
			\item $K_n:\Omega\to\Omega'$ is an isomorphism between $T_n$ and $T_n'$,
			\item $\{\mathcal{P}_n\}_{n\in \N}$ and $\{K_n(\mathcal{P}_n)\}_{n\in \N}$ are generating sequences of partitions for $\Omega$ and $\Omega'$,
			\item $D_{\mu}(K_{n+1}(\mathcal{P}_n),K_n(\mathcal{P}_n))<\varepsilon_n$. 
		\end{enumerate}  
		Then the sequence $\{K_n\}_{n\in \N}$ converges in the weak topology to a measure-theoretic isomorphism between $T$ and $T'$.
	\end{fact}

Finally, we also recall some relevant facts on the concept of periodic processes. We refer to \cite{Ka03} for a detailed exposition. 

\begin{defn}
	A \emph{periodic process} is a pair $(\tau,\mathcal{P})$ where $\mathcal{P}$ is a partition of $(\Omega,\mathcal{M},\mu)$ and $\tau$ is a permutation of $\mathcal{P}$ such that the lengths of all cycles are equal\footnote{This is not a necessary requirement for the definition, but it suffices for our applications.} and the atoms in each cycle have the same measure.
\end{defn}
We refer to these cycles as \emph{towers} and their length is called height of the tower. We also choose an atom from each tower arbitrarily and call it the base of the tower. In particular, if $t_1, \ldots,t_s$ are the towers (of height $h$) of this periodic process with $B_1,\ldots,B_s$ as their respective bases, then any tower $t_i$ can be explicitly written as $B_i,\tau(B_i),\ldots,\tau^{h-1}(B_i)$. We refer to $\tau^k(B_i)$ as the $k$-th level of the tower $t_i$. Furthermore, we call $\tau^{h-1}(B_i)$ the top level.

	\subsection{Symbolic systems} 	\label{subsec:SymSpace}
	An \emph{alphabet} is a countable or finite collection of symbols. 
	In the following, let $\Sigma$ be a finite alphabet endowed with
	the discrete topology. Then $\Sigma^{\mathbb{Z}}$ with the product
	topology is a separable, totally disconnected and compact space. A
	usual base of the product topology is given by the collection of cylinder
	sets of the form $\left\langle u\right\rangle _{k}=\left\{ f\in\Sigma^{\mathbb{Z}}:f\upharpoonright[k,k+n)=u\right\} $
	for some $k\in\mathbb{Z}$ and finite sequence $u=\sigma_{0}\dots\sigma_{n-1}\in\Sigma^{n}$.
	For $k=0$ we abbreviate this by $\left\langle u\right\rangle $. 
	
	The shift map $sh:\Sigma^{\mathbb{Z}}\to\Sigma^{\mathbb{Z}}$, defined by
	$
	sh((x_n)_{n=-\infty}^{\infty})=(x_{n+1})_{n=-\infty}^{\infty}
	$, 
	is a homeomorphism. If $\mu$ is a shift-invariant Borel measure,
	then the measure-preserving dynamical system $\left(\Sigma^{\mathbb{Z}},\mathcal{B},\mu,sh\right)$
	is called a \emph{symbolic system}. The closed support of $\mu$ is
	a shift-invariant subset of $\Sigma^{\mathbb{Z}}$ called a \emph{shift space} or \emph{subshift}. 	
	The subshifts that we use are described by specifying a collection of words that
	constitute a clopen basis for the support of an invariant measure. A
	word $w$ over $\Sigma$ is a finite sequence of elements of $\Sigma$,
	and we denote its length by $|w|$.
	
	\begin{defn}
		\label{def:ConstrSeq} A sequence of collections of words $\left(W_{n}\right)_{n\in\mathbb{N}}$ satisfying the following
		properties is called a \emph{construction sequence}:
		\begin{enumerate}
			\item for every $n\in\mathbb{N}$ all words in $W_{n}$ have the same length
			$h_{n}$,
			\item each $w\in W_{n}$ occurs at least once as a subword of each $w^{\prime}\in W_{n+1}$,
			\item there is a summable sequence $\left(\varepsilon_{n}\right)_{n\in\mathbb{N}}$
			of positive numbers such that for every $n\in\mathbb{N}$, every word
			$w\in W_{n+1}$ can be uniquely parsed into segments $u_{0}w_{1}u_{1}w_{1}\dots w_{l}u_{l+1}$
			such that each $w_{i}\in W_{n}$, each $u_{i}$ (called spacer or
			boundary) is a word over $\Sigma$ of finite length and for this parsing
			\[
			\frac{\sum_{i=0}^{l+1}|u_{i}|}{h_{n+1}}<\varepsilon_{n+1}.
			\]
		\end{enumerate}
	\end{defn}
	
	We will often call words in $W_{n}$\emph{ $n$-words} or \emph{$n$-blocks},
	while a general concatenation of symbols from $\Sigma$ is called
	a \emph{string}. We also associate a subshift with a construction
	sequence: Let $\mathbb{K}$ be the collection of $x\in\Sigma^{\mathbb{Z}}$
	such that every finite contiguous substring of $x$ occurs inside
	some $w\in W_{n}$. Then $\mathbb{K}$ is a closed shift-invariant
	subset of $\Sigma^{\mathbb{Z}}$ that is compact since $\Sigma$ is finite. 
	
	In order to be able to unambiguously parse elements of $\mathbb{K}$
	we will use construction sequences consisting of uniquely readable
	words.
	\begin{defn}
		Let $\Sigma$ be an alphabet and $W$ be a collection of finite words
		over $\Sigma$. 
		Then $W$ is \emph{uniquely readable} if and only if whenever $u,v,w\in W$
		and $uv=pws$ with $p$ and $s$ strings of symbols from $\Sigma$,
		then either $p$ or $s$ is the empty word.
	\end{defn}
	
	Moreover, our construction sequence $\left(W_{n}\right)_{n\in\mathbb{N}}$ will be \emph{strongly uniform}, i.e., for each $n\in\mathbb{N}$ there is a constant $c>0$ such that for all words $w^{\prime}\in W_{n+1}$ and $w\in W_{n}$ we have $r(w,w')=c$, where $r(w,w')$ is the number of occurrences of $w$ in $w'$. 

\begin{rem*}
	A particular type of subshifts are the ones that have odometer
	systems as their timing mechanism to parse typical elements: Let $\left(k_{n}\right)_{n\in\mathbb{N}}$ be a sequence of natural
	numbers $k_{n}\geq2$ and $\left(W_{n}\right)_{n\in\mathbb{N}}$ be a uniquely readable
	construction sequence with $W_{0}=\Sigma$ and $W_{n+1}\subseteq\left(W_{n}\right)^{k_{n}}$
	for every $n\in\mathbb{N}$. The associated subshift is
	called an \emph{odometer-based system}.
\end{rem*}
	
	We introduce the following natural
	set $S$ which will be of measure one for measures that we consider.
	\begin{defn}\label{defn:SetS}
		Suppose that $\left(W_{n}\right)_{n\in\mathbb{N}}$ is a construction
		sequence for a subshift $\mathbb{K}$ with each $W_{n}$ uniquely
		readable. Let $S$ be the collection of $x\in\mathbb{K}$ such that
		there are sequences of natural numbers $\left(a_{n}\right)_{n\in\mathbb{N}}$,
		$\left(b_{n}\right)_{n\in\mathbb{N}}$ going to infinity such that
		for all $m\in\mathbb{N}$ there is $n\in\mathbb{N}$ such that $x\upharpoonright[-a_{m},b_{m})\in W_{n}$.
	\end{defn}
	
	We note that $S$ is a dense shift-invariant $\mathcal{G}_{\delta}$
	subset of $\mathbb{K}$ and we recall the following properties from
	\cite[Lemma 11]{FW1} and \cite[Lemma 12]{FW2}.
	\begin{fact}
		\label{fact:MeasureConstrSeq}Fix a construction sequence $\left(W_{n}\right)_{n\in\mathbb{N}}$
		for a subshift $\mathbb{K}$ over a finite alphabet $\Sigma$.
		Then:
		\begin{enumerate}
			\item $\mathbb{K}$ is the smallest shift-invariant closed subset of $\Sigma^{\mathbb{Z}}$
			such that for all $n\in\mathbb{N}$ and $w\in W_{n}$, $\mathbb{K}$
			has non-empty intersection with the basic open interval $\left\langle w\right\rangle \subset\Sigma^{\mathbb{Z}}$.
			\item Suppose that $\left(W_{n}\right)_{n\in\mathbb{N}}$ is a uniform construction
			sequence. Then there is a unique non-atomic shift-invariant measure
			$\nu$ on $\mathbb{K}$ concentrating on $S$ and $\nu$ is ergodic.
			\item If $\nu$ is a shift-invariant measure on $\mathbb{K}$ concentrating
			on $S$, then for $\nu$-almost every $s\in S$ there is $N\in\mathbb{N}$
			such that for all $n>N$ there are $a_{n}\leq0<b_{n}$ such that $s\upharpoonright[a_{n},b_{n})\in W_{n}$.
		\end{enumerate}
	\end{fact}
	
	Since our subshifts will be built from a uniquely readable uniform construction sequence,
	they will automatically be ergodic and we will identify $\mathbb{K}$ with the symbolic system $\left(\Sigma^{\mathbb{Z}},\mathcal{B},\nu,sh\right)$. To each symbolic system we will
	also consider its inverse $\mathbb{K}^{-1}$ which stands for $\left(\mathbb{K},sh^{-1}\right)$.
	Since it will often be convenient to have the shifts going in the
	same direction, we also introduce another convention.
	\begin{defn}
		If $w$ is a finite or infinite string, we write $rev(w)$ for the
		reverse string of $w$. In particular, if $x$ is in $\mathbb{K}$
		we define $rev(x)$ by setting $rev(x)(k)=x(-k)$. Then for $A\subseteq\mathbb{K}$
		we define $rev(A)=\left\{ rev(x):x\in A\right\} .$ 		
		If we explicitly view a finite word $w$ positioned at a location
		interval $[a,b)$, then we take $rev(w)$ to be positioned at the same interval
		$[a,b)$ and we set $rev(w)(k)=w(a+b-(k+1))$. For a collection $W$
		of words $rev(W)$ is the collection of reverses of words in $W$.
	\end{defn}
	
	Then we introduce the symbolic system $\left(rev(\mathbb{K}),sh\right)$
	as the one built from the construction sequence $\left(rev(W_{n})\right)_{n\in\mathbb{N}}$.
	Clearly, the map sending $x$ to $rev(x)$ is a canonical isomorphism
	between $\left(\mathbb{K},sh^{-1}\right)$ and $\left(rev(\mathbb{K}),sh\right)$.
	We often abbreviate the symbolic system $\left(rev(\mathbb{K}),sh\right)$
	as $rev(\mathbb{K})$. 
	
	\begin{defn} \label{dfn:freq}
		Let $\Sigma$ be an alphabet. For a word $w\in\Sigma^{k}$ and $x\in\Sigma$
		we write $r(x,w)$ for the number of times that $x$ occurs in $w$
		and $\freq(x,w)=\frac{r(x,w)}{k}$ for the frequency of occurrences
		of $x$ in $w$. Similarly, for $(w,w')\in\Sigma^{k}\times\Sigma^{k}$
		and $(x,y)\in\Sigma\times\Sigma$ we write $r(x,y,w,w')$ for the
		number of $i<k$ such that $x$ is the $i$-th member of $w$ and
		$y$ is the $i$-th member of $w'$. We also introduce $\freq(x,y,w,w')=\frac{r(x,y,w,w')}{k}$.
	\end{defn}
	
	\subsection{The $\overline{f}$ metric} \label{subsec:fbar}
	In the study of Kakutani equivalence Feldman \cite{Fe} introduced a notion of distance, now called $\overline{f}$, as a substitute for the Hamming distance $\overline{d}$ in Ornstein's isomorphism theory.
	\begin{defn}
		\label{def:fbar}A \emph{match} between two strings of symbols $a_{1}a_{2}\dots a_{n}$
		and $b_{1}b_{2}\dots b_{m}$ from a given alphabet $\Sigma,$ is a
		collection $\mathcal{M}$ of pairs of indices $(i_{s},j_{s})$, $s=1,\dots,r$
		such that $1\le i_{1}<i_{2}<\cdots<i_{r}\le n$, $1\le j_{1}<j_{2}<\cdots<j_{r}\le m$
		and $a_{i_{s}}=b_{j_{s}}$ for $s=1,2,\dots,r.$ Then 
		\begin{equation}
			\begin{array}{ll}
				\overline{f}(a_{1}a_{2}\dots a_{n},b_{1}b_{2}\dots b_{m})=\hfill\\
				{\displaystyle 1-\frac{2\sup\{|\mathcal{M}|:\mathcal{M}\text{\ is\ a\ match\ between\ }a_{1}a_{2}\cdots a_{n}\text{\ and\ }b_{1}b_{2}\cdots b_{m}\}}{n+m}.}
			\end{array}\label{eq:cl}
		\end{equation}
		We will refer to $\overline{f}(a_{1}a_{2}\cdots a_{n},b_{1}b_{2}\cdots b_{m})$
		as the ``$\overline{f}$-distance'' between $a_{1}a_{2}\cdots a_{n}$
		and $b_{1}b_{2}\cdots b_{m},$ even though $\overline{f}$ does not
		satisfy the triangle inequality unless the strings are all of the
		same length. A match $\mathcal{M}$ is called a\emph{ best possible
			match} if it realizes the supremum in the definition of $\overline{f}$.
	\end{defn}

\begin{rem*}
	Alternatively, one can view a match as an injective order-preserving
	function $\pi:\mathcal{D}(\pi)\subseteq\left\{ 1,\dots,n\right\} \to\mathcal{R}(\pi)\subseteq\left\{ 1,\dots,m\right\} $
	with $a_{i}=b_{\pi(i)}$ for every $i\in\mathcal{D}(\pi)$. Then $\overline{f}\left(a_{1}\dots a_{n},b_{1}\dots b_{m}\right)=1-\max\left\{ \frac{2|\mathcal{D}(\pi)|}{n+m}:\pi\text{ is a match}\right\} $.
\end{rem*}

We also state the following fact that can be proved easily
by considering the \emph{fit} $1-\bar{f}(a,b)$ between two strings
$a$ and $b$.
\begin{fact}[\cite{GK3}, Fact 10]
	\label{fact:omit_symbols}Suppose $a$ and $b$ are strings of symbols
	of length $n$ and $m,$ respectively, from an alphabet $\Sigma$.
	If $\tilde{a}$ and $\tilde{b}$ are strings of symbols obtained by
	deleting at most $\lfloor\gamma(n+m)\rfloor$ terms from $a$ and
	$b$ altogether, where $0<\gamma<1$, then 
	\begin{equation}
		\overline{f}(a,b)\ge\overline{f}(\tilde{a},\tilde{b})-2\gamma.\label{eq:omit_symbols}
	\end{equation}
\end{fact}
	
	\subsection{Feldman patterns} \label{subsec:Feldman}
	To construct the symbolic systems in \cite{GK3}, the $n$-words in the construction sequence are built using specific patterns of blocks. These patterns are called \emph{Feldman patterns} since they originate from Feldman's first example of an ergodic zero-entropy automorphism that is not loosely Bernoulli \cite{Fe}. In particular, different Feldman patterns cannot be matched well in $\overline{f}$ even after a finite coding.	
	Let $T,N,M\in\mathbb{Z}^{+}$. A $(T,N,M)$-Feldman pattern in building
	blocks $A_{1},\dots,A_{N}$ of equal length $L$ is one of the strings
	$B_{k}$, $k=1,\dots , M$, defined by
	\[
		B_{k}=  \left(A_{1}^{TN^{2k}}A_{2}^{TN^{2k}}\dots A_{N}^{TN^{2k}}\right)^{N^{2(M+1-k)}}.
	\]
	
	Thus $N$ denotes the number of building blocks, $M$ is the number
	of constructed patterns, and $TN^2$ gives the minimum number of consecutive occurrences
	of a building block. We also note that $B_{k}$ is built with $N^{2(M+1-k)}$ many
	so-called \emph{cycles}: Each cycle winds through all the $N$ building blocks.
	
	Moreover, we collect the following properties of $(T,N,M)$-Feldman patterns.
	
	\begin{lem}\label{lem:FP}
		Let $B_{k}$, $1\leq k\leq M$,
		be the $(T,N,M)$-Feldman patterns in the building blocks $A_{1},\dots,A_{N}$
		of equal length $L$.
		\begin{enumerate}
			\item Each building block $A_{i}$, $1\leq i\leq N$, occurs $TN^{2M+2}$ times in each pattern.
			\item Every block $B_{k}$, $1\leq k\leq M$, has total length $TN^{2M+3}L$.
			\item For all $i,j \in \{1,\dots , N\}$ and all $k,l \in \{1,\dots , M\}$, $k \neq l$, we have
			\[
			\frac{r(A_i^{TN^2},A_j^{TN^2},B_k,B_l)}{N^{2M+1}} = \frac{1}{N^2}.
			\]
		\end{enumerate}
	\end{lem}

	\subsection{The space of trees} \label{subsec:trees}
	To describe a topology on the collection of trees, let $\left\{ \sigma_{n}:n\in\mathbb{N}\right\} $
	be an enumeration of $\mathbb{N}^{<\mathbb{N}}$ with the property
	that every proper predecessor of $\sigma_{n}$ is some $\sigma_{m}$
	for $m<n$. Under this enumeration subsets $S\subseteq\mathbb{N}^{<\mathbb{N}}$
	can be identified with characteristic functions $\chi_{S}:\mathbb{N}\to\left\{ 0,1\right\} $.
	The collection of such $\chi_{S}$ can be viewed as the members of
	an infinite product space $\left\{ 0,1\right\} ^{\mathbb{N}^{<\mathbb{N}}}$
	homeomorphic to the Cantor space. Here, each function $a:\left\{ \sigma_{m}:m<n\right\} \to\left\{ 0,1\right\} $
	determines a basic open set 
	$
	\left\langle a\right\rangle =\left\{ \chi:\chi\upharpoonright\left\{ \sigma_{m}:m<n\right\} =a\right\} \subseteq\left\{ 0,1\right\} ^{\mathbb{N}^{<\mathbb{N}}}
	$
	and the collection of all such $\left\langle a\right\rangle $ forms
	a basis for the topology. In this topology the collection of trees is a closed (hence
	compact) subset of $\left\{ 0,1\right\} ^{\mathbb{N}^{<\mathbb{N}}}$. Moreover, the collection $\mathcal{T}\kern-.5mm rees$ of trees containing arbitrarily
	long finite sequences is a dense $\mathcal{G}_{\delta}$ subset. Hence, $\mathcal{T}\kern-.5mm rees$ is a Polish space. 	
	Since the topology on the space of trees was introduced via basic
	open sets giving us a finite amount of information about the trees in
	it, we can characterize continuous maps defined on $\mathcal{T}\kern-.5mm rees$
	as follows.
	\begin{fact}
		\label{fact:contTree}Let $Y$ be a topological space. Then a map
		$f:\mathcal{T}\kern-.5mm rees\to Y$ is continuous if and only if for all open
		sets $O\subseteq Y$ and all $\mathcal{T}\in\mathcal{T}\kern-.5mm rees$ with
		$f(\mathcal{T})\in O$ there is $M\in\mathbb{N}$ such that for all
		$\mathcal{T}^{\prime}\in\mathcal{T}\kern-.5mm rees$ we have:
		
		if $\mathcal{T}\cap\left\{ \sigma_{n}:n\leq M\right\} =\mathcal{T}^{\prime}\cap\left\{ \sigma_{n}:n\leq M\right\} $,
		then $f\left(\mathcal{T}^{\prime}\right)\in O$. 
	\end{fact}

	During our constructions the following maps will prove useful.
\begin{defn}
	\label{def:M-and-s}We define a continuous map $M:\mathcal{T}\kern-.5mm rees\to\mathbb{N}^{\mathbb{N}}$
	by setting $M\left(\mathcal{T}\right)(s)=n$ if and only if $n$ is
	the least number such that $\sigma_{n}\in\mathcal{T}$ and the length of $\sigma_n$ is $s$.
	Dually, we also define a map $s:\mathcal{T}\kern-.5mm rees\to\mathbb{N}^{\mathbb{N}}$
	by setting $s\left(\mathcal{T}\right)(n)$ to be the length of the
	longest sequence $\sigma_{m}\in\mathcal{T}$ with $m\leq n$. 
\end{defn}

\subsection{Real-analytic diffeomorphisms of the torus} \label{subsec:DiffeomorphismSpaces}
Following \cite[section 2.2]{BK2} we give a  description
of the spaces $\text{Diff}_{\rho}^{\,\omega}(\mathbb{T}^{2},\lambda)$. Here, $\lambda$ denotes the standard Lebesgue
measure on $\mathbb{T}^{2}\coloneqq\mathbb{R}^{2}/\mathbb{Z}^{2}$.
Any real-analytic diffeomorphism on $\mathbb{T}^{2}$ homotopic to
the identity admits a lift to a map from $\mathbb{R}^{2}$ to $\mathbb{R}^{2}$
which has the form 
\[
F(x_{1},x_{2})=(x_{1}+f_{1}(x_{1},x_{2}),x_{2}+f_{2}(x_{1},x_{2})),
\]
where $f_{i}:\mathbb{R}^{2}\to\mathbb{R}$ are $\mathbb{Z}^{2}$-periodic
real-analytic functions. Any real-analytic $\mathbb{Z}^{2}$-periodic
function on $\mathbb{R}^{2}$ can be extended as a holomorphic function
defined on some open complex neighborhood of $\mathbb{R}^{2}$ in
$\mathbb{C}^{2}$, where we identify $\mathbb{R}^{2}$ inside $\mathbb{C}^{2}$
via the natural embedding $(x_{1},x_{2})\mapsto(x_{1}+\mathrm{i}0,x_{2}+\mathrm{i}0)$.
For a fixed $\rho>0$ we define the neighborhood 
$
\Omega_{\rho}:=\{(z_{1},z_{2})\in\mathbb{C}^{2}:|\text{Im}(z_{1})|<\rho\text{ and }|\text{Im}(z_{2})|<\rho\},
$
that is, $\rho$ describes the width of a band in the imaginary directions. For a function $f$ defined on $\Omega_{\rho}$ we let 
\[
\|f\|_{\rho}:=\sup_{(z_{1},z_{2})\in\Omega_{\rho}}|f((z_{1},z_{2}))|.
\]
We define $C_{\rho}^{\omega}(\mathbb{T}^{2})$ to be the space of
all $\mathbb{Z}^{2}$-periodic real-analytic functions $f$ on $\mathbb{R}^{2}$
that extend to a holomorphic function on $\Omega_{\rho}$ and satisfy $\|f\|_{\rho}<\infty$.
Hereby, we define $\text{Diff}_{\rho}^{\,\omega}(\mathbb{T}^{2},\lambda)$
to be the set of all Lebesgue measure-preserving real-analytic diffeomorphisms
of $\mathbb{T}^{2}$ homotopic to the identity, whose lift $F$ to
$\mathbb{R}^{2}$ satisfies $f_{i}\in C_{\rho}^{\omega}(\mathbb{T}^{2})$,
and we also require that the lift $\tilde{F}(x)=(x_{1}+\tilde{f}_{1}(x),x_{2}+\tilde{f}_{2}(x))$
of its inverse to $\mathbb{R}^{2}$ satisfies $\tilde{f}_{i}\in C_{\rho}^{\omega}(\mathbb{T}^{2})$.
Then the metric in $\text{Diff}_{\rho}^{\,\omega}(\mathbb{T}^{2},\lambda)$
is defined by 
\begin{align*}
	d_{\rho}(f,g)=\max\{\tilde{d}_{\rho}(f,g),\tilde{d}_{\rho}(f^{-1},g^{-1})\},\text{ where }\tilde{d}_{\rho}(f,g)=\max_{i=1,2}\{\inf_{n\in\mathbb{Z}}\|f_{i}-g_{i}+n\|_{\rho}\}.
\end{align*}
We note that if $\{f_{n}\}_{n=1}^{\infty}\subset\text{Diff}_{\rho}^{\,\omega}(\mathbb{T}^{2},\lambda)$
is a Cauchy sequence in the $d_{\rho}$ metric, then $f_{n}$ converges
to some $f\in\text{Diff}_{\rho}^{\,\omega}(\mathbb{T}^{2},\lambda)$.
Thus, this space is Polish. 

\section{Twisted symbolic systems} \label{sec:twist}
In this subsection we introduce a specific type of symbolic system that will turn out to give symbolic representations of our weakly mixing AbC constructions in Section~\ref{subsec:symbolic}. These so-called \emph{twisted systems} are our counterpart of \emph{circular systems} used in \cite[section 4]{FW1} as representations of \emph{untwisted} AbC transformations. 

\subsection{Definition of twisted systems}
To state the definition let $\left(C_{n},l_{n}\right)_{n\in\mathbb{N}}$ be a sequence of pairs of positive
integers such that $\sum_{n\in\mathbb{N}}\frac{1}{l_{n}}<\infty$. We use them to inductively define sequences $(k_n)_{n\in \N}$, $(p_n)_{n\in \N}$, $(q_n)_{n\in \N}$ of positive integers as follows: We set $p_{0}=0$ and $q_{0}=1$. Then for each $n\in \N$ we define
\begin{align*}
	k_n = 2^{n+2} \cdot q_n \cdot C_n, \ \ \  q_{n+1} =k_{n}l_{n}q_{n}^{2}, \ \ \  p_{n+1} =p_{n}k_{n}l_{n}q_{n}+1.
\end{align*}
Obviously, $p_{n+1}$ and $q_{n+1}$ are relatively prime. 

\begin{rem*}
	In \cite{FW1} such a sequence $\left(k_{n},l_{n}\right)_{n\in\mathbb{N}}$ of pairs of positive
	integers is called a \emph{circular coefficient sequence}. Accordingly, we refer to $\left(C_{n},l_{n}\right)_{n\in\mathbb{N}}$ as a \emph{twisting coefficient sequence}
\end{rem*}

Furthermore we introduce numbers $j_i$ as follows: If $n=0$ we take $j_{0}=0$, and for $n>0$ we let $j_{i}\in\{0,\dots,q_{n}-1\}$ be such that
\begin{equation}\label{eq:jiTwist}
	j_{i}\equiv\left(p_{n}\right)^{-1}i\;\mod q_{n}.
\end{equation}
where the $\left(p_{n}\right)^{-1}$ is the multiplicative inverse of $p_n$ modulo $q_n$. We also note that
\begin{equation}\label{eq:ji-relTwist}
	q_n-j_i = j_{q_n-i}.
\end{equation}
Using these numbers $j_i$ we define
\begin{equation} \label{eq:psiTwist}
	\psi_n(i) = \begin{cases}
		0, & \text{ if  $0\leq i < 2^{n+1}q_n$, $i$ even,} \\
		j_{\frac{i+1}{2}\mod q_n}, & \text{ if  $0\leq i < 2^{n+1}q_n$, $i$ odd,} \\
		j_{\frac{i}{2}+1 \mod q_n}, & \text{ if  $2^{n+1}q_n \leq i < 2^{n+2}q_n$, $i$ even,} \\
		j_1, & \text{ if  $2^{n+1}q_n \leq i < 2^{n+2}q_n$, $i$ odd.} \\
	\end{cases}
\end{equation}

\begin{rem*}
	These numbers $\psi_n(i)$ enter the symbolic representation of our weakly mixing AbC transformations via the parameters $a_n(i)$ in the construction of conjugation map $h_{n+1,1}$ in Section~\ref{subsubsec:h1}, namely, $\psi_n(i)=j_{a_n(i)}$. We refer to Remark~\ref{rem:an} for a motivation of the choice of parameters $a_n(i)$. 
\end{rem*}

Let $\Sigma$ be a non-empty
finite alphabet and $b,e$ be two additional symbols.
\begin{defn}[Twisting operator] \label{def:twist}
	Let $w_0,\dots,w_{k_n-1}$ be words over $\Sigma \cup \{b,e\}$. Using the aforementioned notation we let\footnote{We use
		$\prod$ and powers for repeated concatenation of words.}
	\begin{equation}\label{eq:TwistOperator}
		\begin{split}
			&\mathcal{C}^{\text{twist}}_n(w_0,\dots , w_{k_n-1})= \\
			&\prod^{q_n-1}_{m=0}\quad \prod^{2^{n+2}q_n-1}_{i=0} \quad \prod^{C_n-1}_{c=0} b^{q_n-\psi_n(i)-j_m \mod q_n}\,(w_{iC_n+c})^{l_n-1}\,e^{\psi_n(i)+j_m \mod q_n}
		\end{split}
	\end{equation}
	be the \emph{twisting operator} at level $n$.
\end{defn}

\begin{rem*}
	Suppose that each $w_i$ has length $q_n$. Then the length of $\mathcal{C}^{\text{twist}}_n(w_0,\dots , w_{k_n-1})$ is $q_n2^{n+2}q_nC_nl_nq_n=k_nl_nq^2_n=q_{n+1}$.
\end{rem*}

\begin{rem*}
	Our twisting operator should be compared with the \emph{circular operator} $\mathcal{C}_n$ from \cite[section 4]{FW1} defined by
	\begin{equation}\label{eq:circular}
		\mathcal{C}_n(w_0,w_1,\dots , w_{k_n-1}) = \prod^{q_n-1}_{m=0} \prod^{k_n-1}_{c=0} b^{q_n-j_m}w^{l_n-1}_{c}e^{j_m}.
	\end{equation}
\end{rem*}

In the symbolic representation of our specific weakly mixing AbC constructions the twisting operator plays the role of the \emph{circular operator} for the symbolic representation of untwisted AbC transformations in \cite{FW1}. In parallel to the development of \emph{circular systems} in \cite[section 4]{FW1} we introduce so-called \emph{twisted systems}:
Given a twisting coefficient
sequence $\left(C_{n},l_{n}\right)_{n\in\mathbb{N}}$ we build collections
of words $\mathcal{W}_{n}^{\text{twist}}$ over the alphabet $\Sigma\cup\{b,e\}$
by induction as follows:
\begin{itemize}
	\item Set $\mathcal{W}_{0}^{\text{twist}}=\Sigma$.
	\item Having built $\mathcal{W}_{n}^{\text{twist}}$, we choose a set $P_{n+1}\subseteq\left(\mathcal{W}_{n}^{\text{twist}}\right)^{k_{n}}$
	of so-called \emph{prewords} and build $\mathcal{W}_{n+1}^{\text{twist}}$ by
	taking all words of the form 
	\begin{equation*}
		\mathcal{C}^{\text{twist}}_n(w_0,\dots , w_{k_n-1}) \ \text{ with } \ w_{0}\dots w_{k_{n}-1}\in P_{n+1}.
	\end{equation*}
\end{itemize}
\begin{defn}
	\label{def:twistedConstSeq} A construction sequence $\left(\mathcal{W}_{n}^{\text{twist}}\right)_{n\in\mathbb{N}}$
	will be called \emph{twisted} if it is built in this manner using
	the $\mathcal{C}^{\text{twist}}$-operators and a twisting coefficient sequence, and
	each $P_{n+1}$ is uniquely readable in the alphabet with the words
	from $\mathcal{W}_{n}^{\text{twist}}$ as letters. (This last property is called
	the \emph{strong readability assumption}.) 
\end{defn}

\begin{rem*}
	Similar to the proof of \cite[Lemma 45]{FW2} for circular
	construction sequences, one can show that each $\mathcal{W}_{n}^{\text{twist}}$ in a twisted
	construction sequence is uniquely readable even if the prewords are
	not uniquely readable. However, the definition of a twisted construction
	sequence requires this stronger readability assumption. 
\end{rem*}
\begin{defn}
	\label{def:TwistedShift}A symbolic system $\mathbb{K}$ built from
	a circular construction sequence is called a \emph{twisted system}.
	For emphasis we will often denote it by $\mathbb{K}^{\text{twist}}$. 
\end{defn}

Based on Fact \ref{fact:MeasureConstrSeq} we obtain a characterisation of the set $S\subset \mathbb{K}^{\text{twist}}$ from Definition~\ref{defn:SetS} and a strong unique ergodicity result analogous to \cite[Lemma 20]{FW1} for circular systems.

\begin{lem}
	Let $\mathbb{K}^{\text{twist}}$ be a twisted system and let $\nu$ be a shift-invariant measure on $\mathbb{K}^{\text{twist}}$.
	Then the following are equivalent:
	\begin{enumerate}
		\item $\nu$ has no atoms.
		\item $\nu$ concentrates on the collection of $s\in \mathbb{K}^{\text{twist}}$ such that $\Meng{i}{s(i)\notin \{b,e\}}$ is unbounded
		in both $\Z^-$ and $\Z^+$.
		\item $\nu$ concentrates on $S$
	\end{enumerate}
	If $\mathbb{K}^{\text{twist}}$ is a uniform twisted system, then there is a unique invariant measure
	concentrating on $S$.
\end{lem}

There are only two ergodic invariant measures with atoms: the one concentrating
on the constant sequence $\ldots bbb\ldots $ and the one concentrating on $\ldots eee\ldots $.

\begin{defn}\label{defn:boundary}
	Suppose that $w=\mathcal{C}^{\text{twist}}_n(w_0,\dots , w_{k_n-1})$. Then $w$ consists of repetitions $w^{l_n-1}_i$ of words $w_i$ and some letters $b$ and $e$ that are not in the words $w_i$. The entries of $w$ in the words $w_i$ are called the \emph{interior} of $w$. The remainder of $w$ consists of blocks of the form $b^{q_n-\psi_n(i)-j_m \mod q_n}$ and $e^{\psi_n(i)+j_m \mod q_n}$. We call these entries of $w$ the \emph{boundary} of $w$.
\end{defn}

	The boundary of $w$ constitutes a small portion of $1/l_n$ of the word $w$. 

\begin{defn}
	If $s\in S$ or $s \in \mathcal{W}_m$ with $m\geq n$ we define $\partial_n(s) \subset \Z$ to be the collection of $i\in \Z$ such that $\sh^i(s)(0)$ is in the boundary portion of an $n$-subword of $s$. Furthermore, we introduce	$\partial_n \coloneqq \Meng{x \in \mathbb{K}^{\text{twist}}}{0 \in \partial_n(x)}$.
\end{defn}

\subsection{An explicit description of $rev(\mathbb{K}^{\text{twist}})$}
To describe an explicit construction sequence $\left\{ rev(\mathcal{W}_{n}^{\text{twist}})\right\} _{n\in\mathbb{N}}$
of $(\mathbb{K}^{\text{twist}})^{-1}\cong rev(\mathbb{K}^{\text{twist}})$ we introduce the operator
\begin{equation*}
	\begin{split}
		&\widetilde{\mathcal{C}}_{n}^{\text{twist}}\left(w_{0},w_{1},\dots,w_{k_{n}-1}\right)= \\
		&\prod^{q_n-1}_{m=0}\, \prod^{2^{n+2}q_n-1}_{i=0} \, \prod^{C_n-1}_{c=0} e^{q_n-\psi_n(i)-j_m \mod q_n}\,(w_{iC_n+c})^{l_n-1}\,b^{\psi_n(i)+j_m \mod q_n},
	\end{split}
\end{equation*} 
that is, the role of $b$ and $e$ in the twisting operator has been interchanged. Then we note the following connection between forward and reverse words.
\begin{lem} \label{lem:ReverseTwist}
	Let $w_0\ldots w_{k_n-1} \in P_{n+1}\subseteq\left(\mathcal{W}_{n}^{\text{twist}}\right)^{k_{n}}$. Then
	\begin{equation}\label{eq:ReverseTwist}
		rev\left(\mathcal{C}_{n}^{\text{twist}}\left(w_{0},w_{1},\dots,w_{k_{n}-1}\right) \right) = \widetilde{\mathcal{C}}_{n}^{\text{twist}}\left(rev(w_{k_{n}-1}),\dots , rev(w_1),rev(w_0)\right).
	\end{equation}
\end{lem}

\begin{proof}
	From the definition of the numbers $j_i$ in \eqref{eq:jiTwist} and the relation \eqref{eq:ji-relTwist} we obtain $j_{q_n-1-m}=q_n-j_{m+1}=q_n-j_{m}-j_1 \mod q_n$
	and    	
	\begin{equation*}
		\begin{split}
			&\psi_n(2^{n+2}q_n-1-i) = \\
			&\begin{cases}
				j_1 = j_1-\psi_n(i), & \text{ if  $0\leq i < 2^{n+1}q_n$, $i$ even,} \\
				j_{-\frac{i+1}{2}+1 \mod q_n} = q_n - j_{\frac{i+1}{2}}+j_1=q_n-\psi_n(i)+j_1, & \text{ if  $0\leq i < 2^{n+1}q_n$, $i$ odd,} \\
				j_{\frac{-i}{2}\mod q_n}  = q_n - j_{\frac{i}{2}+1\mod q_n}+j_1 = q_n -\psi_n(i)+j_1, & \text{ if  $2^{n+1}q_n \leq i < 2^{n+2}q_n$, $i$ even,} \\
				0=j_1-\psi_n(i), & \text{ if  $2^{n+1}q_n \leq i < 2^{n+2}q_n$, $i$ odd,} \\
			\end{cases}
		\end{split}
	\end{equation*}
	that is, 
	\begin{equation} \label{eq:psiTwistRev}
		\psi_n(2^{n+2}q_n-1-i) = j_1-\psi_n(i) \mod q_n.
	\end{equation}
	Using these identities we calculate \eqref{eq:ReverseTwist}.
\end{proof}

Hence, the collections
\[
rev(\mathcal{W}_{n+1}^{\text{twist}})=\Meng{\widetilde{\mathcal{C}}_{n}^{\text{twist}}\left(rev(w_{k_{n}-1}),\dots , rev(w_1),rev(w_0)\right)}{w_{0}w_{1}\dots w_{k_{n}-1}\in P_{n+1}} 
\]
constitute a construction sequence of $(\mathbb{K}^{\text{twist}})^{-1}$.

\subsection{Subscales for twisted words}
We end this section by introducing the following subscales for
a word $w\in\mathcal{W}_{n+1}^{\text{twist}}$ analogous to the terminology for circular words in \cite[Subsection 3.3]{FW2}. 
\begin{rem}\label{rem:Subsection}
	Let $w=\mathcal{C}^{\text{twist}}_n(w_0,\dots , w_{k_n-1})\in\mathcal{W}_{n+1}^{\text{twist}}$.
	\begin{itemize}
		\item Subscale $0$ is the scale of the individual powers of $w_{j}\in\mathcal{W}_{n}^{\text{twist}}$
		of the form $w_{j}^{l_n-1}$ and each such occurrence of a $w_{j}^{l_n-1}$
		is called a \emph{$0$-subsection}. 
		\item Subscale $1$ is the scale of each term of $\mathcal{C}^{\text{twist}}_n(w_0,\dots , w_{k_n-1})$
		that has the form $\left(b^{q_n-\psi_n(i)-j_m \mod q_n}\,(w_{iC_n+c})^{l_n-1}\,e^{\psi_n(i)+j_m \mod q_n}\right)$
		and these terms are called \emph{$1$-subsections}. 
		\item Subscale $2$ is the scale of each term of $\mathcal{C}^{\text{twist}}_n(w_0,\dots , w_{k_n-1})$
		that has the form $\prod^{2^{n+2}q_n-1}_{i=0} \prod^{C_n-1}_{c=0} b^{q_n-\psi_n(i)-j_m \mod q_n}\,(w_{iC_n+c})^{l_n-1}\,e^{\psi_n(i)+j_m \mod q_n}$
		and these terms are called \emph{$2$-subsections}. 
	\end{itemize}
\end{rem}

	\section{\label{sec:wm}Weakly mixing AbC constructions}
	We start by presenting the general scheme of the abstract Approximation by Conjugation method for the construction of measure-preserving transformations. In this framework we provide a criterion for weak mixing in Section~\ref{subsec:WMcrit}. We proceed by constructing specific twisted conjugation maps in Section~\ref{subsec:constr} so that the resulting AbC transformations satisfy our criterion for weak mixing. In Section~\ref{subsec:symbolic} we find symbolic representations for our specific constructions of weakly mixing AbC maps. In this symbolic representation we use the  \emph{twisting operator} introduced in Section~\ref{sec:twist}. Finally, we show that our specific weakly mixing AbC maps allow realization as smooth or even real-analytic diffeomorphisms. 
	
	\subsection{\label{subsec:abstract}Abstract AbC constructions}
	Our constructions can be viewed as taking place on $\mathbb{T}^2=\R^2 / \Z^2$, $\mathbb{D}$, or $\mathbb{A}=\mathbb{S}^1\times [0,1]$. We use $M$ as a proxy for these spaces equipped with Lebesgue measure $\lambda$ and circle actions $\{R_t\}_{t\in \mathbb{S}^1}$ defined by
	\[
	R_t(\theta , r) = (\theta + t, r).
	\]
	Furthermore, we introduce the following notation with $r,s\in \Z^+$:
	\begin{equation}
		\Delta^{i,j}_{r,s} \coloneqq \left[ \frac{i}{r}, \frac{i+1}{r}\right) \times \left[ \frac{j}{s}, \frac{j+1}{s} \right).
	\end{equation}
	We collect the above sets to form the following partition
	\begin{equation*}
		\xi_{r, s} \coloneqq \{\Delta_{r,s}^{i,j}: 0\leq i< r,\; 0\leq j< s\}.
	\end{equation*}

	Our transformations will be obtained as the limit of an inductive construction process of conjugates
	\begin{equation}\label{eq:AbC}
		T_n = H_n \circ R_{\alpha_n} \circ H^{-1}_n
	\end{equation}
with conjugation maps $H_n = H_{n-1}\circ h_n$ and $\alpha_n = \frac{p_n}{q_n} \in \mathbb{Q}$, where $p_n$ and $q_n$ are relatively prime. For a start, we choose some arbitrary $\alpha_0 \in \mathbb{Q}$ and set $H_0 = \operatorname{id}$. In step $n+1$ of the construction we build an additional conjugation map $h_{n+1}$ satisfying
\begin{equation}\label{eq:commute}
	h_{n+1} \circ R_{\alpha_{n}} = R_{\alpha_{n}} \circ h_{n+1}.
\end{equation}
In the measure-theoretic AbC construction this map $h_{n+1}$ will be a permutation of partition elements $\Delta^{i,s}_{k_nq_n,s_{n+1}} $
of the partition $\xi_{k_nq_n,s_{n+1}}$
with some $k_n,s_{n+1} \in \Z^+$, where we make the following requirement on the sequence $(s_n)_{n\in \N}$:
\begin{enumerate}[label={\bf(R\arabic*)}]
	\item\label{item:R1} $s_{n+1}$ is a multiple of $s_n$ and $s_n \to \infty$ as $n \to \infty$. 
\end{enumerate}

	Finally, we complete stage $n+1$ of the construction process by setting
\begin{equation}\label{eq:alpha}
	\alpha_{n+1} = \frac{p_{n+1}}{q_{n+1}} = \alpha_n + \frac{1}{k_nl_nq^2_n}
\end{equation}
for some sequence $(l_n)_{n\in \N}$ of positive integers satisfying
\begin{equation}\label{eq:lgeneral}
	\sum_{n\in \N} \frac{1}{l_n}<\infty.
\end{equation}
In case of smooth (or even real-analytic) AbC constructions in Section~\ref{subsec:smooth}, the numbers $l_n$ will have to grow sufficiently fast to allow convergence of the sequence $(T_n)_{n\in \N}$ to a limit diffeomorphism. In the so-called \emph{abstract AbC method} of this section we obtain a MPT as a limit of periodic processes.

\begin{lem}\label{lem:MPconv}
	Let $(T_n)_{n\in \N}$ be a sequence of MPT's constructed by the abstract AbC method with $(s_n)_{n\in \N}$ satisfying requirement~\ref{item:R1} and with any sequence $(l_n)_{n\in\N}$ satisfying \eqref{eq:lgeneral}. Then $(T_n)_{n\in \N}$ converges in the weak topology to a measure-preserving transformation $T$. Furthermore, the sequence of partitions
	\begin{equation}\label{eq:eta}
		\zeta_n \coloneqq H_n(\xi_{q_n,s_n})
	\end{equation}
	is decreasing and generating. We have for all $-q_{n+1}\leq t\leq q_{n+1}$ that
	\begin{equation}\label{eq:Tclose}
	d\left(\zeta_n , T^t, T^t_{n+1}\right) \coloneqq	\sum_{c\in \zeta_n} \lambda\left(T^t(c)\triangle T^t_{n+1}(c)\right) < \sum_{i=n+1}^{\infty}\frac{1}{l_{i}}
	\end{equation}
\end{lem}

\begin{proof}
	We recall that $h_n$ acts as a permutation of the atoms of $\xi_{k_{n-1}q_{n-1},s_n}$. Since $q_n=k_{n-1}l_{n-1}q^2_{n-1}$ by \eqref{eq:alpha}, $\xi_{q_n,s_n}$ refines $\xi_{k_{n-1}q_{n-1},s_n}$. Accordingly, we can view $h_n$ as permuting the atoms of $\xi_{q_n,s_n}$. In this sense, each $h_m$ is a permutation of $\xi_{q_n,s_n}$ for $m\leq n$. Hence, $\zeta_n = H_n(\xi_{q_n,s_n})$ is decreasing and generating.	
	Using $h_{i+1}\circ R_{\alpha_i}=R_{\alpha_i}\circ h_{i+1}$ we also note for $-q_{n+1} \leq t\leq q_{n+1}$ and $n$ sufficiently large that
	\begin{align*}
		& \sum_{i=n}^{\infty}\sum_{c\in \zeta_i} \lambda\left(T^t_{i+2}(c)\triangle T^t_{i+1}(c)\right) \\
		= & \sum_{i=n}^{\infty}\sum_{c\in \zeta_i} \lambda\left(H_{i+2}R^t_{\alpha_{i+2}}H^{-1}_{i+2}(c)\triangle H_{i+1} \circ h_{i+2} \circ R^t_{\alpha_{i+1}} \circ h^{-1}_{i+2} \circ H^{-1}_{i+1}(c)\right) \\
		= & \sum_{i=n}^{\infty}\sum_{\tilde{c}\in \xi_{q_i,s_i}} \lambda\left(R^t_{\alpha_{i+2}}\circ h^{-1}_{i+2}\circ h^{-1}_{i+1}(\tilde{c})\triangle  R^t_{\alpha_{i+1}} \circ h^{-1}_{i+2} \circ h^{-1}_{i+1}(\tilde{c})\right) \\
		= & \sum_{i=n}^{\infty}\sum_{d\in \xi_{k_{i+1}q_{i+1},s_{i+2}}} \lambda\left(R^t_{\alpha_{i+2}-\alpha_{i+1}}(d)\triangle d\right) 
		\leq \sum_{i=n}^{\infty}\frac{\abs{t}}{l_{i+1}q_{i+1}} \leq \sum_{i=n+1}^{\infty}\frac{1}{l_{i}} < \varepsilon
	\end{align*}
	by our assumption on $(l_n)_{n\in\N}$. In particular, this shows convergence of $(T_n)_{n\in \N}$ to a measure-preserving transformation $T$ in the weak topology by Fact~\ref{fact:ConvWT}. Moreover, we have shown by triangle inequality that
	\eqref{eq:Tclose} holds. 
\end{proof}

\begin{rem} \label{rem:AbCperPro}
	Since $R_{\alpha_n}$ gives a periodic process with partition $\xi_n \coloneqq \xi_{q_n,s_n}$, the map $T_n = H_n \circ R_{\alpha_{n}} \circ H^{-1}_n$ induces a periodic process with partition $\zeta_n = H_n(\xi_n)$, which we denote by $\tau_n$. When we want to view $\tau_n$ as a collection of towers, we take the bases of $\tau_n$ to be the sets $\Meng{H_n(\Delta^{0,s}_{q_n , s_n})}{0\leq s < s_n}$.
\end{rem}

\subsection{\label{subsec:WMcrit}Criterion for weak mixing}

We prove the following criterion for weak mixing in our setting of abstract AbC constructions. Our criterion bases upon the original construction of weakly mixing diffeomorphisms in \cite[section 5]{AK70}.

\begin{prop}[Criterion for weak mixing] \label{prop:WM}
  Let $(T_n)_{n\in \N}$ be a sequence of measure-preserving transformations constructed by the abstract AbC method with $(s_n)_{n\in \N}$ satisfying requirement~\ref{item:R1} and with any sequence $(l_n)_{n\in\N}$ such that $\sum_{n\in \N}\frac{1}{l_n}<\infty$. Furthermore, we suppose that there is an increasing sequence $(m_n)_{n\in \N}$ of positive integers $m_n \leq q_{n+1}$ such that for every $n \in \N$ we have
  \begin{equation} \label{eq:WMassum}
  	\begin{split}
  	&\abs{\lambda\left( h_{n+1}\circ R^{m_n}_{\alpha_{n+1}} \circ h^{-1}_{n+1}(\Delta^{i,t}_{q_n,s_{n}}) \cap \Delta^{j,u}_{q_n,s_{n}}\right)- \lambda(\Delta^{i,t}_{q_n,s_{n}})\cdot \lambda(\Delta^{j,u}_{q_n,s_{n}})} \\
  	< & \frac{1}{n}\cdot \lambda(\Delta^{i,t}_{q_n,s_{n}})\cdot \lambda(\Delta^{j,u}_{q_n,s_{n}})
  	\end{split}
  \end{equation}
  for all $0\leq i,j<q_n$ and $0\leq t,u< s_{n}$. Then $(T_n)_{n\in \N}$ converges in the weak topology to a weakly mixing transformation $T$. The integers $m_n$ will be called \emph{mixing times}.
\end{prop}

\begin{proof}
	Since $\sum_{n\in \N}\frac{1}{l_n}<\infty$, Lemma \ref{lem:MPconv} implies the convergence of our sequence $(T_n)_{n\in \N}$ of AbC transformations to a measure-preserving transformation $T$. 
	
	By \cite[Theorem 5.1]{AK70} a measure-preserving transformation $T$ is weakly mixing if and only if there exists a sequence of finite partitions $\eta_n$ converging to the decomposition into points (that is, for every measurable set $A$ and for every $n\in \N$ there exists a set $A_n$, which is a union of elements of $\eta_n$, such that $\lim_{n\to \infty}\lambda(A\triangle A_n)=0$) and an increasing sequence of positive integers $m_n$ such that
	\begin{equation}\label{eq:WMcritAK}
		\lim_{n\to \infty} \sum_{c_1,c_2 \in \eta_n} \abs{\lambda\left(T^{m_n}c_1 \cap c_2\right) - \lambda(c_1) \lambda(c_2)} =0.
	\end{equation}
We take the partitions $\eta_n= H_n\left(\xi_{q_n,s_n}\right)$ as in \eqref{eq:eta}. By Lemma~\ref{lem:MPconv}, $(\eta_n)_{n\in \N}$ converges to the decomposition into points. Since $\sum_{n\in \N}\frac{1}{l_n} < \infty$ and $m_n \leq q_{n+1}$, Lemma~\ref{lem:MPconv} also implies that $d(\eta_n, T^{m_n},T^{m_n}_{n+1}) \to 0$ as $n \to \infty$. Hence, in order to check \eqref{eq:WMcritAK} it suffices to show that
\begin{equation}\label{eq:WMproof}
	\lim_{n\to \infty} \sum_{c_1,c_2 \in \eta_n} \abs{\lambda\left(T^{m_n}_{n+1}c_1 \cap c_2\right) - \lambda(c_1) \lambda(c_2)} =0.
\end{equation}
For $c_1= H_n(\Delta^{i,t}_{q_n,s_{n}}) \in \eta_n$ and $c_2=H_n(\Delta^{j,u}_{q_n,s_{n}})\in \eta_n$ we calculate
\begin{align*}
	&\abs{\lambda\left(T^{m_n}_{n+1}c_1 \cap c_2\right) - \lambda(c_1) \lambda(c_2)}\\
	=& \abs{\lambda\left( H_{n+1}\circ R^{m_n}_{\alpha_{n+1}} \circ H^{-1}_{n+1}\left(H_n(\Delta^{i,t}_{q_n,s_{n}})\right) \cap H_n(\Delta^{j,u}_{q_n,s_{n}}) \right) - \lambda(H_n(\Delta^{i,t}_{q_n,s_{n}}))\cdot \lambda(H_n(\Delta^{j,u}_{q_n,s_{n}}))} \\
	=& 	\abs{\lambda\left( h_{n+1}\circ R^{m_n}_{\alpha_{n+1}} \circ h^{-1}_{n+1}\left(\Delta^{i,t}_{q_n,s_{n}}\right) \cap \Delta^{j,u}_{q_n,s_{n}} \right) - \lambda(\Delta^{i,t}_{q_n,s_{n}})\cdot \lambda(\Delta^{j,u}_{q_n,s_{n}})} \\
	< & \frac{1}{n}\cdot \lambda(\Delta^{i,t}_{q_n,s_{n}})\cdot \lambda(\Delta^{j,u}_{q_n,s_{n}}),	
\end{align*}
where we used assumption \eqref{eq:WMassum} in the last step. Hence, equation \eqref{eq:WMproof} is satisfied and we conclude that $T$ is weakly mixing.
\end{proof}

\subsection{\label{subsec:constr}Construction of weakly mixing AbC transformations}
	In our construction of weakly mixing AbC transformations we will take
	\begin{equation}\label{eq:k}
		k_n = 2^{n+2}q_nC_n
	\end{equation}
with some $C_n \in \Z^+$ that is a multiple of $s^2_n$. We also define
\begin{equation}\label{eq:m}
	m_n \coloneqq C_nl_nq_n = \frac{q_{n+1}}{2^{n+2}q^2_n} 
\end{equation}
which we show to be mixing times in the proof of Proposition~\ref{prop:WM2}. By definition we have
\begin{equation}
	m_n \alpha_{n+1} \equiv \frac{1}{2^{n+2}q^2_n} \mod 1.
\end{equation}
	Furthermore, the conjugation map $h_{n+1}$ will be a composition 
	\begin{equation}\label{eq:compos}
		h_{n+1}=h_{n+1,2} \circ h_{n+1,1}
	\end{equation}
	of two measure-preserving and $1/q_n$-equivariant transformations $h_{n+1,1}$ and $h_{n+1,2}$. 
	
	Here, $h_{n+1,1}$ acts as varying horizontal translations by multiples of $1/q_n$ on vertical stripes of full length. In particular, some of these vertical stripes $\Delta^{i,0}_{k_nq_n,1}$ are mapped into different fundamental domains $\Delta^{j,0}_{q_n,1}$. We sometimes refer to $h_{n+1,1}$ as the \emph{twist map} in contrast to the untwisted AbC constructions in \cite{FW1} and \cite{FW3}. The definitions of $h_{n+1,1}$ and $m_n$ will ensure that $h_{n+1,1} \circ R^{m_n}_{\alpha_{n+1}}\circ h^{-1}_{n+1,1}$ distributes each fundamental domain $\Delta^{i,0}_{q_n,1}$ \emph{almost uniformly in the horizontal direction} over all the fundamental domains (see Lemma~\ref{lem:horidistri} and Figure~\ref{fig:fig1}). 
	
	In contrast, the map $h_{n+1,2}$ leaves the fundamental domains invariant. It will allow us to also obtain \emph{almost uniform distribution in the vertical direction} of rectangles $\Delta^{i,t}_{q_n,s_n}$ into rectangles $\Delta^{j,u}_{q_n,s_n}$ under $h_{n+1} \circ R^{m_n}_{\alpha_{n+1}}\circ h^{-1}_{n+1}$. Altogether, we will be able to verify assumption \eqref{eq:WMassum} from our criterion for weak mixing in Proposition~\ref{prop:WM}.
	
	\subsubsection{\label{subsubsec:h1}Construction of the conjugation map $h_{n+1,1}$}
	Each $0\leq i < k_nq_n$ can be written in a unique way as
	\begin{equation}\label{eq:idecomp}
	i=i_1\cdot k_n + i_2\cdot C_n+i_3
	\end{equation}
	with $0\leq i_1<q_n$, $0\leq i_2<2^{n+2}q_n$, and $0\leq i_3 < C_n$. Using this decomposition we have
	\begin{equation*} 
	\frac{i}{k_nq_n}=\frac{i_1}{q_n} + \frac{i_2}{2^{n+2}q^2_n}+\frac{i_3}{k_nq_n}.
	\end{equation*}
	In particular, we notice for our number $m_n$ from \eqref{eq:m} that
	\begin{equation} \label{eq:idecompCor}
	R^{m_n}_{\alpha_{n+1}}\left( \Delta^{i_1k_n+i_2C_n+i_3,0}_{k_nq_n,1}\right) = \begin{cases}
		\Delta^{i_1k_n+(i_2+1)C_n+i_3,0}_{k_nq_n,1}, & \text{ if } i_2<2^{n+2}q_n-1, \\
		\Delta^{((i_1+1)\mod q_n)\cdot k_n+i_3,0}_{k_nq_n,1}, & \text{ if } i_2=2^{n+2}q_n-1.
	\end{cases}
	\end{equation}	
	We use the decomposition from \eqref{eq:idecomp} to define the conjugation map $h_{n+1,1}$ by 
	\begin{equation*}
		h_{n+1,1}\left( \Delta^{i,0}_{k_nq_n,1}\right) = h_{n+1,1}\left( \Delta^{i_1k_n+i_2C_n+i_3,0}_{k_nq_n,1}\right) = \Delta^{\left((i_1+a_n(i_2))\mod q_n\right) \cdot k_n+i_2C_n+i_3,0}_{k_nq_n,1}
	\end{equation*}
with
\begin{equation} \label{eq:an}
	a_n(i_2) = \begin{cases}
		0, & \text{ if  $0\leq i_2 < 2^{n+1}q_n$, $i_2$ even,} \\
		\frac{i_2+1}{2}\mod q_n, & \text{ if  $0\leq i_2 < 2^{n+1}q_n$, $i_2$ odd,} \\
		\frac{i_2}{2}+1 \mod q_n, & \text{ if  $2^{n+1}q_n \leq i_2 < 2^{n+2}q_n$, $i_2$ even,} \\
		1, & \text{ if  $2^{n+1}q_n \leq i_2 < 2^{n+2}q_n$, $i_2$ odd.} \\
	\end{cases}
\end{equation}
As required, we have
$
	h_{n+1,1}\circ R_{1/q_n} = R_{1/q_n}\circ h_{n+1,1}.
$

\begin{rem} \label{rem:an}
	The choice of $a_n$ allows us to deduce the following Lemma~\ref{lem:horidistri} on almost uniform distribution in the horizontal direction. The underlying mechanism is illustrated in Figure~\ref{fig:fig1}. This distribution result in turn is used in the proof of weak mixing in Proposition~\ref{prop:WM2}. Our mechanism to produce weak mixing is inspired by the original construction of weakly mixing AbC transformations in \cite[section~5]{AK70}. We use the different definitions of $a_n(i_2)$ for indices $0\leq i_2 < 2^{n+1}q_n$ and $2^{n+1}q_n \leq i_2 < 2^{n+2}q_n$ in order to achieve that spacer symbols in our symbolic representation will occur at the same positions in forward and reverse words (see Lemma~\ref{lem:ReverseTwist}).
\end{rem}

\begin{figure}
	\centering
	\includegraphics[width=\textwidth]{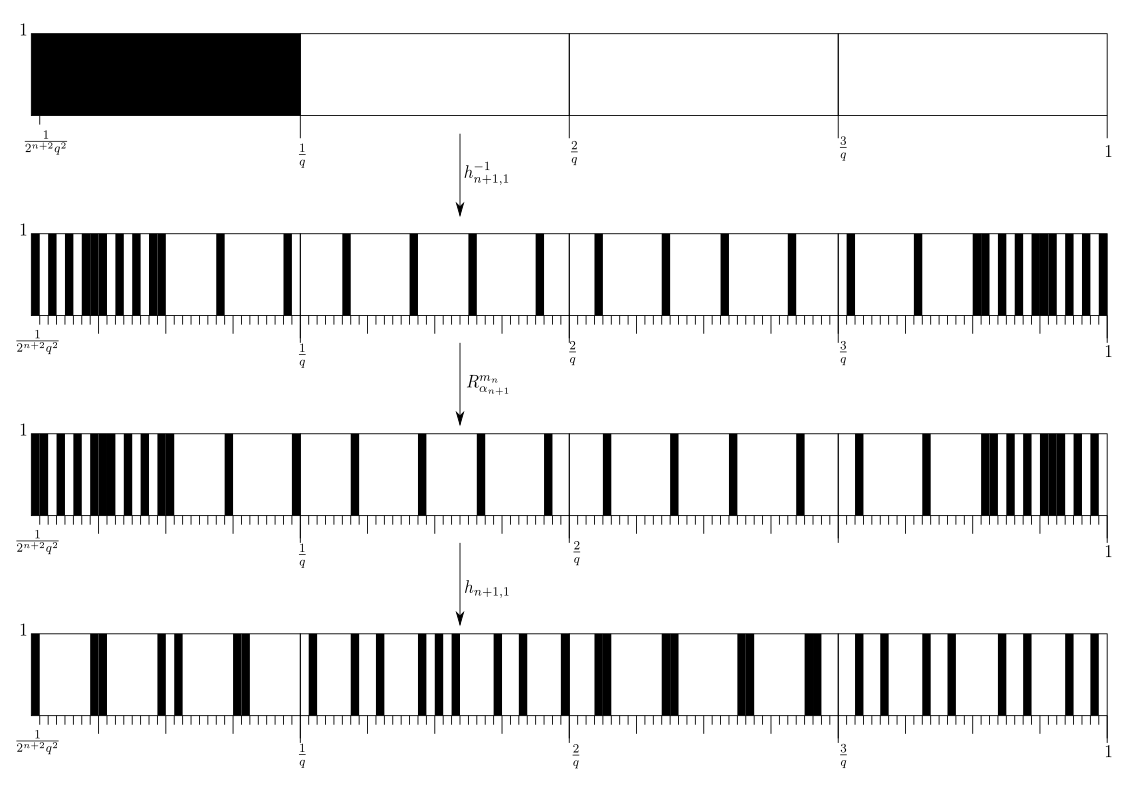}
	\caption{Visualization of the action of $h_{n+1,1} \circ R^{m_n}_{\alpha_{n+1}}\circ h^{-1}_{n+1,1}$. Here, $n=1$ and $q=q_1=4$ (for illustration purposes; actual values will be much larger).}
	\label{fig:fig1}
\end{figure}

\begin{lem}\label{lem:horidistri}
	For all pairs of $0\leq j,k < q_n$ we have
	\begin{align*}
		& 2^{n+2}-1 \leq \\
		&\abs{\Meng{0\leq i_2<2^{n+2}q_n}{h_{n+1,1} \circ R^{m_n}_{\alpha_{n+1}}\circ h^{-1}_{n+1,1}\left(\Delta^{2^{n+2}q_n \cdot j + i_2,0}_{2^{n+2}q^2_n,1}\right)=\Delta^{2^{n+2}q_n \cdot k + i_2,0}_{2^{n+2}q^2_n,1}}} \\
		&\leq 2^{n+2}+1.
	\end{align*}
	In particular, we conclude
	\begin{equation*}
		\abs{\lambda\left( h_{n+1,1} \circ R^{m_n}_{\alpha_{n+1}}\circ h^{-1}_{n+1,1}(\Delta^{j,0}_{q_n,1}) \triangle \Delta^{k,0}_{q_n,1}\right) - \frac{1}{q^2_n}}<\frac{1}{2^{n+2}} \cdot \lambda(\Delta^{j,0}_{q_n,1}) \cdot \lambda(\Delta^{k,0}_{q_n,1}).
	\end{equation*}
\end{lem}

\begin{proof}
	Using our observation from \eqref{eq:idecompCor} we obtain that
	\begin{equation*}
		\begin{split}
		& h_{n+1,1} \circ R^{m_n}_{\alpha_{n+1}}\circ h^{-1}_{n+1,1}\left(\Delta^{jk_n+i_2C_n+i_3,0}_{k_nq_n,1}\right) \\
		=& \begin{cases}
			\Delta^{\left((j+a_n(i_2+1)-a_n(i_2))\mod q_n \right)\cdot k_n + (i_2+1)\cdot C_n +i_3 , 0}_{k_nq_n, 1}, & \text{ if } 0\leq i_2<2^{n+2}q_n-1, \\
			\Delta^{j\cdot k_n +i_3 , 0}_{k_nq_n, 1}, & \text{ if }  i_2=2^{n+2}q_n-1.
		\end{cases}
	\end{split}
	\end{equation*}
The definition of $a_n(i_2)$ in \eqref{eq:an} implies that for every $2\leq k <q_n$ and every $0\leq \ell<2^{n+1}$ there are two indices $i_2 \in \{\ell 2q_n, \dots , (\ell+1) 2 q_n -1\}$ such that $a_n(i_2+1)-a_n(i_2) = k \mod q_n$. Thus, there are $2^{n+2}$ many indices $0\leq i_2 <2^{n+2}q_n$ with $a_n(i_2+1)-a_n(i_2) = k \mod q_n$. Similarly, we count that there are $2^{n+2}+1$ many indices $0\leq i_2 <2^{n+2}q_n$ such that $h_{n+1,1} \circ R^{m_n}_{\alpha_{n+1}}\circ h^{-1}_{n+1,1}$ causes a horizontal translation by $1/q_n$ on $\Delta^{2^{n+2}q_n \cdot j + i_2,0}_{2^{n+2}q^2_n,1}$. Moreover, there are $2^{n+2}-1$ many indices $0\leq i_2 <2^{n+2}q_n$ such that $h_{n+1,1} \circ R^{m_n}_{\alpha_{n+1}}\circ h^{-1}_{n+1,1}$ does not cause a horizontal translation on $\Delta^{2^{n+2}q_n \cdot j + i_2,0}_{2^{n+2}q^2_n,1}$.
\end{proof}
	
	\subsubsection{\label{subsubsec:h2}Construction of the conjugation map $h_{n+1,2}$}
We start the description of the map $h_{n+1,2}$ on the fundamental domain $\Delta^{0,0}_{q_n,1}$. Given $0\leq i < 2^{n+2}q_n$ and $0\leq s<s_{n+1}$ we can associate a $C_n$-tuple 
\begin{equation*}
\mathfrak{b}_n(i,s)\coloneqq \left(b_n(i,0,s),b_n(i,1,s),\dots , b_n(i,C_n-1,s) \right) \in \{0,1,\dots , s_n-1 \}^{C_n}
\end{equation*}
so that for $0\leq j < s_n$
	\begin{equation*}
		h_{n+1,2}\left(\Delta^{iC_n+j,s}_{k_nq_n,s_{n+1}} \right) \subseteq \Delta^{iC_n,b_n(i,j,s)}_{k_nq_n,s_{n}},
	\end{equation*}
where we impose the following conditions:

\begin{enumerate}[label={\bf(R\arabic*)}]
	\setcounter{enumi}{1}
	\item\label{item:R2} For every  $0\leq s < s_{n+1}$, $0\leq t <s_n$ we have 
	\begin{equation}\label{eq:R2}
		r\left(t,\,\mathfrak{b}_n(0,s)\dots \mathfrak{b}_n(2^{n+2}q_n-1,s)\right) = \frac{k_n}{s_n},
	\end{equation}
	where we recall that $k_n$ is a multiple of $s_n$ by our condition \eqref{eq:k}.
    \item\label{item:R3} We assume that the map $s \mapsto \left(\mathfrak{b}_n(0,s),\dots , \mathfrak{b}_n(2^{n+2}q_n,s)\right)$ is one-to-one. 
\end{enumerate}
In other words, requirement~\ref{item:R2} expresses the strong uniformity of symbols $t\in \{0,1,\dots , s_n\}$ in the sequences $\mathfrak{b}_n(0,s)\dots \mathfrak{b}_n(2^{n+2}q_n-1,s)$ while \ref{item:R3} says that we get different sequences for different $s\in \{0,1,\dots, s_{n+1}-1\}$.

Finally, the definition of $h_{n+1,2}$ is extended to the whole space by 
$ 
	h_{n+1,2}\circ R_{1/q_n} = R_{1/q_n}\circ h_{n+1,2}.
$

	\subsubsection{Verification of the weak mixing property}
	In order to prove the weak mixing property for our AbC transformation we have to strengthen the uniformity assumption~\ref{item:R2} to the following requirement.
	\begin{enumerate}[label={\bf(R\arabic*)}]
		\setcounter{enumi}{3}
		\item\label{item:R4} For every $0\leq  i<2^{n+2}q_n$, every $0\leq s <s_{n+1}$ and all pairs $(t,u)$ with $0\leq t,u <s_n$  we have that
		\begin{equation}\label{eq:R4}
			r\left(t,u,\mathfrak{b}_n(i,s), \mathfrak{b}_n(i+1 \mod 2^{n+2}q_n,s)\right) = \frac{C_n}{s^2_n},
		\end{equation}
		where we recall the notation for $r(\cdot , \cdot , \cdot , \cdot)$ from Definition~\ref{dfn:freq}.
	\end{enumerate}
	In other words, the requirement \ref{item:R4} says that all pairs $(t,u)$ occur uniformly in the adjacent tuples $\mathfrak{b}_n(i,s)$ and $\mathfrak{b}_n(i+1 \mod 2^{n+2}q_n,s)$.
	
	Then we can verify the assumptions of our criterion for weak mixing in Proposition~\ref{prop:WM} for the AbC constructions described above.	
	\begin{prop}\label{prop:WM2}
		Suppose that $(T_n)_{n\in \N}$ is a sequence of AbC transformations with parameters $(k_n)_{n\in \N}$ as in \eqref{eq:k}, $(s_n)_{n\in \N}$ satisfying requirement~\ref{item:R1}, and $(l_n)_{n\in \N}$ satisfying  $\sum_{n\in \N}\frac{1}{l_n}<\infty$. Furthermore, we assume conjugation maps of the form $h_{n+1}=h_{n+1,2} \circ h_{n+1,1}$ with maps $h_{n+1,1}$ as in Subsection~\ref{subsubsec:h1} and $h_{n+1,2}$ as in Subsection~\ref{subsubsec:h2} satisfying requirement~\ref{item:R4}. 
		Then $(T_n)_{n\in \N}$ converges in the weak topology to a weakly mixing transformation $T$.
	\end{prop}

\begin{proof}
	In order to apply Proposition \ref{prop:WM} we have to check condition \eqref{eq:WMassum}. For every $0\leq s < s_{n+1}$, $0\leq i_2 < 2^{n+2}q_n$, and $0\leq t < s_n$ there is a set of indices
	\[
	\mathcal{I}(i_2,t,s) \coloneqq \Meng{0\leq i_3 <C_n}{h_{n+1,2}\left( \Delta^{ i_2 C_n + i_3, s}_{k_nq_n , s_{n+1}}\right) \subseteq \Delta^{i_2C_n,t}_{k_n q_n , s_n}}.
	\] 
	Then we observe that
	\begin{align*}
		& h_{n+1} \circ R^{m_n}_{\alpha_{n+1}} \circ h^{-1}_{n+1,1} \left( h^{-1}_{n+1,2} \left( \Delta^{i_12^{n+2}q_n+i_2,t}_{2^{n+2}q^2_n,s_n}\right) \cap \Delta^{i_12^{n+2}q_n+i_2,s}_{2^{n+2}q^2_n,s_{n+1}}\right) \\
		= & h_{n+1} \circ R^{m_n}_{\alpha_{n+1}} \circ h^{-1}_{n+1,1} \left( h^{-1}_{n+1,2} \left( \bigcup^{C_n - 1}_{i_3=0} \Delta^{i_1 k_n + i_2 C_n + i_3 , t}_{k_nq_n,s_n}\right) \cap \bigcup^{C_n - 1}_{j_3=0} \Delta^{i_1 k_n + i_2 C_n + j_3 , s}_{k_nq_n,s_{n+1}}\right)  \\
		= & h_{n+1,2} \circ h_{n+1 , 1} \circ R^{m_n}_{\alpha_{n+1}} \circ h^{-1}_{n+1,1} \left( \bigcup_{\ell \in \mathcal{I}(i_2,t,s)} \Delta^{i_1 k_n + i_2 C_n + \ell , s}_{k_nq_n,s_{n+1}}\right) \\
		= &  h_{n+1,2} \left( \bigcup_{\ell \in \mathcal{I}(i_2,t,s)} \Delta^{b k_n + ((i_2+1) \mod 2^{n+2}q_n) C_n + \ell , s}_{k_nq_n,s_{n+1}}\right)
	\end{align*}
with
\begin{equation*}
	b = \begin{cases}
		(i_1+a_n(i_2+1)-a_n(i_2)) \mod q_n, & \text{ if } i_2<2^{n+2}q_n-1, \\
		i_1, & \text{ if } i_2 = 2^{n+2}q_n.
	\end{cases}
\end{equation*}
Furthermore, for every $0\leq s < s_{n+1}$, $0\leq i_2 < 2^{n+2}q_n$, and all pairs $(t,u)$ with $0\leq t,u < s_n$ there is a set of indices
\begin{align*}
&\mathcal{I}(i_2,t,u,s)\coloneqq  \\
& \Meng{\ell \in \mathcal{I}(i_2,t,s)}{h_{n+1,2}\left( \Delta^{ (i_2+1 \mod 2^{n+2}q_n)\cdot  C_n + \ell, s}_{k_nq_n , s_{n+1}}\right) \subseteq \Delta^{(i_2+1 \mod 2^{n+2}q_n)\cdot C_n,u}_{k_n q_n , s_n}}.
\end{align*}
By requirement \ref{item:R4} we have $
	\abs{\mathcal{I}(i_2,t,u,s)} = C_n/s^2_n$.
Continuing the calculation from above, we obtain that
\begin{equation*}
	\begin{split}
	& \lambda\Bigg(h_{n+1} \circ R^{m_n}_{\alpha_{n+1}} \circ h^{-1}_{n+1,1} \left( h^{-1}_{n+1,2} \left( \Delta^{i_12^{n+2}q_n+i_2,t}_{2^{n+2}q^2_n,s_n}\right)\cap \Delta^{i_12^{n+2}q_n+i_2,s}_{2^{n+2}q^2_n,s_{n+1}}\right) \\
	& \quad \quad   \cap \Delta^{b2^{n+2}+((i_2+1) \mod 2^{n+2}q_n) , u}_{2^{n+2}q^2_n,s_n} \Bigg) \\
	= & \lambda \left(h_{n+1,2} \left( \bigcup_{\ell \in \mathcal{I}(i_2,t,u,s)} \Delta^{b k_n + ((i_2+1) \mod 2^{n+2}q_n) C_n + \ell , s}_{k_nq_n,s_{n+1}}\right) \right) \\
	= & \frac{1}{s^2_n}\cdot \lambda\left(\Delta^{i_12^{n+2}q_n+i_2,s}_{2^{n+2}q^2_n,s_{n+1}}\right).
	\end{split}
\end{equation*}
Since this holds true for every $0\leq s < s_{n+1}$, we conclude that
\begin{align*}
	& \lambda\left(h_{n+1} \circ R^{m_n}_{\alpha_{n+1}} \circ h^{-1}_{n+1} \left(\Delta^{i_12^{n+2}q_n+i_2,t}_{2^{n+2}q^2_n,s_n}\right) \cap  \Delta^{j_12^{n+2}q_n+((i_2+1) \mod 2^{n+2}q_n ) , u}_{2^{n+2}q^2_n,s_n} \right) \\
	= & \begin{cases}
		\frac{1}{s^2_n} \cdot \frac{1}{2^{n+2}q^2_n}, & \text{ if } j_1 = b, \\
		0, & \text{ otherwise}. 
	\end{cases}
\end{align*}
Then we use Lemma \ref{lem:horidistri} to estimate for any $0\leq i_1,j_1<2^{n+2}q_n$ and $0\leq t,u<s_n$ that
\begin{align*}
	\lambda\left( h_{n+1}\circ R^{m_n}_{\alpha_{n+1}} \circ h^{-1}_{n+1}(\Delta^{i_1,t}_{q_n,s_{n}}) \cap \Delta^{j_1,u}_{q_n,s_{n}}\right) 
	\geq  \left(1-\frac{1}{2^{n+2}}\right) \cdot \left(\frac{1}{s_nq_n}\right)^2 
\end{align*}
as well as
\begin{align*}
	\lambda\left( h_{n+1}\circ R^{m_n}_{\alpha_{n+1}} \circ h^{-1}_{n+1}(\Delta^{i_1,t}_{q_n,s_{n}}) \cap \Delta^{j_1,u}_{q_n,s_{n}}\right) 
	\leq  \left(1+\frac{1}{2^{n+2}}\right) \cdot \left(\frac{1}{s_nq_n}\right)^2.
\end{align*}
Altogether, assumption \eqref{eq:WMassum} is satisfied and Proposition~\ref{prop:WM} yields that $T$ is weakly mixing.
\end{proof}
	
	\subsection{\label{subsec:symbolic}Symbolic representation of our weakly mixing AbC transformations}
	We follow the approach in \cite[section 7]{FW1} to find a symbolic representation for our specific twisted constructions of weakly mixing AbC transformations from the previous subsection.
	
	\subsubsection{The dynamical and geometric orderings} 
	We start by recalling the dynamical and geometric orderings of intervals from \cite[section 7.2]{FW1}. For $q \in \Z^+$ we let 
	$\mathcal{I}_q \coloneqq \Meng{I^i_q}{0\leq i <q}$
	be the partition of $[0,1)$ and $\mathbb{S}^1=\R/\Z$, respectively, with atoms
	\begin{equation*}
		I^i_q \coloneqq \left[ \frac{i}{q}, \frac{i+1}{q}\right).
	\end{equation*}
\begin{defn}
	The \emph{geometric ordering} of the intervals in $\mathcal{I}_q$ is given by
	\[
	I^i_q <_g I^j_q \ \text{ if and only if } \ i<j,
	\]
	that is, we order these intervals from left to right according to their left endpoints.
\end{defn}

	To define the dynamical ordering, we fix a rational number $\alpha = p/q$ with $p,q$ relatively prime. Set
	\begin{equation}\label{eq:ji}
		j_i = p^{-1}i \mod q ,
	\end{equation}
    where the $p^{-1}$ is the multiplicative inverse of $p$ modulo $q$. We also note that
    \begin{equation}\label{eq:ji-rel}
    	q-j_i = j_{q-i}.
    \end{equation}
    The rotation by $\alpha$ defined by $\mathcal{R}_{\alpha}: \mathbb{S}^1 \to \mathbb{S}^1, \ x \mapsto x+\alpha \mod 1$,
    gives us another ordering of the intervals in $\mathcal{I}_q$.
	\begin{defn}
		The \emph{dynamical ordering} of the intervals in $\mathcal{I}_q$ is given by 
		\[
		I^i_q <_d I^j_q \text{ iff there are $n<m<q$ such that $np=i \mod q$ and $mp=j \mod q$.}
		\]
	\end{defn}
	In other words, the list $I^0_q, \, \mathcal{R}_{\alpha}I^0_q,\, \mathcal{R}^2_{\alpha}I^0_q, \dots , \, \mathcal{R}^{q-1}_{\alpha}I^0_q$
	gives the dynamical ordering of $\mathcal{I}_q$.
	\begin{rem*}
		With $j_i = p^{-1}i \mod q $ from equation \eqref{eq:ji} the $i$-th interval in the geometric ordering, $I^i_q$ , is the $j_i$-th interval in the dynamical ordering.
	\end{rem*}

\subsubsection{An analysis on the circle}
     To find a symbolic representation of our AbC transformations we start with a simplified analysis of the projection to the horizontal $\mathbb{S}^1$-coordinate. We explore how the dynamical ordering determined by 
     \[
     \alpha_{n+1}=\frac{p_{n+1}}{q_{n+1}}=\alpha_n + \frac{1}{k_nl_nq^2_n} =: \alpha_n+\beta_n
     \]
     interacts with the dynamical ordering determined by $\alpha_n=p_n/q_n$ and the varying horizontal translation by $a_n(\cdot)$ caused by the conjugation map $h_{n+1,1}$. For that purpose, we introduce the notation $\overline{h}_{n+1,1}$ for the projectivized action of $h_{n+1,1}$ on $\mathbb{S}^1$. Furthermore, we divide the atoms of $\mathcal{I}_{k_nq_n}$ into $k_n$ many ordered sets $\omega_0,\dots , \omega_{k_n-1} $ defined by
     \begin{equation*}
     	\omega_j \coloneqq \Ordered{I^{j+tk_n}_{k_nq_n}}{0\leq t < q_n} = \Ordered{\mathcal{R}^t_{\alpha_n}(I^j_{k_nq_n})}{0\leq t < q_n} .
     \end{equation*}
     Each $\omega_j$ can be viewed as a word of length $q_n$ over the alphabet $\mathcal{I}_{k_nq_n}$. We now want to determine a $\mathcal{I}_{k_nq_n}$-name for the trajectory of $J\coloneqq [ 0, 1/q_{n+1})$ under 
     \[
     \Phi^m_{n+1}\coloneqq \overline{h}_{n+1,1} \circ \mathcal{R}^m_{\alpha_{n+1}} = \mathcal{R}^m_{\alpha_n}\circ \overline{h}_{n+1,1} \circ \mathcal{R}^m_{\beta_n},
     \]
     where we used the commutativity relation $h_{n+1,1}\circ R_{\alpha_n} = R_{\alpha_n} \circ h_{n+1,1}$. Recall that the twist map $h_{n+1,1}$ caused a varying horizontal translation by $a_n(\cdot)$ from \eqref{eq:an}. Based upon these horizontal translations we define
     \begin{equation}\label{eq:psi}
     	\psi_n(i) = j_{a_n(i)} \ \text{ for all } 0\leq i < 2^{n+2}q_n,
     \end{equation}
     with the numbers $j_{a_n(i)}$ as defined in equation \eqref{eq:ji}, that is,
     \begin{equation*} 
     	\psi_n(i) = \begin{cases}
     		0, & \text{ if  $0\leq i < 2^{n+1}q_n$, $i$ even,} \\
     		j_{\frac{i+1}{2}\mod q_n}, & \text{ if  $0\leq i < 2^{n+1}q_n$, $i$ odd,} \\
     		j_{\frac{i}{2}+1 \mod q_n}, & \text{ if  $2^{n+1}q_n \leq i < 2^{n+2}q_n$, $i$ even,} \\
     		j_1, & \text{ if  $2^{n+1}q_n \leq i < 2^{n+2}q_n$, $i$ odd.} \\
     	\end{cases}
     \end{equation*}
     
     We now follow our original interval $J=[ 0, 1/q_{n+1})$ through the $w_j$'s under the iterates $\Phi^m_{n+1}$. For $0\leq m< C_nl_nq_n$ we have $R^m_{\beta_n}(J)\subset I^{(m \mod q_n)\cdot 2^{n+2}q_n}_{2^{n+2}q^2_n}$. Since $a_n(0)=0$, we could write the $\mathcal{I}_{k_nq_n}$-name of any $x \in J$ in the first $C_nl_nq_n$ iterates as 
     $
     \omega^{l_n}_0\omega_1^{l_n}\dots \omega^{l_n}_{C_n-1}.
     $
     Applying $R^{C_nl_nq_n}_{\beta_n}$ on $J$ makes it the geometrically first $1/q_{n+1}$-subinterval of $I^1_{2^{n+2}q^2_n}$. On $I^1_{2^{n+2}q^2_n}$ the map $\overline{h}_{n+1,1}$ causes a horizontal translation by $1/q_n$ because of $a_n(1)=1$. Thus, $\overline{h}_{n+1,1} \circ \mathcal{R}^{C_nl_nq_n}_{\beta_n}(J)$ is a subinterval of $I^{C_n+k_n}_{k_nq_n}$, that is, the $j_1$-th element of $\omega_{C_n}$. Thus, we must wait $q_n-j_1$ further iterates to have $\Phi^{C_nl_nq_n+q_n-j_1}_{n+1}(J) \subset I^{C_n}_{k_nq_n}$. Then we can follow $l_n-1$ copies of $\omega_{C_n}$. With the remaining $j_1$ many iterates we can write the $\mathcal{I}_{k_nq_n}$-name of any $x \in J$ in the iterates $C_nl_nq_n\leq m < (C_n+1)l_nq_n$ as
     \[
     b^{q_n-j_1}\,\omega^{l_n-1}_{C_n}\,e^{j_1} \ = \ b^{q_n-\psi_n(1)}\,\omega^{l_n-1}_{C_n}\,e^{\psi_n(1)}.
     \]  
     Since $\overline{h}_{n+1,1} \circ \mathcal{R}^{(C_n+1)l_nq_n}_{\beta_n}(J)$ is a subinterval of $I^{C_n+1+k_n}_{k_nq_n}$ (i.e., the $j_1$-th element of $\omega_{C_n+1}$), we must wait $q_n-j_1$ further iterates before can follow $l_n-1$ copies of $\omega_{C_n+1}$. Altogether, 
     we can write the $\mathcal{I}_{k_nq_n}$-name of any $x \in J$ in the iterates $C_nl_nq_n\leq m < 2C_nl_nq_n$ as
     \[
     b^{q_n-\psi_n(1)}\,\omega^{l_n-1}_{C_n}\,e^{\psi_n(1)}\,b^{q_n-\psi_n(1)}\,\omega^{l_n-1}_{C_n+1}\,e^{\psi_n(1)} \dots b^{q_n-\psi_n(1)}\,\omega^{l_n-1}_{2C_n-1}\,e^{\psi_n(1)}.
     \]  
     
     Continuing like this, we deduce the $\mathcal{I}_{k_nq_n}$-name of any $x \in J$ in the iterates $0\leq m < k_nl_nq_n$ as 
     \begin{equation}\label{eq:pattern1}
     \prod^{2^{n+2}q_n-1}_{i=0} \prod^{C_n-1}_{c=0} b^{q_n-\psi_n(i)}\,(\omega_{iC_n+c})^{l_n-1}\,e^{\psi_n(i)}.
     \end{equation}
     
     Applying $R^{k_nl_nq_n}_{\beta_n}$ on $J$ makes it the geometrically first $1/q_{n+1}$-subinterval of $I^1_{q_n}$ and $I^{k_n}_{k_nq_n}$, respectively. The pattern in \eqref{eq:pattern1} would repeat itself up to the fact that we start in the $j_1$-th element of $\omega_0$. By the same reasoning as above, we can write  the
     $\mathcal{I}_{k_nq_n}$-name of any $x \in J$ in the iterates $k_nl_nq_n\leq m < 2k_nl_nq_n$ as 
     \begin{equation*}
     	\prod^{2^{n+2}q_n-1}_{i=0} \prod^{C_n-1}_{c=0} b^{q_n-\psi_n(i)-j_1 \mod q_n}\,(\omega_{iC_n+c})^{l_n -1}\,e^{\psi_n(i)+j_1\mod q_n}.
     \end{equation*}
     
     In this way, we obtain the $\mathcal{I}_{k_nq_n}$-name of any $x \in J$ in the iterates $0\leq m < k_nl_nq^2_n=q_{n+1}$ as 
     \begin{equation}\label{eq:CodeCircle}
     	\prod^{q_n-1}_{m=0}\prod^{2^{n+2}q_n-1}_{i=0} \prod^{C_n-1}_{c=0} b^{q_n-\psi_n(i)-j_m \mod q_n}\,(\omega_{iC_n+c})^{l_n-1}\,e^{\psi_n(i)+j_m \mod q_n},
     \end{equation}
	 that is, its coding is given by $\mathcal{C}^{\text{twist}}_n(\omega_0,\dots , \omega_{k_n-1})$ using our twisting operator from Definition~\ref{def:twist}. This finding motivated the definition of the twisting operator.
	 
	 For our goal to find a symbolic representation of our weakly mixing AbC transformations, we note that the action of $R_{\alpha}$ on $M\in \{\mathbb{T}^2, \mathbb{D}^2,\mathbb{A}\}$ exactly mimics the action of $\mathcal{R}_{\alpha}$ on the circle in the first coordinate. We use this to label all atoms $\Delta^{i,s}_{q_{n+1},s_{n+1}}$ of $\xi_{q_{n+1},s_{n+1}}$ by $b$ (respectively $e$) whose projection $I^i_{q_{n+1}} \in \mathcal{I}_{q_{n+1}}$ on the first coordinate is labelled with $b$ (respectively $e$) in \eqref{eq:CodeCircle}. 	 
	 Then we inductively define sequences of subsets $B_n$ and $E_n$ of $M$ as follows: 
	 \begin{defn}
	 	Put $B_0 = E_0 = \emptyset$. If $B_n$ (respectively $E_n$) has been defined, let $B_{n+1}$ (resp. $E_{n+1}$) be the union
	 	of $B_n$ (resp. $E_n$) with the set of all $x\in M$ whose projection onto the horizontal axis is contained
	 	in an atom of $\mathcal{I}_{q_{n+1}}$ labelled with $b$ (resp. $e$) in \eqref{eq:CodeCircle}. Furthermore, we define sets $B^{\prime}_{n+1} = B_{n+1}\setminus B_n$, $E^{\prime}_{n+1} = E_{n+1}\setminus E_n$, and $\Gamma_n = \Meng{x \in M}{\text{for all $m>n$: } x\notin H_m(B^{\prime}_m \cup E^{\prime}_m)}$.
	 	Finally, we put
	 	\begin{align*}
	 		B & \coloneqq \Meng{x\in M}{ \text{for some $m\leq n$: } x\in \Gamma_n \text{ and } x\in H_m(B_m)}, \\
	 		E & \coloneqq \Meng{x\in M}{ \text{for some $m\leq n$: } x\in \Gamma_n \text{ and } x\in H_m(E_m)} 
	 	\end{align*}
	 \end{defn}
	 For every $n \in \N$ the measure of $B^{\prime}_{n+1} \cup E^{\prime}_{n+1}$ is $1/l_n$ because the occurences of $b$ and $e$ comprise a proportion of $1/l_n$ of the symbols in \eqref{eq:CodeCircle}. Since $\Gamma_n \subseteq \Gamma_{n+1}$ and $\sum_{n\in \N}1/l_n < \infty$ by \eqref{eq:lgeneral}, the Borel-Cantelli Lemma then implies that for almost every $x \in M$ there is $m \in \N$ such that $x \in \Gamma_n$ for all $n>m$.
	 
	 \subsubsection{The symbolic representation}
	We follow \cite[section 7.5]{FW1}. We start by defining the partition of $M$,
	\begin{equation} \label{eq:partQ}
		\mathcal{Q}\coloneqq \{A_i:i<s_0\}\cup\{B,E\},
	\end{equation}
	where $A_i\coloneqq [0,1)\times[i/s_0,(i+1)/s_0)\setminus (B\cup E)$. 
	
	For the limit $T$ of our Anosov Katok process, we can construct $(T,\mathcal{Q})$-names for each $x\in M$ using the alphabet $\Sigma\cup\{b,e\}$, where $\Sigma\coloneqq\{a_i\}_{i=0}^{s_0-1}$. Hence, the name of point $x\in M$ will be an $f\in (\Sigma\cup\{b,e\})^\Z$ with $f(n)=a_i\iff T^n(x)\in A_i$, $f(n)=b\iff T^n(x)\in B$, and $f(n)=e\iff T^n(x)\in E$. 
	
	We describe how to associate a construction sequence to an AbC transformation $T$ obtained as the limit of periodic transformations $T_n\coloneqq H_n\circ R_{\alpha_n}\circ H_n^{-1}$. Let $H_{n+1}(\Delta^{0,s^*}_{q_{n+1},s_{n+1}})$ for some $s^*<s_{n+1}$ be the base of a tower of $\tau_{n+1}$, where $\tau_n$ is the periodic process given by $T_n$ on the partition 
	\begin{equation} \label{eq:zeta}
		\zeta_n \coloneqq H_n(\xi_n) \quad \text{ with } \quad \xi_n \coloneqq \xi_{q_n,s_n}.	
	\end{equation}
	Inductively we assume that for every $0\leq s <s_n$ the $(\tau_n,\mathcal{Q})$-name of the tower with base $H_n(\Delta^{0,s}_{q_n,s_n})$ is $u_s$. At the $(n+1)$-th stage of the construction we define for each $0\leq s^{\ast} < s_{n+1}$ words $w_{0,s^{\ast}},\dots , w_{k_n-1,s^{\ast}}$ by setting
	\begin{equation} \label{eq:StepWords}
		w_{j,s^{\ast}} = u_s \ \ \Longleftrightarrow \ \ h_{n+1,2}\left( \Delta^{j,s^{\ast}}_{k_nq_n , s_{n+1}}\right) \subseteq \Delta^{0,s}_{q_n, s_n}.
	\end{equation}
    We say that $(w_{0,s^{\ast}},\dots , w_{k_n-1,s^{\ast}})$ is the sequence of $n$-words associated with the $s^{\ast}$-th tower of $\tau_{n+1}$.
	
Then the $(\tau_{n+1},\mathcal{Q})$-name of the tower with base $H_{n+1}(\Delta^{0,s^*}_{q_{n+1},s_{n+1}})$ is given by
\begin{equation}
	\prod^{q_n-1}_{m=0} \prod^{2^{n+2}q_n-1}_{i=0}\prod^{C_n-1}_{c=0}b^{q_n-\psi_n(i)-j_m \mod q_n}\, w^{l_n-1}_{iC_n+c,s^*} \, e^{\psi_n(i)+j_m \mod q_n}.
\end{equation}

So, we can define a symbolic twist system: Put $\mathcal{W}_0\coloneqq\{a_i\}_{i=0}^{s_0-1}$; having defined $\mathcal{W}_n$ we define $\mathcal{W}_{n+1}$ as
\begin{equation*}
	\Meng{\mathcal{C}^{\text{twist}}_{n}(w_0,\ldots,w_{k_n-1})}{(w_0,\ldots,w_{k_n-1})\text{ is associated with a tower  in }\tau_{n+1}}.
\end{equation*}
We say that $\{\mathcal{W}_n\}_{n\in\N}$ is the construction sequence associated with the AbC construction. Under some conditions we can then show that our abstract weakly mixing constructions from Section~\ref{subsec:constr} are isomorphic to symbolic twist systems. This is the content of the subsequent proposition. It should be compared with \cite[Theorem 58]{FW1} for an analogous result that some untwisted AbC transformations are isomorphic to circular symbolic systems. 

\begin{prop} \label{prop:CodeAbC}
	Suppose that the measure-preserving system $(M,\mathcal{B},\lambda, T)$ is built by
	the AbC method with parameters $(k_n)_{n\in \N}$ as in \eqref{eq:k}, $(s_n)_{n\in \N}$ satisfying requirement~\ref{item:R1}, and $(l_n)_{n\in \N}$ satisfying  $\sum_{n\in \N}\frac{1}{l_n}<\infty$. Furthermore, we assume conjugation maps of the form $h_{n+1}=h_{n+1,2} \circ h_{n+1,1}$ with maps $h_{n+1,1}$ as in Subsection \ref{subsubsec:h1} and $h_{n+1,2}$ as in Subsection \ref{subsubsec:h2} satisfying requirements \ref{item:R2} and \ref{item:R3}. 
	Let $\mathcal{Q}$ be the partition defined in \eqref{eq:partQ}. Then the $\mathcal{Q}$-names describe a strongly uniform twisting construction sequence $\{\mathcal{W}_n\}_{n\in\N}$. Let $\mathbb{K}$ be the associated twisted system and $\phi:M\to\mathbb{K}$ be the map sending each $x\in M$ to its $\mathcal{Q}$-name. Then $\phi$ is one-to-one on a set of $\lambda$-measure one. Moreover, there is a unique non-atomic shift-invariant measure $\nu$ concentrating on the range of $\phi$. In particular, $(M,\mathcal{B},\lambda, T)$ is isomorphic to $(\mathbb{K},\mathcal{B},\nu,\text{sh})$. 
\end{prop}

\begin{proof}
	We note that our requirements \ref{item:R1} and \ref{item:R3} correspond to Requirements 1 and 3, respectively, in \cite{FW1}. Moreover, our requirement \ref{item:R2} implies that for each $\Delta^{0,u}_{q_n,s_n} \in \xi_{q_n,s_n}$ and every $0\leq s <s_{n+1}$ we have 
	\[
	\abs{\Meng{0\leq i <k_n}{\h_{n+1,2}(\Delta^{i,s}_{k_nq_n,s_{n+1}})\subseteq \Delta^{0,u}_{q_n,s_n}}}=\frac{k_n}{s_n},
	\]
	that is, the construction sequence $\{\mathcal{W}_n\}_{n\in\N}$ is strongly uniform (which corresponds to Requirement 2 in \cite{FW1}). Then the proof follows along the lines of the proof of \cite[Theorem 58]{FW1} using the twisting operator instead of the circular operator.  
\end{proof}

	\subsection{\label{subsec:smooth}Smooth realization}
	We now show how the AbC method can be used to construct $C^{\infty}$ diffeomorphisms isomorphic to the abstract AbC transformations described in Sections \ref{subsec:abstract} and \ref{subsec:constr}.
	
	\subsubsection{Approximating partition permutations by smooth diffeomorphisms}
	In our abstract construction of weakly mixing AbC transformations we used specific partition permutations $h_{n+1}$ introduced in Section \ref{subsec:constr}.
	We will show that on the manifold $M \in \{ \mathbb{T}^2, \mathbb{D}, \mathbb{A}\}$ we can find for each of these partition permutations $h_{n+1}$ an area-preserving $C^{\infty}$ diffeomorphism $h^{(\mathfrak{s})}_{n+1}$ that closely approximates $h_{n+1}$. The diffeomorphism $h^{(\mathfrak{s})}_{n+1}$ will coincide with the identity in a neighborhood of the boundary of $M$ and with the action of $h_{n+1}$ on the ``inner kernels''
	\begin{equation}\label{eq:InnerKernels}
	\tilde{\Delta}^{i,j}_{k_nq_n,s_{n+1}, \varepsilon_n} \coloneqq \Bigg[\frac{i+\varepsilon_n}{k_nq_n}, \frac{i+1-\varepsilon_n}{k_nq_n}\Bigg] \times \Bigg[\frac{j+\varepsilon_n}{s_{n+1}}, \frac{j+1-\varepsilon_n}{s_{n+1}} \Bigg]
	\end{equation} 
	of all the partition elements $\Delta^{i,j}_{k_nq_n,s_{n+1}}\in \xi_{k_nq_n,s_{n+1}}$. To be more precise, we show the following realization result in this section.
	
	\begin{prop}\label{prop:hReal}
		Let $h_{n+1}=h_{n+1,2}\circ h_{n+1,1}$ be a partition permutation as defined in Section \ref{subsec:constr} satisfying requirement \ref{item:R2}. Then for any $\varepsilon>0$ there is a diffeomorphism $h^{(\mathfrak{s})}_{n+1}\in \text{Diff}^{\,\infty}_{\,\lambda}(M)$ such that
		 $h^{(\mathfrak{s})}_{n+1} \circ R_{1/q_n}=R_{1/q_n} \circ h^{(\mathfrak{s})}_{n+1}$,
		$h^{(\mathfrak{s})}_{n+1}$ is the identity in a neighborhood of the boundary of $M$,
		and for all $0\leq i<k_nq_n$ and $0\leq j<s_{n+1}$ we have $h^{(\mathfrak{s})}_{n+1}(x) = h_{n+1}(x)$ for all $x\in \tilde{\Delta}^{i,j}_{k_nq_n,s_{n+1},\varepsilon}$.
	\end{prop}
	
	To build $h^{(\mathfrak{s})}_{n+1}$ we start with the following realization result from \cite{FW1} based on ``Moser's trick'' (there are similar results in \cite[section 1]{AK70}).
	
	\begin{lem}\label{lem:FWreal}
		Let $\sigma$ be a permutation of the rectangles $\Delta^{i,j}_{kq,s}\in \xi_{kq,s}$ which commutes with $R_{1/q}$ and is untwisted, that is, $\sigma(\Delta^{0,0}_{q,1})=\Delta^{0,0}_{q,1}$. Then for any $\varepsilon>0$ there is a diffeomorphism $\phi\in \text{Diff}^{\,\infty}_{\,\lambda}(M)$ such that
		$\phi \circ R_{1/q}=R_{1/q} \circ \phi$, $\phi$ is the identity in a neighborhood of the boundary of $M$,
			for all $0\leq i<kq$ and $0\leq j<s$ we have
			$\phi(x) = \sigma(x)$ for all $x\in \tilde{\Delta}^{i,j}_{kq,s,\varepsilon}$.
			
		We say that the partition permutation $\sigma$ is $\varepsilon$-approximated by $\phi$.
	\end{lem} 

\begin{proof}
	By \cite[Theorem 35]{FW1} we can $\varepsilon$-approximate $\sigma|_{\Delta^{0,0}_{q,1}}$ by a smooth area-preserving diffeomorphism $\phi$ of $[0,1/q]\times [0,1]$ that is the identity in a neighborhood of the boundary. Hence, we can extend $\phi$ to a diffeomorphism of $M$ by $\phi \circ R_{1/q}=R_{1/q} \circ \phi$. Then $\phi$ is still a $\varepsilon$-approximation of $\sigma$ since $\sigma$ commutes with $R_{1/q}$ as well.
\end{proof}
	
	In particular, this lemma allows us to find smooth approximations to the untwisted conjugation map $h_{n+1,2}$.
	
	\begin{lem}\label{lem:h2real}
		Let $h_{n+1,2}$ be a partition permutation as defined in Subsection \ref{subsubsec:h2} with tuples $\mathfrak{b}_n(i,s)$ satisfying requirement \ref{item:R2}, that is, for every  $0\leq s < s_{n+1}$, $0\leq u <s_n$ we have 
		\begin{equation}\label{eq:h2unif}
			\abs{\Meng{0\leq i <k_n}{\h_{n+1,2}(\Delta^{i,s}_{k_nq_n,s_{n+1}})\subseteq \Delta^{0,u}_{q_n,s_n}}}=\frac{k_n}{s_n}.
		\end{equation}
	    Then for every $\varepsilon>0$ there is an $1/q_n$-equivariant area-preserving $C^{\infty}$ diffeomorphism $h^{(\mathfrak{s})}_{n+1,2}$ that is equal to the identity in a neighborhood of the boundary and $\varepsilon$-approximates $h_{n+1,2}$.
	\end{lem} 

\begin{proof}
	By our assumption \eqref{eq:h2unif}, the map $h_{n+1,2}$ is a well-defined partition permutation satisfying $h_{n+1,2}(\Delta^{0,0}_{q_n,1})=\Delta^{0,0}_{q_n,1}$ and $h_{n+1,2}\circ R_{1/q_n} =R_{1/q_n}\circ h_{n+1,2}$. Hence, we can apply Lemma \ref{lem:FWreal} to conclude the proof of Lemma \ref{lem:h2real}.
\end{proof}
	
	To construct smooth approximations to the twisting map $h_{n+1,1}$ we use the following ``pseudo-rotations'' introduced in \cite{FS}.
	
	\begin{lem}[\cite{FS}, Lemma 5.3]\label{lem:Pseudorotate}
		For any $\delta < 1/2$ there exists a smooth area-preserving diffeomorphism $\varphi_{\delta}$ of $\mathbb{R}^2$, that is equal to the identity outside $[\delta,1-\delta]^2$ and rotating the square $[2\delta,1-2\delta]^2$ by $\pi/2$.
	\end{lem} 
	
	In the construction of smooth approximations to $h_{n+1,1}$ we use these pseudo-rotations to map horizontal stripes into vertical ones. The construction is visualised in Figure \ref{fig:fig2}.
	
	\begin{figure}
		\centering
		\includegraphics[width=\textwidth]{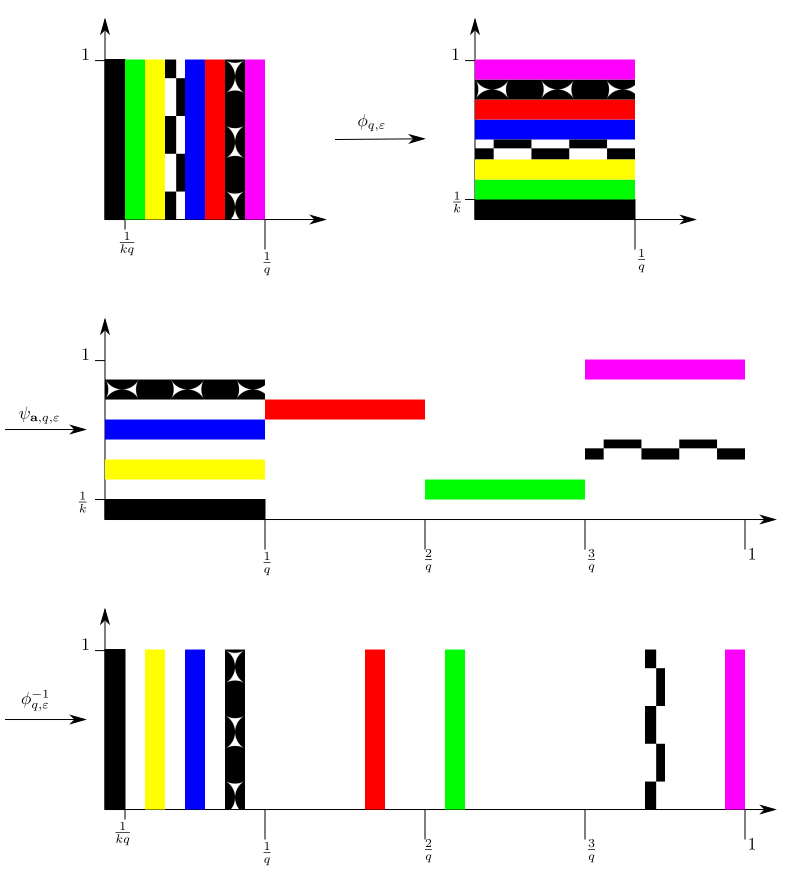}
		\caption{Visualization of the construction of $h_{n+1,1}$. Here, $k=8$ and $q=4$ (for illustration purposes; actual values will be much larger).}
		\label{fig:fig2}
	\end{figure}
	
	\begin{lem}\label{lem:h1real}
		Let $h_{n+1,1}$ be a partition permutation as defined in Subsection \ref{subsubsec:h1}. For every $\varepsilon>0$ there is an $1/q_n$-equivariant area-preserving $C^{\infty}$ diffeomorphism $h^{(\mathfrak{s})}_{n+1,1}$, which is equal to the identity in a neighborhood of the boundary and $\varepsilon$-approximates $h_{n+1,1}$.
	\end{lem} 
	
	\begin{proof}
		Let $D_{n}: [0,1/q_n] \times [0,1]$ be defined by $D_{n}(x,y)=(q_nx,y)$. We use this map and the pseudo-rotation $\varphi_{\delta}$ from Lemma \ref{lem:Pseudorotate} to define the area-preserving diffeomorphism
		\[
		\phi_{q_n,\delta}:[0,1/q_n] \times [0,1] \to [0,1/q_n] \times [0,1], \ \phi_{q_n,\delta} = D^{-1}_n \circ \varphi_{\delta} \circ D_n.
		\]
		Since $\phi_{q_n,\delta}$ coincides with the identity in a neighborhood of the boundary, we can extend it to a diffeomorphism $\phi_{q_n,\delta} \in \text{Diff}^{\,\infty}_{\,\lambda}(M)$ commuting with $R_{1/q_n}$. Furthermore, we note that
		\begin{equation} \label{eq:PhiMap}
			\begin{split}
			& \phi_{q_n, \frac{\varepsilon}{2s_{n+1}}}\left( \tilde{\Delta}^{i,j}_{k_nq_n, s_{n+1}, \varepsilon} \right) \\
			= & \Bigg[ \frac{1}{q_n}-\frac{j+1-\varepsilon}{q_n s_{n+1}},  \frac{1}{q_n}-\frac{j+\varepsilon}{q_n s_{n+1}}\Bigg] \times \Bigg[ \frac{i+\varepsilon}{k_n}, \frac{i+1-\varepsilon}{k_n}\Bigg] = \tilde{\Delta}^{s_{n+1}-j-1, i}_{s_{n+1}q_n , k_n, \varepsilon}
			\end{split}
		\end{equation}
	    for all $0\leq i < k_n$, $0\leq j < s_{n+1}$.
	    
	    In the next step, we build a smooth map that will introduce some horizontal translation depending on the height value $i$ of $\tilde{\Delta}^{s_{n+1}-j-1, i}_{s_{n+1}q_n , k_n, \varepsilon}$. For that purpose, let $\rho: \mathbb{R} \rightarrow \mathbb{R}$ be a smooth increasing function that equals $0$ for $x\leq \frac{1}{2}$ and $1$ for $x\geq 1$. Then we define the map $\tilde{\psi}_{{\bf a}_n,q_n,\varepsilon}: \left[0,1\right]\rightarrow\mathbb{R}$ by 
	    \begin{equation*}
	    	\tilde{\psi}_{{\bf a}_n,q_n,\varepsilon}\left(y\right)=\sum^{k_n-1}_{i=0} \frac{a_n\left( \lfloor \frac{i}{C_n}\rfloor\right)}{q_n}\cdot \left(\rho \left(\frac{k_n \cdot y}{\varepsilon}-\frac{i}{\varepsilon} \right) - \rho \left(\frac{k_n \cdot y}{\varepsilon}-\frac{i+1}{\varepsilon}+\frac{3}{2} \right) \right),
	    \end{equation*}
        where $a_n(\cdot)$ are the numbers from \eqref{eq:an} defined in the construction of $h_{n+1,1}$. Note that $\tilde{\psi}_{{\bf a}_n,q_n,\varepsilon}$ coincides with the identity in a neighborhood of the boundary and for every $0 \leq i < k_n$ we have 
        \begin{equation}\label{eq:psiTrans}
        \tilde{\psi}_{{\bf a}_n,q_n,\varepsilon}|_{\left[\frac{i+\varepsilon}{k_n}, \frac{i+1-\varepsilon}{k_n}\right]} \equiv \frac{a_n\left( \lfloor \frac{i}{C_n}\rfloor\right)}{q_n} \mod 1.
        \end{equation}
		Using this map $\tilde{\psi}_{{\bf a}_n,q_n,\varepsilon}$ we define the area-preserving diffeomorphism $\psi_{{\bf a}_n,q_n,\varepsilon}:M \to M$ by
		\begin{equation*}
			\psi_{{\bf a}_n,q_n,\varepsilon}\left(x,y\right)= \left(x+\tilde{\psi}_{{\bf a}_n,q_n,\varepsilon}\left(y\right) ,y\right).
		\end{equation*}
	    By \eqref{eq:psiTrans}, we have for every $0 \leq i < k_n$ that
	    \begin{equation}\label{eq:PsiMap}
	    \psi_{{\bf a}_n,q_n,\varepsilon} \left(\tilde{\Delta}^{j, i}_{s_{n+1}q_n , k_n, \varepsilon}\right) =\tilde{\Delta}^{a_n\left( \lfloor \frac{i}{C_n}\rfloor\right)s_{n+1}+j, i}_{s_{n+1}q_n , k_n, \varepsilon}
	\end{equation}
	    for all $0\leq j <s_{n+1}q_n$.    
	    Finally, we define the area-preserving smooth diffeomorphism $h^{(\mathfrak{s})}_{n+1,1}:M \to M$ by 
	    \[
	    h^{(\mathfrak{s})}_{n+1,1} = \phi^{-1}_{q_n, \frac{\varepsilon}{2s_{n+1}}} \circ \psi_{{\bf a}_n,q_n,\varepsilon} \circ \phi_{q_n, \frac{\varepsilon}{2s_{n+1}}}. 
	    \]
	    It is $1/q_n$-equivariant since all composed maps commute with $R_{1/q_n}$. Using equations \eqref{eq:PhiMap} and \eqref{eq:PsiMap} we also conclude that
	    \[
	    h^{(\mathfrak{s})}_{n+1,1}|_{\tilde{\Delta}^{i,j}_{k_nq_n, s_{n+1}, \varepsilon}} = h_{n+1,1}|_{\tilde{\Delta}^{i,j}_{k_nq_n, s_{n+1}, \varepsilon}}
	    \]
	    for all $0\leq i < k_nq_n$, $0\leq j < s_{n+1}$, that is, $h^{(\mathfrak{s})}_{n+1,1}$ $\varepsilon$-approximates $h_{n+1,1}$.
	\end{proof}

We are ready to prove Proposition \ref{prop:hReal}.
	  
	  \begin{proof}[Proof of Proposition \ref{prop:hReal}]
	  	Combining Lemmas \ref{lem:h1real} and \ref{lem:h2real}, we obtain $h^{(\mathfrak{s})}_{n+1}\in \text{Diff}^{\,\infty}_{\,\lambda}(M)$ approximating $h_{n+1}=h_{n+1,2}\circ h_{n+1,1}$ in the sense of Proposition \ref{prop:hReal}.
	  \end{proof}

	\subsubsection{Smooth AbC method}
	After approximating the partition permutation $h_{n+1}$ by smooth area-preserving diffeomorphisms that commute with $R_{1/q_n}$, we show that we can realize any weakly mixing transformation built by the abstract AbC method from Sections \ref{subsec:abstract} and \ref{subsec:constr} as an area-preserving $C^{\infty}$-diffeomorphism provided that the sequence $(l_n)_{n\in \N}$ grows sufficiently fast. This result is the counterpart of \cite[Theorem 38]{FW1}. 
	
	\begin{prop} \label{prop:SmoothReal}
		Let $(\varepsilon_n)_{n\in \N}$ be a summable sequence of positive reals satisfying
		\begin{equation} \label{eq:EpsDecrease}
			\sum_{m>n}\varepsilon_m < \frac{\varepsilon_n}{4} \ \text{ for every } n \in \mathbb{N}.
		\end{equation}
		Suppose $T:M\to M$ is a MPT built by the abstract AbC method from Sections \ref{subsec:abstract} and \ref{subsec:constr} satisfying requirements \ref{item:R1}, \ref{item:R2}, \ref{item:R3} and using parameter sequences $(k_n)_{n\in \N}$ and $(l_n)_{n\in \N}$. If $(l_n)_{n\in \N}$ grows fast enough (see condition \eqref{eq:fast enough} and Remark \ref{rem:FastEnough}), then there exists a  sequence of smooth AbC diffeomorphisms $(T^{(\mathfrak{s})}_n)_{n\in \N}$ that satisfies
		\begin{equation} \label{eq:ClosenesDiffeos}
			d_{\infty}(T^{(\mathfrak{s})}_n,T^{(\mathfrak{s})}_{n+1})<\frac{\varepsilon_n}{4} \ \text{ for every } n \in \mathbb{N}
		\end{equation}
		and converges to a diffeomorphism $T^{(\mathfrak{s})}\in\text{Diff }^{\infty}(M,\lambda)$ which is measure-theoretically isomorphic to $T$.
	\end{prop}
	
	\begin{proof}
		Let $\{h_n\}_{n\in \N}$, $\{H_n=h_1\circ\ldots\circ h_n\}_{n\in \N}$, and $\{T_n=H_n\circ R_{\alpha_n}\circ H_n^{-1}\}_{n\in \N}$ be a sequence of measure-preserving transformations constructed using parameters $\{k_n\}_{n\in \N}$, $\{l_n\}_{n\in \N}$, and $\{s_n\}_{n\in \N}$ via our abstract AbC method described in Sections \ref{subsec:abstract} and \ref{subsec:constr}. Let $T$ be the limit of the $\{T_n\}_{n=1}^\infty$ in the weak topology.
		
		Using Proposition \ref{prop:hReal} we construct diffeomorphisms $h_n^{(\mathfrak{s})}\in\text{Diff }^\infty_\lambda(M)$ that satisfy $h_n^{(\mathfrak{s})}\circ R_{1/q_{n-1}}= R_{1/q_{n-1}} \circ h_n^{(\mathfrak{s})}$ and coincide with $h_n$ on the set
		\begin{equation} \label{eq:Ln}
			L_n \coloneqq \bigcup^{k_{n-1}q_{n-1}-1}_{i=0} \bigcup^{s_n-1}_{j=0}\tilde{\Delta}^{i,j}_{k_{n-1}q_{n-1},s_n,\varepsilon_{n-1}}
		\end{equation}
		that is, the union of the ``inner kernels'' $\tilde{\Delta}^{i,j}_{k_{n-1}q_{n-1},s_n,\varepsilon_{n-1}}$ defined in \eqref{eq:InnerKernels}. We put 
		\begin{equation}
			H_n^{(\mathfrak{s})} \coloneqq h_1^{(\mathfrak{s})}\circ\ldots\circ h_n^{(\mathfrak{s})} \ \text{ and } \ T_n^{(\mathfrak{s})} \coloneqq H_n^{(\mathfrak{s})}\circ R_{\alpha_n}\circ (H_n^{(\mathfrak{s})})^{-1}.
		\end{equation}
		
		Exploiting the commutation relation we obtain for any $n \in \N$ that
		\begin{align*}
			T_{n+1}^{(\mathfrak{s})} = H_{n}^{(\mathfrak{s})}\circ R_{\alpha_{n}}\circ [h_{n+1}^{(\mathfrak{s})}\circ R_{1/(k_nl_nq_n^2)}\circ (h_{n+1}^{(\mathfrak{s})})^{-1}]\circ (H_n^{(\mathfrak{s})})^{-1}.
		\end{align*}
		Since the number $l_n$ is chosen last in the induction step, we can choose $l_n \in \Z^+$ large enough to obtain 
		\begin{equation} \label{eq:fast enough}
			d_{\lceil \frac{8}{\varepsilon_n}\rceil } \left(T_{n+1}^{(\mathfrak{s})},T_n^{(\mathfrak{s})}\right)<\frac{\varepsilon_n}{8}.
		\end{equation}  
		This yields
		\begin{align*}
			& d_{\infty} \left(T_{n+1}^{(\mathfrak{s})},T_n^{(\mathfrak{s})}\right) \\
			\leq & \sum^{\lceil \frac{8}{\varepsilon_n}\rceil}_{k=1} \frac{d_k\left(T_{n+1}^{(\mathfrak{s})},T_n^{(\mathfrak{s})}\right)}{2^k \cdot \left(1+d_k\left(T_{n+1}^{(\mathfrak{s})},T_n^{(\mathfrak{s})}\right)\right)} + \sum^{\infty}_{k=\lceil \frac{8}{\varepsilon_n}\rceil+1} \frac{d_k\left(T_{n+1}^{(\mathfrak{s})},T_n^{(\mathfrak{s})}\right)}{2^k \cdot \left(1+d_k\left(T_{n+1}^{(\mathfrak{s})},T_n^{(\mathfrak{s})}\right)\right)} \\
			< & \sum^{\lceil \frac{8}{\varepsilon_n}\rceil}_{k=1} \frac{\varepsilon_n/8}{2^k } + \sum^{\infty}_{k=\lceil \frac{8}{\varepsilon_n}\rceil+1} \frac{1}{2^k} \leq \frac{\varepsilon_n}{8} + 2^{-\lceil \frac{8}{\varepsilon_n}\rceil} \leq \frac{\varepsilon_n}{8} + \frac{\varepsilon_n}{8} = \frac{\varepsilon_n}{4}, 
		\end{align*} 
		that is, \eqref{eq:ClosenesDiffeos} holds. Furthermore, this implies for all $n,m \in \Z^+$ that 
		\[
		d_{\infty} \left(T_{n+m}^{(\mathfrak{s})},T_n^{(\mathfrak{s})}\right) < \sum^{n+m-1}_{k=n}\frac{\varepsilon_k}{4}.
		\]
		Since $(\varepsilon_n)_{n\in \N}$ is a summable sequence, the sequence $T_n^{(\mathfrak{s})}$ is a Cauchy sequence and, hence, converges to some $T^{(\mathfrak{s})}\in \text{Diff }^\infty_\lambda(M)$.
		
		In the next step, we prove that $T^{(\mathfrak{s})}$ is in fact measure-theoretically isomorphic to $T$. Our plan is to use Fact \ref{fact:isomorphism} for the proof. In the terminology of the lemma, we put $(\Omega,\mathcal{M},\mu)=(\Omega',\mathcal{M}',\mu')=(M,\mathcal{B},\lambda)$. We define $K_n:M \to M$ by 
		$K_n \coloneqq H_n^{(\mathfrak{s})}\circ H_n^{-1}$.
		From the definition it follows that $K_n$ is an isomorphism between $T_n$ and $T_n^{(\mathfrak{s})}$. We define the two sequences of partitions $\mathcal{P}_n:=\zeta_n=H_n(\xi_n)=H_n(\xi_{q_n,s_n})$ and $\mathcal{P}_n':=K_n(\zeta_n)=H_n^{(\mathfrak{s})}(\xi_n)$. Using Lemma \ref{lem:MPconv} we observe that $\{\mathcal{P}_n\}^{\infty}_{n=1}$ is generating. Next we need to show that $\mathcal{P}_n'$ is generating, too.
		
		We recall the definition of the set $L_{n}$ from equation \eqref{eq:Ln} and note that
		\begin{equation}
			\lambda(L_n) \geq 1-4\varepsilon_{n-1}.
		\end{equation}
		Then we consider the following sequence of sets: 
		\begin{align*}
			G_n\coloneqq L_n\cap \bigcap_{m=n+1}^{\infty} h_{n+1}^{(\mathfrak{s})}\circ\ldots\circ h_{m}^{(\mathfrak{s})}(L_m).
		\end{align*}
		Note that $\lambda(G_n)\nearrow 1$ by summability of $(\varepsilon_n)_{n\in \N}$ and Borel--Cantelli. By definition we have for every $y \in G_n$ and all $m>n$ that 
		\[(h_{n+1}^{(\mathfrak{s})}\circ\ldots\circ h_{m}^{(\mathfrak{s})})^{-1}(y) \in L_m \cap (h_{m}^{(\mathfrak{s})})^{-1}(L_{m-1}) \cap \ldots \cap (h_{m}^{(\mathfrak{s})})^{-1}\circ \dots \circ (h_{n+1}^{(\mathfrak{s})})^{-1}(L_n).
		\]
		Then we conclude for $c \in \xi_m$ and $x \in c \cap (h_{n+1}^{(\mathfrak{s})}\circ\ldots\circ h_{m}^{(\mathfrak{s})})^{-1}(G_n)$ that $h_{n+1}^{(\mathfrak{s})}\circ\ldots\circ h_{m}^{(\mathfrak{s})}(x)$ belongs to the same atom of $\xi_n$ as $h_{n+1}\circ\ldots\circ h_{m}(x)$ does.
		
		We pick $\delta>0$. There exists some $n_0$ such that $\lambda(G_m)>1-\frac{\d}{2}$ for all $m>n_0$. By the observation in the previous paragraph we get for any $m>n_0$ that 
		\[
		\lambda\big(\bigcup_{c\in\xi_m}  h_{n_0+1}^{(\mathfrak{s})}\circ\ldots\circ h_{m}^{(\mathfrak{s})}(c) \;\triangle\; h_{n_0+1}\circ\ldots\circ h_m(c)\big)<\delta/2.
		\]
		This implies
		\begin{align*}
			\lambda\big(\bigcup_{c\in\xi_m} (H_{n_0}^{(\mathfrak{s})}\circ h_{n_0+1}^{(\mathfrak{s})}\circ\ldots\circ h_{m}^{(\mathfrak{s})}(c) \;\triangle\; H_{n_0}^{(\mathfrak{s})}\circ h_{n_0+1}\circ\ldots\circ h_{m}(c)\big)<\delta/2,
		\end{align*}
	    that is,
	    \begin{align*}
			\lambda\big(\bigcup_{c\in\xi_m} (H_{m}^{(\mathfrak{s})} (c) \;\triangle\; H_{n_0}^{(\mathfrak{s})} (\Pi(c))\big)<\delta/2,
		\end{align*}
		where $\Pi\coloneqq h_{n_0+1}\circ\ldots\circ h_{m}$ is a permutation of $\xi_m$.		
		Now let $D\subset \T^2$ be any measurable set.  We put $D'=(H^{(\mathfrak{s})}_{n_0})^{-1}(D)$ and since $\{\xi_n\}^{\infty}_{n=1}$ is a generating sequence, there exists an $m>n_0$ and a collection $\mathcal{C}^{\prime}_m\subset \xi_m$ such that 
		\begin{align*}
			\lambda\big((\bigcup_{C\in\mathcal{C}^{\prime}_m} C)\triangle D'\big)<\d/2, \quad \text{ that is, } \quad \lambda\big((\bigcup_{C\in\mathcal{C}^{\prime}_m} H_{n_0}^{(\mathfrak{s})}(C))\triangle D\big)<\d/2.
		\end{align*}
		
		Combining the previous estimates, we obtain
		\begin{align*}
			& \lambda\big(\bigcup_{c\in\Pi^{-1}(\mathcal{C}^{\prime}_m)} H_{m}^{(\mathfrak{s})} (c) \;\triangle\; D\big) \\
			\leq & \; \lambda\big(\bigcup_{c\in\Pi^{-1}(\mathcal{C}^{\prime}_m)} H_{m}^{(\mathfrak{s})} (c) \;\triangle\; H_{n_0}^{(\mathfrak{s})} (\Pi(c))\big) + \lambda\big(\bigcup_{c\in\Pi^{-1}(\mathcal{C}^{\prime}_m)} H_{n_0}^{(\mathfrak{s})} (\Pi(c)) \;\triangle\; D\big)\\
			\leq & \; \lambda\big(\bigcup_{c\in\xi_m} H_{m}^{(\mathfrak{s})} (c) \;\triangle\; H_{n_0}^{(\mathfrak{s})} (\Pi(c))\big) + \lambda\big(\bigcup_{C\in\mathcal{C}^{\prime}_m} H_{n_0}^{(\mathfrak{s})} (C) \;\triangle\; D\big)
			\leq  \d/2 +\d/2 = \d.
		\end{align*}
		This shows that $\{\mathcal{P}_n'\}_{n=1}^{\infty}$ is a generating sequence of partitions.
		
		To verify the remaining assumption of Fact \ref{fact:isomorphism} we have to show that $$D_\lambda(K_{n+1}(\mathcal{P}_n),K_n(\mathcal{P}_n))<\varepsilon_n.$$ On the one hand, we compute
		\begin{align*}
			K_n(\mathcal{P}_n)= & \; K_n(\zeta_n)
			=  \; H_n^{(\mathfrak{s})}\circ H_n^{-1}(H_n(\xi_n))
			=  \; H_n^{(\mathfrak{s})}(\xi_n)
			=  \; H_{n}^{(\mathfrak{s})}\circ h_{n+1}\circ h_{n+1}^{-1}(\xi_n).
		\end{align*}
		On the other hand, 
		\begin{align*}
			K_{n+1}(\mathcal{P}_n)= & \; K_{n+1}(\zeta_n)
			= \; H_{n+1}^{(\mathfrak{s})}\circ H_{n+1}^{-1}(H_n(\xi_n))
			= \; H_{n}^{(\mathfrak{s})}\circ h_{n+1}^{(\mathfrak{s})}\circ h_{n+1}^{-1}(\xi_n).
		\end{align*}
		We put $\mathcal{Q}_n\coloneqq h_{n+1}^{-1}(\xi_n)$ and note that by construction $h_{n+1}^{(\mathfrak{s})}(\mathcal{Q}_n)$ approximates $h_{n+1}(\mathcal{Q}_n)$. This finishes the proof.
	\end{proof}

\begin{rem}\label{rem:FastEnough}
	For each choice of sequences $\{k_n\}_{m=1}^{n}$, $\{l_n\}_{m=1}^{n-1}$ and $\{s_n\}_{m=1}^{n+1}$ of natural numbers, we have finitely many permutations of $\xi_{k_nq_n , s_{n+1}}$ and hence finitely many choices of $h_{n+1}$. As seen in the proof of Proposition \ref{prop:SmoothReal}, for each such choice there exists a natural number $l_n \coloneqq l_n(h_{n+1}, \{k_m\}_{m=1}^{n},\{l_m\}_{m=1}^{n-1},\{s_m\}_{m=1}^{n+1}, \{\varepsilon_m\}_{m=1}^{n})$ such that for any $l\geq l_n$ we can choose $h_{n+1}^{(\mathfrak{s})}$ such that 
	$d_\infty(T_n^{(\mathfrak{s})},T_{n+1}^{(\mathfrak{s})})<\varepsilon_n/4$.
	To get an uniform estimate, we set 
	\begin{align*}
		l_n^*&=l_n^*(\{k_m\}_{m=1}^{n},\{l_m\}_{m=1}^{n-1},\{s_m\}_{m=1}^{n+1}, \{\varepsilon_m\}_{m=1}^{n})\\
		&\coloneqq \max_{h_{n+1}}l_n(h_{n+1}, \{k_m\}_{m=1}^{n},\{l_m\}_{m=1}^{n-1},\{s_m\}_{m=1}^{n+1}, \{\varepsilon_m\}_{m=1}^{n}).
	\end{align*}
\end{rem}

	With information on how to associate a construction sequence with an AbC transformation, we can now state the main theorem of this subsection (compare with \cite[Theorem 60]{FW1}).
	
	\begin{thm}\label{theo:TwistRealSmooth}
		Consider three sequences of natural numbers $\left(k_n\right)_{n\in\N},\left(l_n\right)_{n\in\N},\left(s_n\right)_{n\in\N}$ tending to infinity. Assume that 
		\begin{itemize}
			\item[(1)] $l_n$ grows sufficiently fast (see the previous Remark \ref{rem:FastEnough});
			\item[(2)] $k_n$ is of the form $k_n = 2^{n+2}q_nC_n$ for some $C_n \in \Z^+$;
			\item[(3)] $s_n$ divides  both $k_n$ and $s_{n+1}$.
		\end{itemize} 
		Let $\{\mathcal{W}_n\}_{n\in\N}$ be a twisting construction sequence over $\Sigma\cup\{b,e\}$ such that
		\begin{itemize}
			\item[(4)] $\mathcal{W}_0=\Sigma$, $|\mathcal{W}_{n+1}|=s_{n+1}$ for any $n\in \mathbb{N}$.
			\item[(5)] For each $w'\in\mathcal{W}_{n+1}$ and $w\in\mathcal{W}_n$, if $w'=\mathcal{C}^{\text{twist}}_{n}(w_0,\ldots,w_{k_n-1})$, then there are $k_n/s_n$ many $j$ with $w = w_j$. 
		\end{itemize}
		Furthermore, let $\mathbb{K}$ be the associated subshift and $\nu$ its unique non-atomic ergodic measure. Then there is $T^{(\mathfrak{s})}\in\text{Diff}^{\infty}_{\lambda}(M)$ such that the system $(M,\mathcal{B},\lambda,T^{(\mathfrak{s})})$ is isomorphic to $(\mathbb{K},\mathcal{B},\nu,sh)$.
	\end{thm}
	
	  \begin{proof}
	  	Let $P_{n+1}=\{\tilde{w}_0,\ldots,\tilde{w}_{s_{n+1}-1}\}\subset \mathcal{W}^{k_n}_n$ be the prewords of the twisting construction sequence (see Definition \ref{def:twistedConstSeq}). By assumption (5) for each $0\leq i< s_n$ the word $w_i$ occurs $k_n/s_n$ times in each $\tilde{w}_j$. These prewords describe the combinatorics of an untwisted permutation $h_{n+1,2}$ of $\xi_{k_nq_n,s_{n+1}} \cap \Delta^{0,0}_{q_n,1}$ as follows: if $w_{j_t}$ is the $t$-th $n$-word of $\tilde{w}_s$, then 
	  	\begin{equation*}
	  		h_{n+1,2}\left(\Delta^{t,s}_{k_nq_n,s_{n+1}} \right) \subset \Delta^{0,j_t}_{q_n,s_n}.
	  	\end{equation*}
  	  Afterwards, we extend this map to an invertible measure-preserving transformation $h_{n+1,2}$ commuting with $R_{1/q_n}$. This map is of the form as described in Section \ref{subsubsec:h2}. It satisfies requirement \ref{item:R2} by assumption (5) and requirement \ref{item:R3} since the pre-words in $P_{n+1}$ are distinct.
  	  
  	  Additionally, we take the conjugation map $h_{n+1,1}$ as in Section \ref{subsubsec:h1} and set $h_{n+1} = h_{n+1,2} \circ h_{n+1,1}$. The associated AbC construction $(T_n)_{n\in \N}$ also satisfies requirement \ref{item:R1} by assumption (3). Hence, Proposition \ref{prop:SmoothReal} guarantees that there is $T^{(\mathfrak{s})}\in\text{Diff}^{\infty}_{\lambda}(M)$ measure-theoretically isomorphic to the abstract AbC map $T$ which is isomorphic to the symbolic system $\mathbb{K}$ by Proposition \ref{prop:CodeAbC}.
	  \end{proof}
	
	We note that the sequence $\{P_n\}_{n\in\N}$ of prewords determines the conjugation maps $h_n$ in the AbC method which in turn determine $h^{(\mathfrak{s})}_n$ and a neighborhood in the smooth topology which the resulting AbC diffeomorphism belongs to. Therefore, different choices of $P_n$ give distant maps $h_n$ and, hence, distant diffeomorphisms $h_n^{(\mathfrak{s})}$ in the smooth topology. 
	
		\begin{lem}\label{lem:closeSmooth}
		Let $(\varepsilon_n)_{n\in \N}$ be a summable sequence of positive reals satisfying \eqref{eq:EpsDecrease}. Suppose $\{\mathcal{U}_n\}_{n\in\N}$ and $\{\mathcal{W}_n\}_{n\in\N}$ are two construction sequences for twisting systems and $N\in \Z^+$ such that $\mathcal{U}_n=\mathcal{W}_n$ for all $n\leq N$. If $S^{(\mathfrak{s})}$ and $T^{(\mathfrak{s})}$ are the smooth realizations of the twisting systems using the AbC method given in this paper, then
		\begin{equation} \label{eq:closeSmooth}
			d_{\infty}(S^{(\mathfrak{s})},T^{(\mathfrak{s})})<\varepsilon_N.
		\end{equation}  
	\end{lem}
	
	\begin{proof}
		Associated with the two construction sequences $\{\mathcal{U}_n\}_{n\in\N}$ and $\{\mathcal{W}_n\}_{n\in\N}$ there are sequences $\{k_n^\mathcal{U},l_n^\mathcal{U},h_n^\mathcal{U},s_n^\mathcal{U}\}_{n\in\N}$ and $\{k_n^\mathcal{W},l_n^\mathcal{W},h_n^\mathcal{W},s_n^\mathcal{W}\}_{n\in\N}$ determining the approximations $\{S^{(\mathfrak{s})}_n\}_{n\in\N}$ and $\{T^{(\mathfrak{s})}_n\}_{n\in\N}$ to the AbC diffeomorphisms $S^{(\mathfrak{s})}$ and $T^{(\mathfrak{s})}$. Since $\mathcal{U}_n=\mathcal{W}_n$ for all  $n\leq N$, these sequences have the property 
		$k_n^\mathcal{U}=k_n^\mathcal{W},l_n^\mathcal{U}=l_n^\mathcal{W}$ for all $n\leq N-1$ as well as $h_n^\mathcal{U}=h_n^\mathcal{W},s_n^\mathcal{U}=s_n^\mathcal{W}$ for all $n\leq N$. Thus, $S^{(\mathfrak{s})}_N=T^{(\mathfrak{s})}_N$. It follows from equations \eqref{eq:EpsDecrease} and \eqref{eq:ClosenesDiffeos} that
		\begin{align*}
			d_\infty(S^{(\mathfrak{s})}_N,S^{(\mathfrak{s})})\leq \sum^{\infty}_{n=N} d_{\infty}(S^{(\mathfrak{s})}_n,S^{(\mathfrak{s})}_{n+1}) < \sum^{\infty}_{n=N} \frac{\varepsilon_n}{4} <\frac{\varepsilon_N}{2}
		\end{align*}
		and, similarly, $d_\infty(T^{(\mathfrak{s})}_N,T^{(\mathfrak{s})}) <\varepsilon_N/2$. We conclude \eqref{eq:closeSmooth} by combining these two estimates together with the triangle inequality.
	\end{proof}

	\subsection{\label{subsec:analytic}Real-analytic realization}
	We now upgrade the realization results to the real-analytic category. Here, we have to restrict to $M=\mathbb{T}^2$.
	
	\subsubsection{Approximating partition permutations by real-analytic diffeomorphisms}
	Following \cite[section 4.2]{BK2} we find area-preserving real-analytic diffeomorphisms $h^{(\mathfrak{a})}_{n+1}$ that closely approximate the partition permutations $h_{n+1}$ from our abstract weakly mixing constructions in Section \ref{subsec:constr}. 
	\begin{prop}
		\label{prop:hRealAnalytic}
		Let $h_{n+1}=h_{n+1,2}\circ h_{n+1,1}$ be a partition permutation as defined in Section \ref{subsec:constr}. Then for any $\varepsilon>0$ there is a diffeomorphism $h^{(\mathfrak{a})}_{n+1}\in \text{Diff}_{\infty}^{\,\omega}(\mathbb{T}^2,\lambda)$ such that
		\begin{itemize}
			\item $h^{(\mathfrak{a})}_{n+1} \circ R_{1/q_n}=R_{1/q_n} \circ h^{(\mathfrak{a})}_{n+1}$,
			\item there is a set $L\subset \mathbb{T}^2$ with $\lambda(L)>1-\varepsilon$ satisfying that for all $0\leq i<k_nq_n$ and $0\leq j<s_{n+1}$ we have
			\[
			h^{(\mathfrak{a})}_{n+1}(x) \in h_{n+1}(\Delta^{i,j}_{k_nq_n,s_{n+1}}) \ \text{ for all } x\in L \cap \Delta^{i,j}_{k_nq_n,s_{n+1}}.
			\]
		\end{itemize}
	\end{prop}
	
	These approximations are the counterparts of the smooth realization results in Proposition \ref{prop:hReal} and of the real-analytic realization results for circular systems from \cite[section 4]{BK2}. As in \cite{BK2}, we use the concept of \emph{block-slide type of maps} introduced in \cite{Ba17} and their sufficiently precise approximation by area-preserving real-analytic diffeomorphisms. 	
	We recall that a \emph{step function} on the unit interval is a finite linear combination of indicator functions on intervals. We define the following two types of piecewise continuous maps on $\T^2$,
	\begin{align}
		& \mathfrak{h}_1:\T^2\to\T^2\qquad\text{defined by}\qquad\mathfrak{h}_1(x_1,x_2):=(x_1,\; x_2 + s_1(x_1)\mod 1),\\
		& \mathfrak{h}_2:\T^2\to\T^2\qquad\text{defined by}\qquad\mathfrak{h}_2(x_1,x_2):=(x_1 + s_2(x_2)\mod 1,\; x_2),
	\end{align}
	where $s_1$ and $s_2$ are step functions on the unit interval. Descriptively, the first map $\mathfrak{h}_1$ decomposes $\T^2$ into smaller rectangles using vertical lines and slides those rectangles vertically according to $s_1$, while the second map $\mathfrak{h}_2$ decomposes $\T^2$ into smaller rectangles using horizontal lines and slides those rectangles horizontally according to $s_2$. Any finite composition of maps of the above kind is called a \textit{block-slide} type of map on $\T^2$. 	
	By the following lemma, block-slide type of maps on $\T^2$ can be approximated well by real-analytic diffeomorphisms that can be extended to entire maps. This can be achieved because step functions can be approximated extremely well by real-analytic functions  (see e.g. \cite[Lemma 2.13]{Ba-Ku}). 
	
	\begin{lem}[\cite{Ba-Ku}, Proposition 2.22] \label{proposition approximation}
		Let $\mathfrak{h}:\T^2\to\T^2$ be a block-slide type of map which commutes with $R_{1/q}$ for some natural number $q$. Then for any $\varepsilon>0$ and $\delta>0$ there exists an area-preserving diffeomorphim $h^{(\mathfrak{a})}\in\text{Diff }^{\omega}_\infty(\T^2,\lambda)$ such that the following conditions are satisfied:
		\begin{enumerate}
			\item Proximity property: There exists a set $E\subset\T^2$ such that $\lambda(E)<\delta$ and $$\sup_{x\in\T^2\setminus E}\|h^{(\mathfrak{a})}(x)-\mathfrak{h}(x)\|<\varepsilon.$$ 
			\item Commutative property: $h^{(\mathfrak{a})}\circ R_{1/q}=R_{1/q}\circ h^{(\mathfrak{a})}$.
		\end{enumerate} 
	\end{lem}
	
	To approximate our partition permutations $h_{n+1}$ from Section \ref{subsec:constr} by real-analytic diffeomorphisms, we exploit that they are block-side type maps.
	
	\begin{lem}[\cite{Ba-Ku}, Theorem E] \label{permutation = block-slide}
		Let $k,q,s \in \N$ and $\Pi$ be a partition permutation of $\xi_{kq,s}$ of $\mathbb{T}^2$. Assume that $\Pi$ commutes with $R_{1/q}$. Then $\Pi$ is of block-slide type.
	\end{lem}
	
	
	\begin{proof}[Proof of Proposition \ref{prop:hRealAnalytic}]
		We apply Lemma \ref{permutation = block-slide} followed by Lemma \ref{proposition approximation}.  
	\end{proof}
	
	\subsubsection{Real-analytic AbC method}
	As an analogue of Proposition \ref{prop:SmoothReal}, we can realize any weakly mixing transformation built by the abstract AbC method from Sections \ref{subsec:abstract} and \ref{subsec:constr} as an area-preserving real-analytic diffeomorphism provided that the sequence $(l_n)_{n\in \N}$ grows sufficiently fast.
	\begin{prop} \label{theorem analytic AbC}
		Fix a number $\rho>0$. Let $(\varepsilon_n)_{n\in \N}$ be a summable sequence of positive reals satisfying
		$
			\sum_{m>n}\varepsilon_m < \varepsilon_n/4.
		$
	Suppose $T:\mathbb{T}^2\to\mathbb{T}^2$ is a measure-preserving transformation built by the abstract AbC method from Sections \ref{subsec:abstract} and \ref{subsec:constr} using parameter sequences $(k_n)_{n\in \N}$ and $(l_n)_{n\in \N}$. If $(l_n)_{n\in \N}$ grows fast enough, then there exists a 
	sequence of real-analytic AbC diffeomorphisms $(T^{(\mathfrak{a})}_n)_{n\in \N}$ that satisfies
	$d_{\rho}(T^{(\mathfrak{a})}_n,T^{(\mathfrak{a})}_{n+1})<\varepsilon_n/4$  for every $n \in \mathbb{N}$
	and converges to a diffeomorphism $T^{(\mathfrak{a})}\in \text{Diff}_{\rho}^{\,\omega}(\mathbb{T}^2,\lambda)$  which is measure-theoretically isomorphic to $T$.
	\end{prop} 

\begin{proof}
	Using Proposition \ref{prop:hRealAnalytic}, the proof follows along the lines of Proposition \ref{prop:SmoothReal} and \cite[Theorem 4.17]{BK2}.
\end{proof}

\begin{rem}\label{rem:analyticCounterpart}
	For any fixed  $\rho>0$ we can use this Proposition~\ref{theorem analytic AbC} to find real-analytic counterparts of the realization result for twisted symbolic systems in Theorem~\ref{theo:TwistRealSmooth} and the proximity result in Lemma~\ref{lem:closeSmooth}.
\end{rem}

	\section{\label{sec:reduction}Building the reduction}
	
	In this section we build the continuous reduction $\Phi:\mathcal{T}\kern-.5mm rees\to \text{Diff}^{\,\infty}_{\,\lambda}(M)$ that will satisfy the properties required in Theorem~\ref{thm:criterion}. For that purpose, we start by constructing strongly uniform and uniquely readable odometer-based construction sequences $\left(\mathtt{W}_{n}\left(\mathcal{T}\right)\right)_{n\in\mathbb{N}}$ similarly to the ones in \cite{GK3}. These constructions also specify and use equivalence relations $\mathcal{Q}_{s}^{n}(\mathcal{T})$ on the collections $\mathtt{W}_{n}(\mathcal{T})$ of $n$-words and group actions on the equivalence classes in $\mathtt{W}_{n}(\mathcal{T})/\mathcal{Q}_{s}^{n}(\mathcal{T})$ as in \cite{FRW}. Then odometer-based $(n+1)$-words are constructed by substituting Feldman patterns of finer equivalence classes of $n$-words into Feldman patterns of coarser classes. We collect important properties of these systems in Section \ref{subsec:spec} and describe such a substitution step in detail in Section \ref{subsec:Substitution}. Here, we point out small modifications to the substitution step from \cite{GK3} in order to verify the weak mixing property using our criterion from Proposition \ref{prop:WM2}.	
	We continue the inductive construction process by applying the twisting operator under some growth condition on the parameter sequence $(l_n)_{n\in \N}$ that will allow the smooth realization of the associated twisted systems according to Theorem \ref{theo:TwistRealSmooth}. We present the details of this construction process to get $\Phi:\mathcal{T}\kern-.5mm rees\to \text{Diff}^{\,\infty}_{\,\lambda}(M)$ in Section \ref{subsec:Construction}. In the next Section \ref{sec:Proof} we finally verify that $\Phi$ satisfies the properties stated in Theorem~\ref{thm:criterion}.
	
	\subsection{\label{subsec:spec}Specifications}
	Slightly modifying the constructions in \cite{GK3} we will construct for each $\mathcal{T}\in\mathcal{T}\kern-.5mm rees$ an odometer-based construction sequence $\Meng{\mathtt{W}_{n}\left(\mathcal{T}\right)}{\sigma_n \in \mathcal{T}}$ over the basic alphabet $\Sigma=\{1,\dots,2^{12}\}$, where for each $n\in\mathbb{N}$ with $\sigma_{n}\in\mathcal{T}$ the set of words $\mathtt{W}_{n}=\mathtt{W}_{n}(\mathcal{T})$ depends only on $\mathcal{T}\cap\left\{ \sigma_{m}:m\leq n\right\} $. 	
	The structure of the tree $\mathcal{T}\subset\mathbb{N}^{\mathbb{N}}$ is also used to build a sequence of groups $G_s(\mathcal{T})$: We define $G_{0}(\mathcal{T})$ to be the trivial group and assign
	to each level $s>0$ a so-called \emph{group of involutions} 
	\[
	G_s(\mathcal{T})=\sum_{\tau\in\mathcal{T},\,lh(\tau)=s}\left(\mathbb{Z}_{2}\right)_{\tau}.
	\]
	We have a well-defined notion of \emph{parity}
	for elements in such a group of involutions: an element is called \emph{even} if
	it can be written as the sum of an even number of generators.
	Otherwise, it is called \emph{odd}. 
	
	For levels $0<s<t$ of $\mathcal{T}$ we have a canonical homomorphism
	$\rho_{t,s}:G_{t}\left(\mathcal{T}\right)\to G_{s}\left(\mathcal{T}\right)$
	that sends a generator $\tau$ of $G_{t}\left(\mathcal{T}\right)$
	to the unique generator $\sigma$ of $G_{s}\left(\mathcal{T}\right)$
	that is an initial segment of $\tau$. The map $\rho_{t,0}$ is the trivial homomorphism $\rho_{t,0}:G_t(\mathcal{T})\to G_0(\mathcal{T})=\{0\}.$ We denote the inverse limit
	of $\left\langle G_{s}\left(\mathcal{T}\right),\rho_{t,s}:s<t\right\rangle $
	by $G_{\infty}\left(\mathcal{T}\right)$ and we let $\rho_{s}:G_{\infty}\left(\mathcal{T}\right)\to G_{s}\left(\mathcal{T}\right)$
	be the projection map.	
	Since there is a one-to-one correspondence between the infinite branches
	of $\mathcal{T}$ and infinite sequences $\left(g_{s}\right)_{s\in\mathbb{Z}^+}$
	of generators $g_{s}\in G_{s}\left(\mathcal{T}\right)$ with $\rho_{t,s}\left(g_{t}\right)=g_{s}$
	for $t>s>0$, we obtain the following characterization.
	\begin{fact} \label{fact:oddElement}
		Let $\mathcal{T}\subset\mathbb{N}^{\mathbb{N}}$ be a tree. Then $G_{\infty}\left(\mathcal{T}\right)$
		has a nonidentity element of odd parity if and only if $\mathcal{T}$
		has an infinite branch. 
	\end{fact}
	
	During the construction one uses the following finite approximations:
	For every $n\in \N$ we let $G_{0}^{n}\left(\mathcal{T}\right)$ be the trivial group and
	for $s>0$ we let 
	\[
	G_{s}^{n}\left(\mathcal{T}\right)=\sum\left(\mathbb{Z}_{2}\right)_{\tau}\text{ where the sum is taken over }\tau\in\mathcal{T}\cap\left\{ \sigma_{m}:m\leq n\right\} ,\,lh(\tau)=s.
	\]
	We also introduce the finite approximations $\rho_{t,s}^{(n)}:G_{t}^{n}(\mathcal{T})\to G_{s}^{n}(\mathcal{T})$
	to the canonical homomorphisms.	
	In the following, we simplify notation by enumerating $\Meng{\mathtt{W}_{n}\left(\mathcal{T}\right)}{\sigma_n \in \mathcal{T}}$ and $\Meng{G_{s}^{n}\left(\mathcal{T}\right)}{\sigma_n \in \mathcal{T}}$ as $\{\mathtt{W}_n\}_{n\in \mathbb{N}}$ and $\{G^n_s\}_{n\in \mathbb{N}}$, respectively.
	
	During the course of construction one also defines an increasing sequence of prime numbers $(\mathfrak{p}_n)_{n\in \N}$ satisfying
	$
		\sum_{n\in \N} \frac{1}{\mathfrak{p}_n} <\infty.
	$
	We now collect important properties of the odometer-based construction sequence $\{\mathtt{W}_n\}_{n\in \mathbb{N}}$. To start we set $\mathtt{W}_{0}=\Sigma$. 
	\begin{enumerate}[label=(E\arabic*)]
		\item\label{item:E1}  All words in $\mathtt{W}_{n}$ have the same length $h_{n}$ and
		the cardinality $|\mathtt{W}_{n}|$ is a power of $2$.
		\item\label{item:E2}  There are $f_{n},k_n\in\mathbb{Z}^{+}$ such that every word in $\mathtt{W}_{n+1}$ is built by concatenating $k_n$ words
		in $\mathtt{W}_{n}$ and such that
		every word in $\mathtt{W}_{n}$ occurs in each word of $\mathtt{W}_{n+1}$
		exactly $f_{n}$ times. The number $f_{n}$ is a product of $\mathfrak{p}_{n}^{2}$
		and powers of $2$. 
		\item\label{item:E3}  If $\mathtt{w}=\mathtt{w}_{1}\dots \mathtt{w}_{k_n}\in \mathtt{W}_{n+1}$ and $\mathtt{w}^{\prime}=\mathtt{w}_{1}^{\prime}\dots \mathtt{w}_{k_n}^{\prime}\in\mathtt{W}_{n+1}$, where $\mathtt{w}_{i},\mathtt{w}_{i}^{\prime}\in\mathtt{W}_{n}$, then for any $k\geq\lfloor\frac{k_n}{2}\rfloor$ and $1\leq i \leq k_n-k$, we have $\mathtt{w}_{i+1}\dots \mathtt{w}_{i+k}\neq \mathtt{w}_{1}^{\prime}\dots \mathtt{w}_{k}^{\prime}$. 
	\end{enumerate}
		In particular, these specifications say that $\{\mathtt{W}_n\}_{n\in \mathbb{N}}$
		is a uniquely readable and strongly uniform construction sequence
		for an odometer-based system. 
	
For each $s \leq s(n)$, there is an equivalence relation $\mathcal{Q}_{s}^{n}$ on $\mathtt{W}_{n}$ satisfying the following specifications. To start, we let $\mathcal{Q}_{0}^{0}$ be the equivalence relation
on $\mathtt{W}_{0}=\Sigma$ which has one equivalence class, that is,
any two elements of $\Sigma$ are equivalent. 
\begin{enumerate}[label=(Q\arabic*)]
	\setcounter{enumi}{3}
	\item\label{item:Q4} Suppose that $n=M(s)$ for some $s \in \Z^+$. There is a specific number $J_{s,n}\in\mathbb{Z}^{+}$
	such that $2J_{s,n}\mathfrak{p}^2_n$ divides $k_{n-1}$ and two words $\mathtt{w}=\mathtt{w}_{0}\dots \mathtt{w}_{k_{n-1}-1}\in \mathtt{W}_{n}$ and $\mathtt{w}^{\prime}=\mathtt{w}_{0}^{\prime}\dots \mathtt{w}_{k_{n-1}-1}^{\prime}\in\mathtt{W}_{n}$ are in the same $\mathcal{Q}_{s}^{n}$ class iff
	\[
	\mathtt{w}_i = \mathtt{w}^{\prime}_i \text{ for all $i$ with } \frac{k_{n-1}}{2\mathfrak{p}_{n}J_{s,n}} \leq i \mod \frac{k_{n-1}}{J_{s,n}} < \frac{k_{n-1}}{J_{s,n}}-\frac{k_{n-1}}{2\mathfrak{p}_{n}J_{s,n}}.
	\]
	\item\label{item:Q5}  For $n\geq M(s)+1$ we can consider words in $\mathtt{W}_{n}$
	as concatenations of words from $\mathtt{W}_{M(s)}$ and define $\mathcal{Q}_{s}^{n}$
	as the product equivalence relation of $\mathcal{Q}_{s}^{M(s)}$. 
	\item\label{item:Q6} $\mathcal{Q}_{s+1}^{n}$ refines $\mathcal{Q}_{s}^{n}$ and each
	$\mathcal{Q}_{s}^{n}$ class contains $2^{4e(n)}$ many $\mathcal{Q}_{s+1}^{n}$
	classes. 
\end{enumerate}
We write $Q_{s}^{n}$ for the number of equivalence classes in $\mathcal{Q}_{s}^{n}$
and enumerate the classes by $\left\{ c_{j}^{(n,s)}:j=1,\dots,Q_{s}^{n}\right\} $.
Occasionally, we will identify $\mathtt{W}_{n}/\mathcal{Q}_{s}^{n}$
with an alphabet denoted by $\left(\mathtt{W}_{n}/\mathcal{Q}_{s}^{n}\right)^{\ast}$
of $Q_{s}^{n}$ symbols $\left\{ 1,\dots,Q_{s}^{n}\right\} $. 

Each equivalence relation $\mathcal{Q}_{s}^{n}$ will induce an equivalence
relation on $rev(\mathtt{W}_{n})$, which we will also call $\mathcal{Q}_{s}^{n}$,
as follows: $rev(\mathtt{w}),rev(\mathtt{w}')\in rev(\mathtt{W}_{n})$ are equivalent
with respect to $\mathcal{Q}_{s}^{n}$ if and only if $\mathtt{w},\mathtt{w}'\in\mathtt{W}_{n}$
are equivalent with respect to $\mathcal{Q}_{s}^{n}$.

\begin{rem*}
	By \ref{item:Q5} we can view $\mathtt{W}_{n}/\mathcal{Q}_{s}^{n}$ as sequences
	of elements $\mathtt{W}_{M(s)}/\mathcal{Q}_{s}^{M(s)}$ and similarly
	for $rev(\mathtt{W}_{n})/\mathcal{Q}_{s}^{n}$. It allows us to regard elements in $\mathtt{W}_{n}/\mathcal{Q}_{s}^{n}$
	for $n\geq M(s)+1$ as sequences of symbols from the alphabet $\left(\mathtt{W}_{M(s)}/\mathcal{Q}_{s}^{M(s)}\right)^{\ast}$. In particular,
	it follows that $\mathcal{Q}_{0}^{n}$ is the equivalence relation
	on $\mathtt{W}_{n}$ which has one equivalence class. 
\end{rem*}

We now list specifications on actions by the groups of involutions $G^n_s$.
\begin{enumerate}[label=(A\arabic*)]
	\setcounter{enumi}{6}
	\item\label{item:A7}  $G_{s}^{n}$ acts freely on $\mathtt{W}_{n}/\mathcal{Q}_{s}^{n}$
	and the $G_{s}^{n}$ action is subordinate to the $G_{s-1}^{n}$ action
	on $\mathtt{W}_{n}/\mathcal{Q}_{s-1}^{n}$ via the canonical homomorphism
	$\rho_{s,s-1}^{(n)}:G_{s}^{n}\to G_{s-1}^{n}$.
	\item\label{item:A8} Suppose $n>M(s)$. We view $G_{s}^{n}=G_{s}^{n-1}\oplus H$.
	Then the action of $G_{s}^{n-1}$ on $\mathtt{W}_{n-1}/\mathcal{Q}_{s}^{n-1}$
	is extended to an action on $\mathtt{W}_{n}/\mathcal{Q}_{s}^{n}$
	by the skew diagonal action. 
\end{enumerate}
\begin{rem}
	\label{rem:Closed-under-skew} In particular, in the above situation
	with $M(s)<n$ both specifications together yield that $\mathtt{W}_{n}/\mathcal{Q}_{s}^{n}$
	is closed under the skew diagonal action by $G_{s}^{n-1}$. Clearly,
	this also holds if we view each element in $\mathtt{W}_{n}/\mathcal{Q}_{s}^{n}$
	as a sequence over the alphabet $\left(\mathtt{W}_{M(s)}/\mathcal{Q}_{s}^{M(s)}\right)^{\ast}$.
\end{rem}
	
	In addition to the odometer-based construction we will define a twisting coefficient sequence $(C_n,l_n)_{n\in \N}$ with $l_n\in \Z^+$ growing sufficiently fast such that
	\begin{equation}\label{eq:lStrict}
		\sum_{m>n}\frac{1}{l_m} < \frac{1}{l_n} \text{ for every } n\in \N
	\end{equation}
as well as an associated twisted construction sequence $\left(\mathcal{W}_n\right)_{n \in \N}$ and bijections $\kappa_{n}:\mathtt{W}_{n}\to\mathcal{W}_{n}$ by induction:
\begin{itemize}
	\item Let $\mathcal{W}_{0}=\Sigma$ and $\kappa_{0}$ be the identity map. 
	\item Suppose that $\mathtt{W}_{n+1}$, $\mathcal{W}_{n}$ and $\kappa_{n}$
	have already been defined. Then we define 
	\[
	\mathcal{W}_{n+1}=\left\{ \mathcal{C}^{\text{twist}}_{n}\left(\kappa_{n}\left(\mathtt{w}_{0}\right),\kappa_{n}\left(\mathtt{w}_{1}\right),\dots,\kappa_{n}\left(\mathtt{w}_{k_{n}-1}\right)\right)\::\:\mathtt{w}_{0}\mathtt{w}_{1}\dots \mathtt{w}_{k_{n}-1}\in\mathtt{W}_{n+1}\right\} 
	\]
	and the map $\kappa_{n+1}$ by setting 
	\[
	\kappa_{n+1}\left(\mathtt{w}_{0}\mathtt{w}_{1}\dots \mathtt{w}_{k_{n}-1}\right)=\mathcal{C}^{\text{twist}}_{n}\left(\kappa_{n}\left(\mathtt{w}_{0}\right),\kappa_{n}\left(\mathtt{w}_{1}\right),\dots,\kappa_{n}\left(\mathtt{w}_{k_{n}-1}\right)\right).
	\]
	In particular, the prewords are 
	\[
	P_{n+1}=\left\{ \kappa_{n}\left(\mathtt{w}_{0}\right)\kappa_{n}\left(\mathtt{w}_{1}\right)\dots \kappa_{n}\left(\mathtt{w}_{k_{n}-1}\right)\::\:\mathtt{w}_{0}\mathtt{w}_{1}\dots \mathtt{w}_{k_{n}-1}\in\mathtt{W}_{n+1}\right\} .
	\]
\end{itemize}
	
	\subsubsection{Transferring equivalence relations and actions}\label{subsubsec:propagate}
	We also transfer our equivalence relations and group actions to the twisted system. We proceed in an analogous manner to transferring equivalence relations and actions to circular systems in \cite[section 5.10]{FW3}.
	\begin{enumerate}[label=(P\arabic*)]
		\item\label{item:P1}  For all $n\in\mathbb{N}$, $\mathtt{w}_{1},\mathtt{w}_{2}\in\mathtt{W}_{n}$ we let
		$\left(\kappa_{n}(\mathtt{w}_{1}),\kappa_{n}(\mathtt{w}_{2})\right)$ be in the relation $(\mathcal{Q}_{s}^{n})^{\text{twist}}$
		if and only if $(\mathtt{w}_{1},\mathtt{w}_{2})\in\mathcal{Q}_{s}^{n}$. 
		\item\label{item:P2} Given
		$g\in G_{s}^{n}$, we let $g[\kappa_{n}(\mathtt{w}_{1})]_{s}=[\kappa_{n}(\mathtt{w}_{2})]_{s}$
		if and only if $g[\mathtt{w}_{1}]_{s}=[\mathtt{w}_{2}]_{s}$.
	\end{enumerate}
	
	In order to prepare the construction of isomorphisms between $\mathbb{K}^{\text{twist}}$ and $(\mathbb{K}^{\text{twist}})^{-1}$ in case of $\mathcal{T}$ having an infinite branch (see  Section \ref{subsec:Isom}), it proves useful to also give an intrinsic and inductive description of equivalence relations as well as group actions.	
	An inductive definition of $(\mathcal{Q}_{s}^{n})^{\text{twist}}$ is given as follows: 
	\begin{itemize}
		\item Define $(\mathcal{Q}_{0}^{n})^{\text{twist}}$ to have exactly one class in $\mathcal{W}_n$.
		\item For $s\in \Z^+$ and $\mathtt{w}_1,\mathtt{w}_2 \in \mathtt{W}_{M(s)}$ put $\left(\kappa_{M(s)}(\mathtt{w}_1),\kappa_{M(s)}(\mathtt{w}_2)\right) \in (\mathcal{Q}_{s}^{M(s)})^{\text{twist}}$ if and only if $(\mathtt{w}_1,\mathtt{w}_2) \in \mathcal{Q}_{s}^{M(s)}$
		\item Suppose $(\mathcal{Q}_{s}^{n})^{\text{twist}}$ on $\mathcal{W}_n$ is defined. Then we define $(\mathcal{Q}_{s}^{n+1})^{\text{twist}}$ on $\mathcal{W}_{n+1}$ by setting $\mathcal{C}^{\text{twist}}_n(w_0,\dots,w_{k_n-1})$ equivalent to $\mathcal{C}^{\text{twist}}_n(w^{\prime}_0,\dots,w^{\prime}_{k_n-1})$ if and only if  $w_i$ is $(\mathcal{Q}_{s}^{n})^{\text{twist}}$-equivalent to $w^{\prime}_i$ for all $i=0,\dots,k_n-1$.
	\end{itemize}

    To inductively define the skew-diagonal action of $G_s$ it suffices to specify it on the canonical generators of $G_s$:
    \begin{itemize}
    	\item For the canonical generator $g\in G^{M(s)}_s$ and $\mathtt{w}_1,\mathtt{w}_2 \in \mathtt{W}_{M(s)}$ we let $g[\kappa_{M(s)}(\mathtt{w}_{1})]_{s}=[\kappa_{M(s)}(\mathtt{w}_{2})]_{s}$ if and only if $g[\mathtt{w}_{1}]_{s}=[\mathtt{w}_{2}]_{s}$. 
    	\item Suppose $n\geq M(s)$ and we view $G^{n+1}_s = G^n_s \oplus H$. Then the action of $G^n_s$	on $\mathcal{W}_n/(\mathcal{Q}_{s}^{n})^{\text{twist}}$	is extended to an action on $\mathcal{W}_{n+1}/(\mathcal{Q}_{s}^{n+1})^{\text{twist}}$ by the \emph{twisted skew diagonal action}: For a generator $g \in G^n_s$ we set
    	\begin{equation}\label{eq:TwistGenerator}
    		g\mathcal{C}^{\text{twist}}_n\left([\kappa_n(\mathtt{w}_0)]_s,\dots,[\kappa_n(\mathtt{w}_{k_n-1})]_s\right) = \mathcal{C}^{\text{twist}}_n\left(g[\kappa_n(\mathtt{w}_{k_n-1})]_s,\dots,g[\kappa_n(\mathtt{w}_0)]_s\right).
    	\end{equation}
    If $H$ is non-trivial, then for the canonical generator $g\in H$ and $\mathtt{w}_{1},\mathtt{w}_{2}\in\mathtt{W}_{n+1}$ we let $g[\kappa_{n+1}(\mathtt{w}_{1})]_{s}=[\kappa_{n+1}(\mathtt{w}_{2})]_{s}$ if and only if $g[\mathtt{w}_{1}]_{s}=[\mathtt{w}_{2}]_{s}$.
    \end{itemize}

\begin{rem}\label{rem:A8Twist}
	Regarding equation \eqref{eq:TwistGenerator} we stress that $\mathtt{W}_{n+1}/\mathcal{Q}^{n+1}_s$ is closed under the skew-diagonal action by \ref{item:A8}. Thus, for $[\mathtt{w}_0]_s\dots [\mathtt{w}_{k_n-1}]_s \in \mathtt{W}_{n+1}/\mathcal{Q}^{n+1}_s$ we have that $g[\mathtt{w}_{k_n-1}]_s \dots g[\mathtt{w}_0]_s \in \mathtt{W}_{n+1}/\mathcal{Q}^{n+1}_s$ which implies $g[\kappa_n(\mathtt{w}_{k_n-1})]_s\ldots g[\kappa_n(\mathtt{w}_0)]_s \in P_{n+1}/\mathcal{Q}^{n+1}_s$ and, hence, $\mathcal{C}^{\text{twist}}_n\left(g[\kappa_n(\mathtt{w}_{k_n-1})]_s,\dots,g[\kappa_n(\mathtt{w}_0)]_s\right) \in \mathcal{W}_{n+1}/(\mathcal{Q}_{s}^{n+1})^{\text{twist}}$.
\end{rem}

Let $n>m$. By specification \ref{item:E2} we can write $\mathtt{w} \in \mathtt{W}_n \subset (\mathtt{W}_m)^{h_n/h_m}$ as $\mathtt{w}=\mathtt{w}_0\dots \mathtt{w}_{h_n/h_m-1}$ with $\mathtt{w}_i\in \mathtt{W}_m$. Then let $\mathcal{C}^{\text{twist}}_{m,n}$ denote the operator such that
$
\kappa_n(\mathtt{w}) = \mathcal{C}^{\text{twist}}_{m,n}\left(\kappa_m(\mathtt{w}_0),\kappa_m(\mathtt{w}_1),\dots , \kappa_m(\mathtt{w}_{h_n/h_m-1})\right).
$
In particular, $\mathcal{C}^{\text{twist}}_{m,m+1}=\mathcal{C}^{\text{twist}}_{m}$ with the twisting operator from Definition \ref{def:twist} and
\begin{equation*}
	\begin{split}
	\kappa_n(\mathtt{w}) & = \mathcal{C}^{\text{twist}}_{n-1}\left(\kappa_{n-1}(\mathtt{w}_0\dots \mathtt{w}_{\frac{h_{n-1}}{h_m}-1}),\dots , \kappa_{n-1}(\mathtt{w}_{\frac{h_n-h_{n-1}}{h_m}}\dots \mathtt{w}_{\frac{h_{n}}{h_m}-1}) \right) \\
	& =  \mathcal{C}^{\text{twist}}_{n-1}\Bigg( \mathcal{C}^{\text{twist}}_{m,n-1}\left(\kappa_m(\mathtt{w}_0),\dots,\kappa_m(\mathtt{w}_{\frac{h_{n-1}}{h_m}-1})\right),\dots,\\
	& \quad \quad \quad \quad \quad \quad \quad \quad \mathcal{C}^{\text{twist}}_{m,n-1}\left(\kappa_m(\mathtt{w}_{\frac{h_n-h_{n-1}}{h_m}}),\dots, \kappa_m(\mathtt{w}_{\frac{h_{n}}{h_m}-1})\right)\Bigg).
	\end{split}
\end{equation*} 
We define the operator $\widetilde{\mathcal{C}}^{\text{twist}}_{m,n}$ by interchanging the role of $b$ and $e$ in $\mathcal{C}^{\text{twist}}_{m,n}$.

\begin{rem}\label{rem:E2Twist}
	For every $w\in \mathcal{W}_n$ there is a unique $\mathtt{w} \in \mathtt{W}_n$ such that $w=\kappa_n(\mathtt{w})$. Since $\mathtt{w}$ can be written as $\mathtt{w}=\mathtt{w}_0\dots \mathtt{w}_{h_n/h_m-1}$ with $\mathtt{w}_i\in \mathtt{W}_m$ by specification \ref{item:E2}, we have
	$
	w=\mathcal{C}^{\text{twist}}_{m,n}\left(\kappa_m(\mathtt{w}_0),\kappa_m(\mathtt{w}_1),\dots , \kappa_m(\mathtt{w}_{h_n/h_m-1})\right).
	$ 
\end{rem}

\begin{lem}\label{lem:RevTwistInd}
	Let $n>m$. We write $\mathtt{w} \in \mathtt{W}_n \subset (\mathtt{W}_m)^{h_n/h_m}$ as $\mathtt{w}=\mathtt{w}_0\dots \mathtt{w}_{h_n/h_m-1}$ with $\mathtt{w}_i\in \mathtt{W}_m$. Then
	\begin{equation}\label{eq:RevTwistInd}
		rev\left(\mathcal{C}_{m,n}^{\text{twist}}\left(\kappa_m(\mathtt{w}_{0}),\dots,\kappa_m(\mathtt{w}_{h_n/h_m-1})\right) \right) 
		= \widetilde{\mathcal{C}}_{m,n}^{\text{twist}}\left(rev(\kappa_m(\mathtt{w}_{h_{n}/h_m-1})),\dots ,rev(\kappa_m(\mathtt{w}_0))\right).
	\end{equation}
\end{lem}

\begin{proof}
	The proof follows by induction for $n$. The case $n=m+1$ holds by Lemma \ref{lem:ReverseTwist}. Suppose now that \eqref{eq:RevTwistInd} holds for $n$. Then we calculate
	\begin{align*}
		& rev\left(\mathcal{C}_{m,n+1}^{\text{twist}}\left(\kappa_m(\mathtt{w}_{0}),\kappa_m(\mathtt{w}_{1}),\dots,\kappa_m(\mathtt{w}_{h_{n+1}/h_m-1})\right) \right) \\
		= & rev\Bigg(\mathcal{C}^{\text{twist}}_{n}\Big(\mathcal{C}^{\text{twist}}_{m,n}\left(\kappa_m(\mathtt{w}_0),\dots,\kappa_m(\mathtt{w}_{\frac{h_n}{h_m}-1})\right),\dots,\\
		& \quad \quad \quad \quad \quad \quad \quad \quad \mathcal{C}^{\text{twist}}_{m,n}\left(\kappa_m(\mathtt{w}_{\frac{h_{n+1}-h_{n}}{h_m}}),\dots, \kappa_m(\mathtt{w}_{\frac{h_{n+1}}{h_m}-1})\right)\Big)\Bigg) \\
		= & \widetilde{\mathcal{C}}^{\text{twist}}_{n}\Bigg(rev\left(\mathcal{C}^{\text{twist}}_{m,n}\left(\kappa_m(\mathtt{w}_{\frac{h_{n+1}-h_{n}}{h_m}}),\dots, \kappa_m(\mathtt{w}_{\frac{h_{n+1}}{h_m}-1})\right)\right),\dots,\\
		& \quad \quad \quad \quad \quad \quad \quad \quad rev\left(\mathcal{C}^{\text{twist}}_{m,n}\left(\kappa_m(\mathtt{w}_0),\dots,\kappa_m(\mathtt{w}_{\frac{h_n}{h_m}-1})\right)\right)\Bigg) \\
		= & \widetilde{\mathcal{C}}^{\text{twist}}_{n}\Bigg( \widetilde{\mathcal{C}}^{\text{twist}}_{m,n}\left(rev(\kappa_m(\mathtt{w}_{\frac{h_{n+1}}{h_m}-1})),\dots,rev(\kappa_m(\mathtt{w}_{\frac{h_{n+1}-h_{n}}{h_m}}))\right),\dots,\\
		& \quad \quad \quad \quad \quad \quad \quad \quad \widetilde{\mathcal{C}}^{\text{twist}}_{m,n}\left(rev(\kappa_m(\mathtt{w}_{\frac{h_n}{h_m}-1})),\dots,rev(\kappa_m(\mathtt{w}_0))\right)\Bigg) \\
		= & \widetilde{\mathcal{C}}_{m,n+1}^{\text{twist}}\left(rev(\kappa_m(\mathtt{w}_{h_{n+1}/h_m-1})),\dots , rev(\kappa_m(\mathtt{w}_1)),rev(\kappa_m(\mathtt{w}_0))\right)
	\end{align*}
where we used Lemma \ref{lem:ReverseTwist} and the induction assumption from \eqref{eq:RevTwistInd}.
\end{proof}

\begin{lem}\label{lem:groupTwistInd}
Let $s\in \N$ and $n>m\geq M(s)$. We write $\mathtt{w} \in \mathtt{W}_n \subset (\mathtt{W}_m)^{h_n/h_m}$ as $\mathtt{w}=\mathtt{w}_0\dots \mathtt{w}_{h_n/h_m-1}$ with $\mathtt{w}_i\in \mathtt{W}_m$. Then for a canonical generator $g\in G^m_s$ we have
\begin{equation}\label{eq:groupTwistInd}
	\begin{split}
	&g\mathcal{C}_{m,n}^{\text{twist}}\left([\kappa_m(\mathtt{w}_{0})]_s,[\kappa_m(\mathtt{w}_{1})]_s,\dots,[\kappa_m(\mathtt{w}_{h_n/h_m-1})]_s\right) \\
	=& \mathcal{C}_{m,n}^{\text{twist}}\left(g[\kappa_m(\mathtt{w}_{h_n/h_m-1})]_s,\dots,g[\kappa_m(\mathtt{w}_{1})]_s,g[\kappa_m(\mathtt{w}_{0})]_s\right)
	\end{split}
\end{equation} 	
\end{lem}

\begin{proof}
	We proceed similarly to the proof of Lemma \ref{lem:RevTwistInd} by induction in $n$. The case $n=m+1$ holds by definition of the twisted skew-diagonal action in \eqref{eq:TwistGenerator}. Suppose now that \eqref{eq:groupTwistInd} holds for $n$. Then we calculate that
	\begin{equation*}
		\begin{split}
		&g\mathcal{C}_{m,n+1}^{\text{twist}}\left([\kappa_m(\mathtt{w}_{0})]_s,[\kappa_m(\mathtt{w}_{1})]_s,\dots,[\kappa_m(\mathtt{w}_{h_{n+1}/h_m-1})]_s\right) \\
		=&g\Bigg(\mathcal{C}^{\text{twist}}_{n}\Big(\mathcal{C}^{\text{twist}}_{m,n}\left([\kappa_m(\mathtt{w}_0)]_s,\dots,[\kappa_m(\mathtt{w}_{\frac{h_n}{h_m}-1})]_s\right),\dots,\\
		&\quad \quad \quad \quad \quad \quad \quad \quad \mathcal{C}^{\text{twist}}_{m,n}\left([\kappa_m(\mathtt{w}_{\frac{h_{n+1}-h_{n}}{h_m}})]_s,\dots, [\kappa_m(\mathtt{w}_{\frac{h_{n+1}}{h_m}-1})]_s\right)\Big)\Bigg) \\
		=&\mathcal{C}^{\text{twist}}_{n}\Bigg(g\mathcal{C}^{\text{twist}}_{m,n}\left([\kappa_m(\mathtt{w}_{\frac{h_{n+1}-h_{n}}{h_m}})]_s,\dots, [\kappa_m(\mathtt{w}_{\frac{h_{n+1}}{h_m}-1})]_s\right),\dots,\\
		&\quad \quad \quad \quad \quad \quad \quad \quad g\mathcal{C}^{\text{twist}}_{m,n}\left([\kappa_m(\mathtt{w}_0)]_s,\dots,[\kappa_m(\mathtt{w}_{\frac{h_n}{h_m}-1})]_s\right)\Bigg) \\
		=&\mathcal{C}^{\text{twist}}_{n}\Bigg(\mathcal{C}^{\text{twist}}_{m,n}\left(g[\kappa_m(\mathtt{w}_{\frac{h_{n+1}}{h_m}-1})]_s,\dots,g[\kappa_m(\mathtt{w}_{\frac{h_{n+1}-h_{n}}{h_m}})]_s\right),\dots,\\
		&\quad \quad \quad \quad \mathcal{C}^{\text{twist}}_{m,n}\left(g[\kappa_m(\mathtt{w}_{\frac{h_n}{h_m}-1})]_s,\dots,g[\kappa_m(\mathtt{w}_0)]_s\right)\Bigg) \\
		=&\mathcal{C}_{m,n+1}^{\text{twist}}\left(g[\kappa_m(\mathtt{w}_{h_{n+1}/h_m-1})]_s,\dots,g[\kappa_m(\mathtt{w}_{1})]_s,g[\kappa_m(\mathtt{w}_{0})]_s\right),
		\end{split}
	\end{equation*}
where we used \eqref{eq:TwistGenerator} and the induction assumption from \eqref{eq:groupTwistInd}.
\end{proof}

\begin{rem}\label{rem:ClosedUnderTwistSkew}
	In continuation of Remark \ref{rem:A8Twist} we see that $\mathcal{W}_n/(\mathcal{Q}^n_s)^{\text{twist}}$ is closed under the twisted skew diagonal action.
\end{rem}

\subsubsection{Canonical factors of twisted systems}
In \cite[section 5]{FRW} and \cite[section 4.2.2]{GK3} the equivalence relations $\mathcal{Q}_s$ are used to define a canonical sequence of factors $\mathbb{K}_s$. The equivariant transfer of equivalence relations and group actions from the previous Subsection~\ref{subsubsec:propagate} gives factors $\mathbb{K}^{\text{twist}}_s$ of the associated twisted system $\mathbb{K}^{\text{twist}}$. In this section we describe these factors explicitly.

For each $s\geq1$ unique readability allows us to identify in a typical $x\in\mathbb{K}^{\text{twist}}$ the $M(s)$-blocks and spacer symbols $b,e$ introduced at a later stage of the construction, that is, they belong to $\bigcup_{i>M(s)}\partial_i(x)$. This gives a bi-infinite sequence of spacer symbols $b,e$ and classes $\Meng{ c_{j}^{(M(s),s)}}{j=1,\dots,Q_{s}^{M(s)}}$
in $(\mathcal{Q}_{s}^{M(s)})^{\text{twist}}$. In fact, we define a map $$\tilde{\pi}^{\text{twist}}_{s}:\mathbb{K}^{\text{twist}}\to\left\{ 1,\dots,Q_{s}^{M(s)},b,e\right\} =\left(\mathcal{W}_{M(s)}/\mathcal{Q}_{s}^{M(s)}\right)^{\ast}\cup \{b,e\}$$ as follows:
If there is a $M(s)$-block at position $0$ of $x\in\mathbb{K}^{\text{twist}}$, then $\tilde{\pi}^{\text{twist}}_s$ assigns to $x$ the letter $j$, where the word $w\in\mathcal{W}_{M(s)}$
on the $M(s)$-block of $x$ containing the position $0$ satisfies $[w]_{s}=c_{j}^{(M(s),s)}$. Otherwise, there is a spacer symbol $b$ or $e$ at position $0$ of $x$ and $\tilde{\pi}^{\text{twist}}_s$ assigns to $x$ the letter $b$ or $e$, respectively. Hereby, we define a shift-equivariant
map $\pi^{\text{twist}}_{s}:\mathbb{K}^{\text{twist}}\to\left\{ 1,\dots,Q_{s}^{M(s)},b,e\right\} ^{\mathbb{Z}}$
by letting  $\pi_{s}(x)=\left(\tilde{\pi}_{s}(sh^{k}(x))\right)_{k\in\mathbb{Z}}.$ We denote the image of $\pi^{\text{twist}}_s$ by $\mathbb{K}^{\text{twist}}_{s}$ which is a factor of $\mathbb{K}^{\text{twist}}$ by construction. There is an analogous map from $rev(\mathbb{K}^{\text{twist}})$ to $rev(\mathbb{K}^{\text{twist}}_s)$ that we also denote by $\pi^{\text{twist}}_s$.

Next, we describe a convenient base for the topology on $\mathbb{K}^{\text{twist}}_{s}$. For this purpose, we recall from Remark \ref{rem:E2Twist} that for $n\geq M(s)$ any word $w\in \mathcal{W}_n$ can be written as $w=\mathcal{C}^{\text{twist}}_{M(s),n}\left(w_0,w_1,\dots , w_{h_n/h_{M(s)}-1}\right)$ with $w_i \in \mathcal{W}_{M(s)}$. Thus, $[w]_s$ can be written as a concatenation of $\mathcal{Q}_{s}^{M(s)}$ classes and spacers. Accordingly, for $n\geq M(s)$, $w\in\mathcal{W}_{n}$ and $0\leq k<q_{n}$, we
let $\left\langle [w]_{s},k\right\rangle $ be the collection of $x\in\mathbb{K}^{\text{twist}}_{s}$
such that its position $0$ lies in an $n$-block, the position $0$ is at the $k$-th place in the $n$-block
$B$ of $x$ containing the position $0$, and if $v\in\left\{ 1,\dots,Q_{s}^{M(s)},b,e\right\} ^{q_{n}}$
is the word in $x$ at the block $B$, then $v$ is the sequence of
$\mathcal{Q}_{s}^{M(s)}$ classes and spacers given by $[w]_{s}$. Then the collection
of those $\left\langle [w]_{s},k\right\rangle $ for $n\geq M(s)$, $w\in\mathcal{W}_{n}$ and $0\leq k<q_{n}$ forms a basis for the
topology of $\mathbb{K}^{\text{twist}}_{s}$ consisting of clopen sets. 

\begin{rem} \label{rem:ConstrSeqKs}
In particular, for $m=M(s)$ a word $w\in\mathcal{W}_{m}$ gives a word of length
$q_{m}$ over our alphabet $\left(\mathcal{W}_{M(s)}/\mathcal{Q}_{s}^{M(s)}\right)^{\ast}$
that is the repetition of the same letter. We denote the collection
of these words by $\left(\mathcal{W}_{m}\right)_{s}^{\ast}$. Then
for $n>m=M(s)$ each word $w\in \mathcal{W}_{n}$ can be written as $\mathcal{C}^{\text{twist}}_{m,n}\left(w_0,\dots , w_{h_n/h_m-1}\right)$ with $w_i \in \mathcal{W}_{m}$
by Remark \ref{rem:E2Twist} and, thus, determines a sequence of spacer symbols $b,e$ and elements of $\left(\mathcal{W}_{m}\right)_{s}^{\ast}$. We let
$\left(\mathcal{W}_{n}\right)_{s}^{\ast}$ be the collection of words
over the alphabet $\left(\mathcal{W}_{M(s)}/\mathcal{Q}_{s}^{M(s)}\right)^{\ast}\cup\{b,e\}$
arising this way. Then the sequence $\left((\mathcal{W}_{n})_{s}^{\ast}\right)_{n\geq M(s)}$
gives a well-defined twisted construction sequence for $\mathbb{K}^{\text{twist}}_{s}$
over the alphabet $\left(\mathcal{W}_{M(s)}/\mathcal{Q}_{s}^{M(s)}\right)^{\ast}\cup\{b,e\}$.
\end{rem}

We also define the measure $\nu_{s}\coloneqq(\pi^{\text{twist}}_{s})^{\ast}\nu$ on
$\mathbb{K}^{\text{twist}}_{s}$. To be more explicit, with the aid of specifications
\ref{item:E2} and \ref{item:Q6} for the odometer-based system, the transfer specifications~\ref{item:P1} and~\ref{item:P2}, and the proportion of symbols in the boundary for a typical $x \in \mathbb{K}$ we can show that $$\nu_{s}(\left\langle [w]_{s},k\right\rangle)=\frac{1}{q_nQ_{s}^{n}}\cdot (1-\sum_{i\geq n}\frac{1}{l_i})$$ for any $w\in \mathcal{W}_{n}$ and $0\leq k<q_n$. 

Finally, we let $\mathcal{H}^{\text{twist}}_{s}$
be the shift-invariant sub-$\sigma$-algebra of $\mathcal{B}(\mathbb{K}^{\text{twist}})$
generated by the collection of $(\pi^{\text{twist}}_{s})^{-1}(B)$, where $B$ is a
basic open set in $\mathbb{K}^{\text{twist}}_{s}$. Then $\mathcal{H}^{\text{twist}}_{s}$
is the sub-$\sigma$-algebra determined by the factor map $\pi^{\text{twist}}_{s}$. 

Since the equivalence relation $\mathcal{Q}_{s+1}^{M(s+1)}$ refines
$\left(\mathcal{Q}_{s}^{M(s)}\right)^{h_{M(s+1)}/h_{M(s)}}$ by specifications
\ref{item:E2} and \ref{item:Q6}, we have $\mathcal{H}^{\text{twist}}_{s+1}\supseteq\mathcal{H}^{\text{twist}}_{s}$
and a continuous factor map $\pi^{\text{twist}}_{s+1,s}:\mathbb{K}^{\text{twist}}_{s+1}\to\mathbb{K}^{\text{twist}}_{s}$.
To express this one explicitly, we note that a $\mathbb{Z}$-sequence
of $(\mathcal{Q}_{s+1}^{M(s+1)})^{\text{twist}}$-classes and spacers determines a sequence of
$(\mathcal{Q}_{s}^{M(s)})^{\text{twist}}$-classes and spacers, because a $(\mathcal{Q}_{s+1}^{M(s+1)})^{\text{twist}}$
class is contained in a $(\mathcal{Q}_{s}^{M(s+1)})^{\text{twist}}$ class which can be expressed as 
\[
c^{(M(s+1),s)}_j =\mathcal{C}^{\text{twist}}_{M(s),M(s+1)} \left( c^{(M(s),s)}_{0},\dots , c^{(M(s),s)}_{h_{M(s+1)}/h_{M(s)}-1}\right)
\]
with an $h_{M(s+1)}/h_{M(s)}$-tuple of $(\mathcal{Q}_{s}^{M(s)})^{\text{twist}}$ classes $\left( c^{(M(s),s)}_{0},\dots , c^{(M(s),s)}_{h_{M(s+1)}/h_{M(s)}-1}\right)$. Analogously we define a map from $rev(\mathbb{K}^{\text{twist}}_{s+1})$ to $rev(\mathbb{K}^{\text{twist}}_{s})$ that we also denote by $\pi^{\text{twist}}_{s+1,s}$.

Let $n=M(s)$ for some $s \in \Z^+$. We define $G_{n} \subset \mathbb{K}^{\text{twist}}\setminus \bigcup_{m\geq n}\partial_m$ to be the collection of $x \in \mathbb{K}^{\text{twist}}$ such that there is an $n$-block at position $0$ of $x$ and if this principal $n$-word is given by $\mathcal{C}^{\text{twist}}_{n-1}(w_0,\dots , w_{k_{n-1}-1})$, then $x(0)$ belongs to an $(n-1)$-word $w_i$ with
\[
\frac{k_{n-1}}{2\mathfrak{p}_{n}J_{s,n}} \leq i \mod \frac{k_{n-1}}{J_{s,n}} < \frac{k_{n-1}}{J_{s,n}}-\frac{k_{n-1}}{2\mathfrak{p}_{n}J_{s,n}}.
\]
We note that
\[
\nu(G^c_n) \leq \frac{1}{\mathfrak{p}_n} + \sum_{m\geq n}\frac{1}{l_{m-1}}.
\]
Let $G$ be the collection of $x\in\mathbb{K}^{\text{twist}}$ such that for all large $s$,
if $n=M(s)$ then $x$ belongs to $G_{n}$. Then $G$ has measure one by
equation~\eqref{eq:lStrict}, $\sum_{n \in \N}1/\mathfrak{p}_n < \infty$, and the Borel-Cantelli Lemma.
In analogy with \cite[Proposition 23]{FRW} we can prove the following
	statement making use of specification \ref{item:Q4}.
	\begin{lem}
		\label{lem:algebra}$\mathcal{B}\left(\mathbb{K}^{\text{twist}}\right)$ is the smallest
		invariant $\sigma$-algebra that contains $\bigcup_{s}\mathcal{H}^{\text{twist}}_{s}$.
	\end{lem}

	Using the set $G$ from above, we collect the following properties as in Propositions 24 and 25 of
	\cite{FRW}.
	\begin{lem}
		\label{lem:subalgebra}
		\begin{enumerate}
			\item For all $x\neq y$ belonging to $G$, there is an open set $S\in\bigcup_{s}\mathcal{H}^{\text{twist}}_{s}$
			such that $x\in S$ and $y\notin S$.
			\item For all $s\geq1$, $\mathcal{H}^{\text{twist}}_{s}$ is a strict subalgebra of $\mathcal{H}^{\text{twist}}_{s+1}$.
		\end{enumerate}
\end{lem}

\subsection{\label{subsec:Substitution}A general substitution step}
	In this subsection we describe a step in our iteration of substitutions that we use in the following subsection. It is very similar to the general substitution step in \cite[section 7]{GK3} with some small modifications to satisfy the requirements in our criterion for weak mixing in Proposition \ref{prop:WM2}. For the reader's convenience we present the substitution step in detail and emphasize the modifications to \cite[section 7]{GK3}. As in \cite{GK3} the substitution will have the following initial data: 
	\begin{itemize}
		\item An alphabet $\Sigma$ and a collection of words $X\subset\Sigma^{\mathfrak{h}}$
		\item Equivalence relations $\mathcal{P}$ and $\mathcal{R}$ on $X$ with
		$\mathcal{R}$ refining $\mathcal{P}$
		\item Groups of involutions $G$ and $H$ with distinguished generators
		\item A homomorphism $\rho:H\to G$ that preserves the distinguished generators.
		We denote the range of $\rho$ by $G^{\prime}$ and its kernel by
		$H_{0}$ with cardinality $|H_{0}|=2^{t}$ for some $t\in\mathbb{N}$.
		\item A free $G$ action on $X/\mathcal{P}$ and a free $H$ action on $X/\mathcal{R}$
		such that the $H$ action is subordinate to the $G$ action via $\rho$. 
		\item There are $N$ different equivalence classes in $X/\mathcal{P}$ denoted
		by $[A_{i}]_{\mathcal{P}}$, $i=1,\dots,N$, where $N=2^{\nu+N'}$
		with $N',\nu\in\mathbb{N}$.
		\item Each equivalence class $[A_{i}]_{\mathcal{P}}$ contains $2^{4e}$
		elements of $X/\mathcal{R}$, where $e\in\mathbb{Z}^{+}$ with $e\geq \max(2,t)$.
		We subdivide these $\mathcal{R}$ classes contained in $[A_{i}]_{\mathcal{P}}$ into $2^t$ tuples 
		$
		\left(\left[A_{i,u2^{4e-t}+1}\right]_{\mathcal{R}},\dots,\left[A_{i,(u+1)2^{4e-t}}\right]_{\mathcal{R}}\right),
		$
		where $u\in\{0,\dots,2^t-1\}$, such that each tuple intersects each orbit
		of the $H_{0}$ action exactly once and the tuples are images of each
		other under the action by $H_{0}$. 
		\item For some $R\ge 2$ and some $\alpha\in\left(0,\frac{1}{8}\right)$
		we have 
		$
			\overline{f}(A,\bar{A})\geq\alpha
		$
		for any substantial substrings $A$ and $\bar{A}$ of at least $\frac{\mathfrak{h}}{R}$
		consecutive symbols in any representatives of two different $\mathcal{P}$-equivalence
		classes, that is, representatives of $\left[A_{i_{1},j_{1}}\right]_{\mathcal{R}}$
		and $\left[A_{i_{2},j_{2}}\right]_{\mathcal{R}}$ for $i_{1}\neq i_{2}$
		and any $j_{1},j_{2}\in\left\{ 1,\dots,2^{4e}\right\} $.
		\item For some $\beta\in(0,\alpha]$ we have 
		$
			\overline{f}(A,\bar{A})\geq\beta
		$
		for any substantial substrings $A$ and $\bar{A}$ of at least $\frac{\mathfrak{h}}{R}$
		consecutive symbols in any representatives of two different $\mathcal{R}$-equivalence
		classes, that is, representatives of $\left[A_{i,j_{1}}\right]_{\mathcal{R}}$
		and $\left[A_{i,j_{2}}\right]_{\mathcal{R}}$, respectively, for $j_{1}\neq j_{2}$.
	\end{itemize}
	Let $K,T_{2}\in\mathbb{Z}^{+}$ and $\tilde{R}\ge 2$ be given. Moreover, let an even $D \in \mathbb{Z}^+$ be given. 
	\begin{rem*}
		This parameter $D$ is introduced in addition to the parameters from \cite[section 7]{GK3}. In the applications of our substitution step in Section \ref{subsec:Construction}, we choose $D=2^{n+2}q_n$ decribing a division of newly constructed pre-words into segments as required for the application of the twisting operator (recall the condition $k_n=2^{n+2}q_n C_n$ from \eqref{eq:k}). This corresponds to the subdivision of the fundamental domain $\Delta^{0,0}_{q_n,1}$ into $2^{n+2}q_n$ vertical segments in our weakly mixing constructions, particularly in the definition of the numbers $a_n(\cdot)$ in the construction of map $h_{n+1,1}$. 
	\end{rem*}
	 
	 We also suppose that
	 there are numbers $M_{1},P,U_{1}\in\mathbb{Z}^{+}$, where $\nu\cdot(2M_{1}+3)\geq2(\nu+N')$
	 and $U_{1}$ is a multiple of $D2^{2t}$ such that 
	 \begin{equation}
	 	U_{1}\geq2\tilde{R}^{2}.\label{eq:U}
	 \end{equation}
	 Finally, let $M_{2}\in\mathbb{Z}^{+}$ with
	 \begin{equation}
	 	M_{2}\geq K\cdot P\cdot U_{1}\cdot2^{\nu\cdot(2M_{1}+3)}.\label{eq:M2}
	 \end{equation}
	 Hereby, we define the numbers
	 \begin{equation}
	 	T_{1}\coloneqq T_{2}\cdot2^{(4e-t)\cdot(2M_{2}+3)},\label{eq:T1}
	 \end{equation}
	 \begin{equation}
	 	U_{2}\coloneqq U_{1}\cdot2^{\nu\cdot(2M_{1}+3)},\label{eq:U2}
	 \end{equation}
	 and
	 $
	 	k=U_{1}\cdot T_{1}\cdot2^{\nu\cdot(2M_{1}+3)}=U_2 \cdot T_1.
	 $
	 Note that $k$ is a multiple of $N^2$ by our assumption $\nu\cdot(2M_{1}+3)\geq2(\nu+N')$.
	 
	 Suppose we have a collection $\Omega\subset\left(X/\mathcal{P}\right)^{k}$
	 of cardinality $|\Omega|=P$ that satisfies the following properties:
	 \begin{enumerate}[label=(B\arabic*)]
	 	\item\label{item:B1} $\Omega$ is closed under the skew diagonal action of $G$,
	 	\item\label{item:B2} each $\omega\in\Omega$ is a different concatenation of $U_{1}$ many
	 	different $\left(T_{1},2^{\nu},M_{1}\right)$-Feldman patterns  as described in Section \ref{subsec:Feldman}, each
	 	of which is constructed out of a tuple consisting of $2^{\nu}$ many
	 	$[A_{i}]_{\mathcal{P}}$, 
	 	\item\label{item:B3} each $\omega \in \Omega$ can be written as $\omega=[A_{k(1)}]^{T_12^{2t}}_{\mathcal{P}}\dots [A_{k(U_2/2^{2t})}]^{T_12^{2t}}_{\mathcal{P}}$, where $k(1),\dots , k(U_2/2^{2t}) \in \{1,\dots , N\}$. We require for every $\omega \in \Omega$, $i_1,i_2 \in \{1,\dots , N\}$, and $d\in \{0,1,\dots D-1\}$ that the set $K(i_1,i_2)$ defined by
	 	\begin{equation*}
	 		\Meng{1\leq \ell \leq \frac{U_2}{D2^{2t}}}{k(d\frac{U_2}{D2^{2t}}+\ell)=i_1, \, k((d+1)\frac{U_2}{D2^{2t}}+\ell \mod \frac{U_2}{2^{2t}})=i_2}
	 	\end{equation*}
	 	has cardinality $\frac{U_2}{N^2D2^{2t}}$. (At this point, we note that $U_2=U_{1}\cdot2^{\nu\cdot(2M_{1}+3)}$ is a multiple of $N^2D2^{2t}$ since $\nu\cdot(2M_{1}+3)\geq2(\nu+N')$ and $U_1$ is a multiple of $D2^{2t}$.)
	 	This implies for every $\omega \in \Omega$, $i_0 \in \{1,\dots , N\}$, and $d\in \{0,1,\dots D-1\}$ that
	 	\begin{equation}\label{eq:Assum3}
	 		r\left(i_0, \, k(d\frac{U_2}{D2^{2t}}+1)k(d\frac{U_2}{D2^{2t}}+2)\dots k((d+1)\frac{U_2}{D2^{2t}})\right) = \frac{1}{N}\cdot \frac{U_2}{D2^{2t}}.
	 	\end{equation}  
	 \end{enumerate}
 
 \begin{rem*}
 	Clearly, our assumption \ref{item:B3} implies that each $[A_{i}]_{\mathcal{P}}\in X/\mathcal{P}$ occurs exactly $\frac{k}{N}$
 	times in each $\omega\in\Omega$. Hence, this uniformity assumption in \cite{GK3} is satisfied. We ask for the stronger assumption \ref{item:B3} in order to produce words satisfying condition \ref{item:R4} for our weakly mixing constructions.
 \end{rem*}

Then we construct a collection $S\subset\left(X/\mathcal{R}\right)^{k}$
of substitution instances of $\Omega$: 
\begin{enumerate}
	\item We start by choosing a set $\Upsilon\subset\Omega$ that intersects
	each orbit of the action by the group $G^{\prime}$ exactly once.
	\item We construct a collection
	of $M_{2}$ many different $\left(T_{2},2^{4e-t},M_{2}\right)$-Feldman
	patterns, where the tuple of building blocks is to be determined in
	step~(6). Note that each such pattern is constructed as a concatenation
	of $T_{2}2^{(4e-t)\cdot(2M_{2}+3)}$ many building blocks in total
	which motivates the definition of the number $T_{1}$ from equation \eqref{eq:T1}.
	\item By assumption on $\Omega$ we can subdivide each element $r\in\Upsilon$
	as a concatenation of $U_{1}2^{\nu\cdot(2M_{1}+3)}=U_{2}$ strings
	of the form $[A_{i}]_{\mathcal{P}}^{T_{1}}$ and each $i\in\left\{ 1,\dots,N\right\} $
	occurs exactly $\frac{1}{N}U_{2}$ many times in this decomposition. 
	\item For each $r\in\Upsilon$ we choose $K$ different sequences of $U_{2}$
	concatenations of different $\left(T_{2},2^{4e-t},M_{2}\right)$-Feldman
	patterns in ascending order from our collection in the second step
	and enumerate the sequences by $j\in\left\{ 1,\dots,K\right\} $.
	\item This time, we define 
	\begin{equation*}
		\tilde{V} =\frac{1}{ND}U_2= \frac{1}{ND}U_{1}2^{\nu\cdot(2M_{1}+3)}
	\end{equation*}
and we introduce two different sequences of length $\tilde{V}$ that will determine the choice of tuple patterns in step (6) of the construction:
\begin{itemize}
	\item The sequence $\psi=\left(\psi_{1},\dots,\psi_{\tilde{V}}\right)$ with $\psi_{v}=u\in\{0,\dots,2^t-1\}$,
	where $u\equiv v\:\mod\:2^{t}.$ That is, the sequence $\psi$ cycles through the symbols in $\left\{ 0,\dots,2^{t}-1\right\}$.
	\item The sequence $\phi=\left(\phi_{1},\dots,\phi_{\tilde{V}}\right)$ with $\phi_{v}=u\in\{0,\dots,2^t-1\}$,
	where $$u\equiv \lfloor \frac{v}{2^t} \rfloor\:\mod\:2^{t}.$$ That is, the sequence $\phi$ cycles through repetitions of length $2^t$ of symbols in $\left\{ 0,\dots,2^{t}-1\right\}$.
\end{itemize}
	Since $U_{1}$ was chosen as a multiple of $D2^{2t}$ and $\frac{1}{N}2^{\nu\cdot(2M_{1}+3)}\in\mathbb{Z}$
	by the assumption $\nu\cdot(2M_{1}+3)\geq2(\nu+N')$, each symbol from
	$\left\{ 0,\dots,2^{t}-1\right\} $ occurs the same number $\frac{\tilde{V}}{2^{t}}$
	of times in the sequence $\psi$ and $\phi$, respectively.
	\item Let $r\in\Upsilon$ and $j\in\left\{ 1,\dots,K\right\} $, and write 
	\[
	r=\left[A_{i(1)}\right]_{\mathcal{P}}^{T_1}\cdots\left[A_{i(\ell)}\right]_{\mathcal{P}}^{T_1}\cdots
	\left[A_{i(U_2)}\right]_{\mathcal{P}}^{T_1},
	\]
	where $i(1),\dots,i(U_2)\in\{1,\dots,N\}$. By assumption \eqref{eq:Assum3}, we have for every $i_0 \in \{1,\dots , N\}$ and every $d \in \{0,1,\dots , D-1\}$ that
	\begin{equation*}
		r \left( i_0,\, i(d\frac{U_2}{D} +1)i(d\frac{U_2}{D} +2) \dots i((d+1)\frac{U_2}{D})\right) =\frac{1}{ND}U_2 = \tilde{V}.
	\end{equation*}
    Let $\ell \in \{ 1, \dots , U_2\}$. Then there is $d_{\ell} \in \{0,1,\dots , D-1\}$ such that $d_{\ell} \frac{U_2}{D}+1 \leq \ell \leq (d_{\ell}+1)\frac{U_2}{D}$. Suppose $i(\ell)=i_0$
	and this is the $m$th occurrence of $i_0$ in the sequence $i(d_{\ell} \frac{U_2}{D}+1),\dots,i(\ell)$. If $d_{\ell}$ is even, 
	then we let $u=\psi_m$. If $d_{\ell}$ is odd, 
	then we let $u=\phi_m$. Then we substitute a Feldman pattern built with the tuple 
	$
	\left(\left[A_{i_0,u2^{4e-t}+1}\right]_{\mathcal{R}},\dots,\left[A_{i_0,(u+1)2^{4e-t}}\right]_{\mathcal{R}}\right)
	$
	into $\left[A_{i(\ell)}\right]_{\mathcal{P}}^{T_1}.$ 
	The Feldman pattern that is used is the $\ell$th pattern among the $U_2$ 
	patterns previously chosen for the given $r$ and $j$. We follow this procedure for each $\ell=1,\dots,U_2$
	to obtain an element $s\in (X/R)^k$. Let $S$ be the collection of such $s\in (X/R)^k$
	obtained for all $r\in\Upsilon$ and $j\in\{1,\dots,K\}.$
	\end{enumerate}
	
	Using this collection $S\subset\left(X/\mathcal{R}\right)^{k}$ we
	define 
	$
	\Omega^{\prime}=HS.
	$
	
	\begin{rem*}
		Here, steps (5) and (6) are modifications from the corresponding steps in \cite{GK3}. The two different sequences $\psi$ and $\phi$ are used to satisfy assumption \ref{item:B3} for a next substitution step. We refer to the proof of part (3) of the subsequent Proposition \ref{prop:PropertiesCollection}.  In Remark \ref{rem:VerifyWM} we use this part (3) to verify requirement \ref{item:R4} in our weakly mixing constructions.
	\end{rem*}
	
	As in \cite[Proposition 42]{GK3} we collect some properties of the collection $\Omega^{\prime}$.  In addition to strong uniformity of $\mathcal{R}$-classes in elements of $\Omega^{\prime}$, part (3) implies that assumption \ref{item:B3} is satisfied for a next substitution step.
	
		\begin{prop}
		\label{prop:PropertiesCollection}This collection $\Omega^{\prime}\subset\left(X/\mathcal{R}\right)^{k}$
		satisfies the following properties.
		\begin{enumerate}
			\item $\Omega^{\prime}$ is closed under the skew diagonal action by $H$.
			\item For each element in $\omega\in\Omega$ there are $K\cdot|H_{0}|$
			many substitution instances in $\Omega^{\prime}$. 
			\item Let $U_3 \coloneqq U_2 2^{(4e-t)\cdot (2M_2+3)}$. Each element $\omega ^{\prime} \in \Omega^{\prime}$ can be written as $\omega^{\prime}= [A_{k(1)}]^{T_22^{8e-2t}}_{\mathcal{R}} \dots [A_{k(U_3/2^{8e-2t})}]^{T_22^{8e-2t}}_{\mathcal{R}}$. Then for every $i_1,i_2\in \{1,\dots , N\}$, $j_1,j_2 \in \{1,\dots , 2^{4e}\}$, $d \in \{0,1,\dots , D-1\}$ we have that the set
			\begin{equation*}
				\begin{split}
			K\left( (i_1,j_1),(i_2,j_2)\right) \coloneqq	& \Bigg\{1\leq \ell \leq \frac{U_3}{D2^{8e-2t}}\, : \, [A_{k(d\frac{U_3}{D2^{8e-2t}}+\ell)}]^{T_22^{8e-2t}}_{\mathcal{R}}=[A_{i_1,j_1}]^{T_22^{8e-2t}}_{\mathcal{R}}, \\
				& \ [A_{k((d+1)\frac{U_3}{D2^{8e-2t}}+\ell \mod \frac{U_3}{2^{8e-2t}})}]^{T_22^{8e-2t}}_{\mathcal{R}} = [A_{i_2,j_2}]^{T_22^{8e-2t}}_{\mathcal{R}}\Bigg\}
			    \end{split}
			\end{equation*}
			has cardinality $\frac{1}{N^22^{8e}}\cdot \frac{U_3}{D2^{8e-2t}}$.
		\end{enumerate}
	\end{prop}
	
	\begin{proof}
		The first two parts follow as in \cite[Proposition 42]{GK3}. To see the third part we use
		that in each $\omega=[A_{k(1)}]^{T_12^{2t}}_{\mathcal{P}}\dots [A_{k(U_2/2^{2t})}]^{T_12^{2t}}_{\mathcal{P}}\in\Omega$ we have for every $i_1,i_2\in \{1,\dots , N\}$ and $d \in \{0,1,\dots , D-1\}$ that the set $K(i_1,i_2) $ defined by
		\begin{equation*}
			\Meng{1\leq \ell \leq \frac{U_2}{D2^{2t}}}{k(d\frac{U_2}{D2^{2t}}+\ell)=i_1, \, k((d+1)\frac{U_2}{D2^{2t}}+\ell \mod \frac{U_2}{2^{2t}})=i_2}
		\end{equation*}
	    has cardinality $\frac{U_2}{DN^22^{2t}}$. By steps (4) and (6) we substitute $2^{2t}$ many different $\left(T_{2},2^{4e-t},M_{2}\right)$-Feldman
	    patterns into strings of the form $[A_{i}]_{\mathcal{P}}^{T_{1}2^{2t}}$. We recall from the definition of $\left(T_{2},2^{4e-t},M_{2}\right)$-Feldman
	    patterns that such a pattern is built by repeating single $\mathcal{R}$-classes some multiple of $T_22^{8e-2t}$ many times. Thus, $\omega ^{\prime} \in \Omega^{\prime}$ is of the form as described in part~(3) of Proposition \ref{prop:PropertiesCollection}.  
	    
	    For $\ell \in K(i_1,i_2)$ let $a_1(1)\dots a_1(2^{2t})$ and $a_2(1)\dots a_2(2^{2t})$ with $a_1(j),a_2(j) \in \{0,\dots , 2^t-1\}$ enumerate the $2^{2t}$ many
	     building tuples $\left(\left[A_{i_1,a_1+1}\right]_{\mathcal{R}},\dots,\left[A_{i_1,a_1+2^{4e-t}}\right]_{\mathcal{R}}\right)$ and $\left(\left[A_{i_2,a_2+1}\right]_{\mathcal{R}},\dots,\left[A_{i_2,a_2+2^{4e-t}}\right]_{\mathcal{R}}\right)$, respectively, that are substituted into the repetitions $[A_{k(d\frac{U_2}{D2^{2t}}+\ell)}]_{\mathcal{P}}^{T_{1}2^{2t}}$ and $[A_{k((d+1)\frac{U_2}{D2^{2t}}+\ell \mod \frac{U_2}{2^{2t}})}]_{\mathcal{P}}^{T_{1}2^{2t}}$, respectively. This choice of tuple patterns is determined by $\psi$ and $\phi$ in step (6). The construction of the sequences $\psi$ and $\phi$ in step (5) implies that for every pair of $a_1,a_2 \in \{0,2^t-1\}$ we have
        $r\left( a_1,a_2,\, a_1(1)\dots a_1(2^{2t}), \, a_2(1)\dots a_2(2^{2t})\right)=1.$ 
        Since the Feldman patterns substituted into the different $[A_{i(j)}]_{\mathcal{P}}^{T_{1}}$-instances are different from each other by step (4), we can apply part (3) of Lemma \ref{lem:FP} to obtain 
        \[
        \abs{K\left( (i_1,j_1),(i_2,j_2)\right)} = \abs{K(i_1,i_2)} \cdot \frac{2^{(4e-t)\cdot (2M_2+1)}}{2^{8e-2t}}=\frac{U_3}{D2^{8e-2t}N^22^{8e}}.
        \]
	\end{proof}

We end this subsection by pointing out that our modifications do not affect the estimates in \cite[Proposition 43]{GK3} on the $\overline{f}$ distance of elements
in $\Omega'$ that are equivalent with respect to the $\mathcal{P}$
product relation but are not $\mathcal{R}$-equivalent.
	
	\subsection{\label{subsec:Construction}The Construction Process}
	In this section, we describe the construction of our continuous reduction $\Phi:\mathcal{T}\kern-.5mm rees\to\text{Diff}^{\,\infty}_{\,\lambda}(M)$ for the proof of Theorem~\ref{thm:criterion}. For each $\mathcal{T}\in\mathcal{T}\kern-.5mm rees$ we
	build a construction sequence $\Meng{\mathcal{W}_{n}\left(\mathcal{T}\right)}{\sigma_n \in \mathcal{T}}$ based on collections of odometer-based words $\Meng{\mathtt{W}_{n}\left(\mathcal{T}\right)}{\sigma_n \in \mathcal{T}}$
	satisfying our specifications from Section \ref{subsec:spec} and bijections $\kappa_n: \mathtt{W}_{n}\left(\mathcal{T}\right) \to \mathcal{W}_{n}\left(\mathcal{T}\right)$. 
	To prove continuity of our map $\Phi:\mathcal{T}\kern-.5mm rees\to\text{Diff}^{\,\infty}_{\,\lambda}(M)$ in Lemma \ref{lem:ContinuityPhi}, this has to be done in such a way that $\mathcal{W}_{n}\left(\mathcal{T}\right)$
	is entirely determined by $\mathcal{T}\cap\left\{ \sigma_{m}:m\leq n\right\} $,
	that is, 
	\begin{equation}
		\begin{split} & \text{if }\mathcal{T}\cap\left\{ \sigma_{m}:m\leq n\right\} =\mathcal{T}^{\prime}\cap\left\{ \sigma_{m}:m\leq n\right\} ,\\
			& \text{then for all }m\leq n:\:\mathcal{W}_{m}\left(\mathcal{T}\right)=\mathcal{W}_{m}(\mathcal{T}^{\prime}).
		\end{split}
		\label{eq:ContinuityF}
	\end{equation}
    Therefore, we follow \cite{FRW} and \cite[section 8]{GK3} to organize our construction. To simplify notation we enumerate $\Meng{\mathcal{W}_{n}\left(\mathcal{T}\right)}{\sigma_n \in \mathcal{T}}$ as $\{\mathcal{W}_n\}_{n\in \mathbb{N}}$, that is, $(n + 1)$-words are built by concatenating $n$-words. Compared to the construction process in \cite[section 8]{GK3}, we use a slightly different order of choices of parameters $\mathfrak{p}_n$, $R_n$, $e(n)$, and $l_n$. We carefully describe their interdependencies and show that there are no circular dependencies. As in \cite[section 10.4]{GK3} we will also use a sequence $\left(R_{n}^{c}\right)_{n=1}^{\infty}$,
    where $R_{1}^{c}=R_{1}$ (with $R_{1}$ from the sequence $(R_{n})_{n\in\mathbb{N}}$ above)
    and $R_{n}^{c}=\lfloor\sqrt{l_{n-2}}\cdot k_{n-2}\cdot q_{n-2}^{2}\rfloor$
    for $n\geq2$. We note that for $n\geq2$, 
    \begin{equation}
    	\frac{q_{n}}{R_{n}^{c}}=\frac{k_{n-1}\cdot l_{n-1}\cdot\left(k_{n-2}\cdot l_{n-2}\cdot q_{n-2}^{2}\right)\cdot q_{n-1}}{\lfloor\sqrt{l_{n-2}}\cdot k_{n-2}\cdot q_{n-2}^{2}\rfloor}\geq\sqrt{l_{n-2}}\cdot k_{n-1}\cdot l_{n-1}\cdot q_{n-1}.\label{eq:Rcircular}
    \end{equation}
    Hence, for $n\geq2$ a substring of at least $q_{n}/R_{n}^{c}$ consecutive
    symbols in a twisted $n$-block contains at least $\sqrt{l_{n-2}}-1$
    complete $2$-subsections, which have length $k_{n-1}l_{n-1}q_{n-1}$
    (recall the notion of a $2$-subsection from Remark \ref{rem:Subsection}). When conducting $\overline{f}$ estimates on the twisting system, this allows us to ignore incomplete $2$-subsections
    at the ends of the substring.
    
    We describe how to construct $\mathtt{W}_{n+1}\left(\mathcal{T}\right)$, $\mathcal{W}_{n+1}\left(\mathcal{T}\right)$, $\mathcal{Q}_{s}^{n+1}\left(\mathcal{T}\right)$,
    and the action of $G_{s}^{n+1}\left(\mathcal{T}\right)$, which we abbreviate by $\mathtt{W}_{n+1}$, $\mathcal{W}_{n+1}$,
    $\mathcal{Q}_{s}^{n+1}$, and $G_{s}^{n+1}$, respectively. We also define a bijection $\kappa_{n+1}: \mathtt{W}_{n+1} \to \mathcal{W}_{n+1}$. If $n=0$ we have $\mathcal{W}_0=\mathtt{W}_{0}=\Sigma=\left\{ 1,\dots,2^{12}\right\} $ and $\kappa_0$ is the identity map. We also take a prime number $\mathfrak{p}_1> 2$ and an integer $R_1\geq 40\mathfrak{p}_1$.  If $n\geq 1$ our {\bf induction assumption} says that we have $\mathtt{W}_{n}$, $\mathcal{W}_{n}$,
    $\mathcal{Q}_{s}^{n}$,
    and $G_{s}^{n}$ satisfying our
    specifications. In particular, the odometer-based words in $\mathtt{W}_n$ have length $h_{n}$ and the words in $\mathcal{W}_n$ have length $q_n$. We also assume that there is a bijection $\kappa_{n}: \mathtt{W}_{n} \to \mathcal{W}_{n}$ and that there are $\frac{1}{8}>\alpha^{(t)}_{1,n} > \dots > \alpha^{(t)}_{s(n),n}>\beta_{n}^{(t)}$ such that the following estimates on $\overline{f}$ distances hold:
    \begin{itemize}
    	\item For every $s \in \{1,\dots, s(n)\}$ we assume that if $w,\overline{w}\in\mathcal{W}_{n}$ with
    	$[w]_{s}\neq[\overline{w}]_{s}$, then
    	$
    		\overline{f}\left(\mathcal{A},\overline{\mathcal{A}}\right)>\alpha_{s,n}^{(t)}
    	$
    	on any substrings $\mathcal{A}$, $\overline{\mathcal{A}}$ of at
    	least $q_{n}/R_{n}^{c}$ consecutive symbols in $w$ and $\overline{w}$,
    	respectively.
    	\item For $w,\overline{w}\in\mathcal{W}_{n}$ with $w\neq\overline{w}$, we have
    	$
    		\overline{f}\left(\mathcal{A},\overline{\mathcal{A}}\right)>\beta_{n}^{(t)}
    	$
    	on any substrings $\mathcal{A}$, $\overline{\mathcal{A}}$ of at
    	least $q_{n}/R_{n}^{c}$ consecutive symbols in $w$ and $\overline{w}$,
    	respectively.
    \end{itemize}
Moreover, the integer parameters $\{e(m)\}^{n}_{m=1}$, $\{\mathfrak{p}_m\}^{n+1}_{m=1}$, $\{R_m\}^{n+1}_{m=1}$ are given, where
    \begin{equation} \label{eq:pN+1}
    	\mathfrak{p}_{n+1}>  4R_n,
    \end{equation}
    and 
    \begin{equation} \label{eq:RN+1} 
    	R_{n+1}\geq \frac{40\mathfrak{p}_{n+1}}{\beta^{(t)}_n}.
    \end{equation}
Furthermore, we assume
\begin{equation} \label{eq:RN}
	R_{n}\geq \frac{7}{\beta^{(t)}_n}.
\end{equation}

   We start the {\bf induction step} by choosing an integer $e(n+1) >e(n)$ such that
   \begin{equation}\label{eq:eN}
   	2^{e(n+1)} > \max \left( 10R_{n+1}, \max_{s \leq s(n+1)} \abs{G^{n+1}_s}\right). 
   \end{equation} 

We also choose the parameters $\mathfrak{p}_{n+2}, R_{n+2} \in \mathbb{Z}^+$ sufficiently large such that
\begin{equation}\label{eq:pN+2}
	\mathfrak{p}_{n+2} > 4 R_{n+1}
\end{equation}
and
\begin{equation} \label{eq:RN+2}
	R_{n+2} > \frac{40^2\cdot \mathfrak{p}_{n+2}\cdot \mathfrak{p}_{n+1}}{\beta^{(t)}_{n}}.
\end{equation}

We distinguish between the two possible cases $s(n+1)=s(n)$ and $s(n+1)=s(n)$ for constructing $(n+1)$-words and extending $G^n_s$ actions to $G^{n+1}_s$ actions.

{\bf Case 1: $s(n+1)=s(n)$.} We apply the general substitution step from Section \ref{subsec:Substitution} in an iterative manner as in \cite[section 8.1.1]{GK3}. To carry out the successive application of the substitution step we
	always set $\tilde{R}=R_{n+1}$ and $D=2^{n+2}q_n$. We introduce $t_{i}\in\mathbb{N}$
	for $i=1,\dots,s(n)+1$ such that $t_{s(n)+1}=0$ and $2^{t_{i}}=|\ker(\rho_{i,i-1}^{(n-1)})|$ for $i=1,\dots,s(n)$.
	Here, we recall that $\rho_{i,i-1}^{(n-1)}:G_{i}^{n-1}\to G_{i-1}^{n-1}$
	is the canonical homomorphism defined in Subsection \ref{subsec:spec}. Let $t^{\ast}\coloneqq \max_{i=1,\dots,s(n)+1} t_i$. Then we choose 
	$
	U_1 =2^{n+2}q_n2^{U_1^{\prime}+2t^{\ast}},
	$
	where $U^{\prime}_1 \in \mathbb{Z}^+$ satisfies
	$
		2^{U^{\prime}_1} \geq \max\left(2R^2_{n+1},2^{8e(n)}\right)$.
 Notably, condition \eqref{eq:U} is satisfied in all construction steps. Following the successive application of the substitution step as in \cite[section 8.1.1]{GK3} gives us the collection of odometer-based words $\mathtt{W}_{n+1}$ and $\mathtt{W}_{n+1}/\mathcal{Q}^{n+1}_s$ for all $s=0,\dots,s(n)$. In particular, this determines $k_n$, $h_{n+1}=k_nh_n$, and $s_{n+1}=\abs{\mathtt{W}_{n+1}}=2^{4e(n+1)\cdot (s(n)+1)}$. 
We define the group actions for $G^{n+1}_s$, $s=0,\dots,s(n)$, as in \cite[section 8.1.2]{GK3}.
    
    {\bf Case 2: $s(n+1)=s(n)+1$.} This time, we apply the general substitution step from Section \ref{subsec:Substitution} in an iterative manner as in \cite[section 8.2]{GK3}. We start the construction by
    choosing the number $J_{s(n)+1,n+1}=2^{J_{s(n)+1,n+1}^{\prime}}$ appearing in specification \ref{item:Q4} on $\mathcal{Q}_{s(n)+1}^{n+1}$ to be a sufficiently
    large integer such that
    $
    	J_{s(n)+1,n+1}>2R_{n+1}^{2}.
    $ 
    We set $\tilde{R}=2J_{s(n)+1,n+1}^{3}$ and $D=2^{n+2}q_n$ in each substitution
    step and choose 
    $
    U_1 =2^{n+2}q_n2^{U_1^{\prime}+2t^{\ast}},
    $
    where $U^{\prime}_1 \in \mathbb{Z}^+$ satisfies
    $
    	2^{U^{\prime}_1} \geq \max\left(8J_{s(n)+1,n+1}^{6},2^{8e(n)}\right).
    $ 
    As before, condition \eqref{eq:U} is satisfied in all construction steps. Then we construct two collections
    $\mathcal{W}^{\dagger}$ and $\mathcal{W}^{\dagger\dagger}$ of concatenation
    of $n$-words as in sections 8.2.1 and 8.2.2 of \cite{GK3}. As in \cite[8.2.3]{GK3} we use them to build $\mathtt{W}_{n+1}$. This determines $k_n$, $h_{n+1}=k_nh_n$, and $s_{n+1}=\abs{\mathtt{W}_{n+1}}=2^{4e(n+1)\cdot (s(n)+2)}$.     
    Finally, we define the new equivalence relation $\mathcal{Q}^{n+1}_{s(n)+1}$ as in \cite[8.2.4]{GK3} and the actions for $G^{n+1}_s$, $s=0,\dots,s(n)+1$, as in \cite[8.2.5]{GK3}. 

\begin{rem}\label{rem:VerifyWM}
	As in \cite{GK3} these constructions satisfy the specifications stated in Section \ref{subsec:spec}. In both cases we iteratively apply the general substitution step with $D=2^{n+2}q_n$. Hence, $k_n$ is of the form $k_n=2^{n+2}q_nC_n$ for some $C_n \in \mathbb{Z}^+$ as a multiple of $s^2_n$ as required in \eqref{eq:k} for our twisted AbC constructions. Since the odometer-based words in $\mathtt{W}_{n+1}$ from above determine the combinatorics of the abstract conjugation map $h_{n+1}$, we can verify requirements \ref{item:R1}, \ref{item:R3}, and \ref{item:R4} from properties of odometer-based words in $\mathtt{W}_{n+1}$:
	Since $s_{n+1} = \abs{\mathcal{W}_{n+1}} = \abs{\mathtt{W}_{n+1}}$ is a multiple of $s_n$, assumption \ref{item:R1} holds. At each application of the substitution step we can apply part~(3) of Proposition \ref{prop:PropertiesCollection} with $D=2^{n+2}q_n$. By induction we deduce that \ref{item:R4} holds. Clearly, \ref{item:R3} holds because different $(n+1)$-blocks in $\mathtt{W}_{n+1}$ are obtained by different concatenations of $n$-blocks from $\mathtt{W}_n$.
\end{rem}
 
We recall that the odometer-based words $\mathtt{W}_{n+1}$ from above determine the combinatorics of the abstract conjugation map $h_{n+1}$. Finally, we choose the parameter $l_n \in \mathbb{Z}^+$ sufficiently large to allow the smooth or real-analytic realization of the twisting system with odometer-based words $\mathtt{W}_{n+1}$ by Theorem \ref{theo:TwistRealSmooth} or its real-analytic counterpart. Additionally, we can choose $l_n$ large enough to satisfy
\begin{equation}
	l_{n}\geq\max\left(4R_{n+2}^{2},9l_{n-1}^{2}\right).\label{eq:lCondTwist}
\end{equation}
Then we also know
$
q_{n+1}=k_nl_nq^2_n
$
from the relation \eqref{eq:alpha}. 

We are now ready to construct $\mathcal{W}_{n+1}$. Since we have determined the parameters $l_n$ and $C_n$, we can define the twisting operator $\mathcal{C}^{\text{twist}}_n$ according to equation \eqref{eq:TwistOperator}. Using the bijection $\kappa_{n}: \mathtt{W}_{n} \to \mathcal{W}_{n}$ from the induction assumption we define
\[
\mathcal{W}_{n+1}=\left\{ \mathcal{C}^{\text{twist}}_{n}\left(\kappa_{n}\left(\mathtt{w}_{0}\right),\kappa_{n}\left(\mathtt{w}_{1}\right),\dots,\kappa_{n}\left(\mathtt{w}_{k_{n}-1}\right)\right)\::\:\mathtt{w}_{0}\mathtt{w}_{1}\dots \mathtt{w}_{k_{n}-1}\in\mathtt{W}_{n+1}\right\} 
\]
and the map $\kappa_{n+1}$ by setting 
\[
\kappa_{n+1}\left(\mathtt{w}_{0}\mathtt{w}_{1}\dots \mathtt{w}_{k_{n}-1}\right)=\mathcal{C}^{\text{twist}}_{n}\left(\kappa_{n}\left(\mathtt{w}_{0}\right),\kappa_{n}\left(\mathtt{w}_{1}\right),\dots,\kappa_{n}\left(\mathtt{w}_{k_{n}-1}\right)\right).
\]
In particular, the prewords are 
\[
P_{n+1}=\left\{ \kappa_{n}\left(\mathtt{w}_{0}\right)\kappa_{n}\left(\mathtt{w}_{1}\right)\dots \kappa_{n}\left(\mathtt{w}_{k_{n}-1}\right)\::\:\mathtt{w}_{0}\mathtt{w}_{1}\dots \mathtt{w}_{k_{n}-1}\in\mathtt{W}_{n+1}\right\} .
\]

Next, we aim at $\overline{f}$-estimates in the twisted system. For that purpose, we follow the $\overline{f}$-estimates in \cite[section 10.4]{GK3}, where the associated circular systems are analyzed. In that analysis the newly introduced spacers $b$ and $e$ are always ignored. In particular, the explicit form of the circular operator does not matter as long as each $0$-subsection $w^{l_n -1}_j$ comes with $q_n$ many newly introduced spacers, that is, we have a string of the form $b^{q_n - i}w^{l_n -1}_je^i$ for some $0\leq i < q_n$ as an $1$-subsection (recall the terminology from Remark \ref{rem:Subsection}). Since our twisting operator is an operator of such type, we can apply the same analysis to conclude the following analogue of \cite[Lemma 89]{GK3}.

\begin{lem}
	\label{lem:TwistBlock}Let $w,\overline{w}\in\mathcal{W}_{n+1}$
	and $\mathcal{B}$, $\overline{\mathcal{B}}$ be any substrings of
	at least $q_{n+1}/R_{n+1}^{c}$ consecutive symbols in $w$
	and $\overline{w}$, respectively. 
	\begin{enumerate}
		\item If $s(n+1)=s(n)$, then we have for every $s\in \{1,\dots , s(n)\}$ that
		\begin{equation}
			\overline{f}\left(\mathcal{B},\overline{\mathcal{B}}\right)>\begin{cases}
				\alpha_{s,n}^{(t)}-\frac{2}{R_{n}}-\frac{3}{R_{n+1}}, & \text{if }[w]_{s}\neq[\overline{w}]_{s},\\
				\beta_{n}^{(t)}-\frac{2}{R_{n}}-\frac{3}{R_{n+1}}, & \text{if }w\neq\overline{w}.
			\end{cases}\label{eq:TwistBlock1}
		\end{equation}
		\item If $s(n+1)=s(n)+1$, then we have for every $s\in \{1,\dots , s(n)\}$ that
		\begin{equation}
			\overline{f}\left(\mathcal{B},\overline{\mathcal{B}}\right)>\begin{cases}
				\alpha_{s,n}^{(t)}-\frac{2}{R_{n}}-\frac{4}{R_{n+1}}, & \text{if }[w]_{s}\neq[\overline{w}]_{s},\\
				\beta_{n}^{(t)}-\frac{2}{R_{n}}-\frac{4}{R_{n+1}}, & \text{if }[w]_{s(n)+1}\neq[\overline{w}]_{s(n)+1},\\
				\frac{1}{2\mathfrak{p}_{n+1}}\cdot\left(\beta_{n}^{(t)}-\frac{1}{R_{n}^{c}}-\frac{2}{\sqrt{l_{n}}}\right)-\frac{3}{R_{n+1}}, & \text{if }w\neq\overline{w}.
			\end{cases}\label{eq:TwistBlock2}
		\end{equation}
	\end{enumerate}
\end{lem}

From $R_{n}^{c}=\lfloor\sqrt{l_{n-2}}\cdot k_{n-2}\cdot q_{n-2}^{2}\rfloor$ and the growth condition \eqref{eq:lCondTwist} on the sequence $(l_n)_{n \in \mathbb{N}}$ we obtain
\[
\frac{1}{R_{n}^{c}}+\frac{2}{\sqrt{l_{n}}}\leq\frac{1}{\sqrt{l_{n-2}}}+\frac{1}{\sqrt{l_{n-1}}}\leq\frac{2}{\sqrt{l_{n-2}}}\leq\frac{1}{R_{n}}\leq \frac{\beta^{(t)}_n}{7},
\]
where we used assumption \eqref{eq:RN} in the last estimate. Then we note with the aid of \eqref{eq:RN+1} that 
\[
\beta_{n+1}^{(t)}>\frac{\beta_{n}^{(t)}}{4\mathfrak{p}_{n+1}}-\frac{3}{R_{n+1}}\geq\frac{7\beta_{n}^{(t)}}{40\mathfrak{p}_{n+1}}\geq\frac{7}{R_{n+1}},
\]
that is, $R_{n+1}>\frac{7}{\beta_{n+1}^{(t)}}$ and the induction assumption \eqref{eq:RN} for the next step is satisfied. Furthermore, the condition \eqref{eq:RN+2} on $R_{n+2}$ allows us to estimate
\[
\beta_{n+1}^{(t)}>\frac{7\beta_{n}^{(t)}}{40\mathfrak{p}_{n+1}}\geq\frac{40\mathfrak{p}_{n+2}}{R_{n+2}},
\]
that is, $R_{n+2} \geq 40\mathfrak{p}_{n+2}/\beta_{n+1}^{(t)}$ and, hence, the induction assumption \eqref{eq:RN+1} for the next step is satisfied. The induction assumption \eqref{eq:pN+1} for the next step was already fulfilled by our choice of $\mathfrak{p}_{n+2}$ in \eqref{eq:pN+2}. This accomplishes the induction step. 

Following this inductive construction process, we can prove the following analogue of \cite[Proposition 90]{GK3}.

\begin{prop}\label{prop:fTwist}
	For every $s \in \Z^+$ there is $\alpha^{(t)}_s>0$ such that for every $n\geq M(s)$ we have
	$
		\overline{f}\left(\mathcal{A},\overline{\mathcal{A}}\right)>\alpha_{s}^{(t)}
	$
	on any substrings $\mathcal{A}$, $\overline{\mathcal{A}}$ of at
	least $q_{n}/R_{n}^{c}$ consecutive symbols in $w$ and $\overline{w}$,
	respectively, for $w,\overline{w}\in\mathcal{W}_{n}$ with $[w]_{s}\neq[\overline{w}]_{s}$.
\end{prop}    

\begin{proof}
The estimate in \eqref{eq:TwistBlock2} implies
\[
\alpha_{s,M(s)}^{(t)}=\beta_{M(s)-1}^{(t)}-\frac{2}{R_{M(s)-1}}-\frac{4}{R_{M(s)}}\geq\beta_{M(s)-1}^{(t)}-\frac{2}{7}\beta_{M(s)-1}^{(t)}-\frac{\beta_{M(s)-1}^{(t)}}{10\mathfrak{p}_{M(s)}}>\frac{\beta_{M(s)-1}^{(t)}}{2}
\]
by \eqref{eq:RN+1} and \eqref{eq:RN}. From the estimates in \eqref{eq:TwistBlock1} and \eqref{eq:TwistBlock2} we obtain
\[
\alpha_{s,n}^{(t)}\geq\alpha_{s,M(s)}^{(t)}-\frac{2}{R_{M(s)}}-\sum_{j>M(s)}\frac{6}{R_{j}}\geq\alpha_{s,M(s)}^{(t)}-\frac{\beta_{M(s)-1}^{(t)}}{20\mathfrak{p}_{M(s)}}-\frac{6\beta_{M(s)-1}^{(t)}}{40}>\frac{\alpha_{s,M(s)}^{(t)}}{2}
\]
for any $n>M(s)$ by conditions \eqref{eq:RN+1} and \eqref{eq:RN}. Hereby, we see
that for every $s\in\mathbb{N}$ the decreasing sequence $\left(\alpha_{s,n}^{(t)}\right)_{n\geq M(s)}$
is bounded from below by $\frac{\alpha_{s,M(s)}^{(t)}}{2}$, and we denote
its limit by $\alpha_{s}^{(t)}$. This proves the proposition.
\end{proof}

\section{Proof of Theorem \ref{thm:criterion}} \label{sec:Proof}
	We verify that the map $\Phi:\mathcal{T}\kern-.5mm rees\to \text{Diff}^{\,\infty}_{\,\lambda}(M)$ defined in the previous section satisfies the properties required in Theorem~\ref{thm:criterion}. 
	
	\begin{lem}\label{lem:ContinuityPhi}
		The map $\Phi:\mathcal{T}\kern-.5mm rees\to \text{Diff}^{\,\infty}_{\,\lambda}(M)$ is continuous.
	\end{lem} 

\begin{proof}
	Let $T^{\mathfrak{s}}=\Phi(\mathcal{T})$ for $\mathcal{T}\in\mathcal{T}\kern-.5mm rees$ and
	$U$ be an open neighborhood of $T^{\mathfrak{s}}$ in $\text{Diff}^{\,\infty}_{\,\lambda}(M)$.
	Suppose $T$ is the twisted system such that $R(T)=T^{\mathfrak{s}}$. Here, $R$ denotes the smooth realization map from Theorem \ref{theo:TwistRealSmooth} that assigns to each twisted symbolic system an isomorphic smooth diffeomorphism. Moreover, let $\left(\mathcal{W}_{n}\right)_{n\in\mathbb{N}}$ denote
	the construction sequences of $T$.	
	By Lemma \ref{lem:closeSmooth} there is $M\in\mathbb{N}$
	sufficiently large such that for all $S^{\mathfrak{s}}\in\text{Diff}^{\,\infty}_{\,\lambda}(M)$
	in the range of $R$ we have the following property: If the construction sequence $\left(\mathcal{U}_{n}\right)_{n\in\mathbb{N}}$ for the twisted system $S$ with $S^{\mathfrak{s}}=R(S)$ satisfies $\left(\mathcal{W}_{n}\right)_{n\leq M}=\left(\mathcal{U}_{n}\right)_{n\leq M}$,
	then $S^{\mathfrak{s}}\in U$. By \eqref{eq:ContinuityF}
	there is a basic open set $V\subseteq\mathcal{T}\kern-.5mm rees$ containing
	$\mathcal{T}$ such that for all $\mathcal{S}\in V$ the first $M+1$
	members of the construction sequences are the same, that is, $\left(\mathcal{W}_{n}(\mathcal{T})\right)_{n\leq M}=\left(\mathcal{W}_{n}(\mathcal{S})\right)_{n\leq M}$.
	Then it follows that $\Phi(\mathcal{S})\in U$ for all $\mathcal{S}\in V$
	which yields the continuity of $\Phi$ by Fact \ref{fact:contTree}.
\end{proof}

	\subsection{Proof of the weak mixing property}
	In order to prove weak mixing of $T=\Phi(\mathcal{T})$ for any $\mathcal{T}\in \trees$, we check that $T$ fulfills the assumptions of Proposition \ref{prop:WM2}. By construction, $T$ is obtained as a limit of AbC transformations with parameters $k_n =2^{n+2}q_nC_n$ and $l_n$ satisfying $\sum_{n\in \N}1/l_n <\infty$. We verified assumptions \ref{item:R1} and \ref{item:R4} in Remark \ref{rem:VerifyWM}. Altogether, we can apply Proposition \ref{prop:WM2} and conclude that $T$ is weakly mixing.

	
	\subsection{\label{subsec:Isom}Infinite branches give isomorphisms}
	Similarly to \cite[section~5]{GK3} the specifications in Section \ref{subsec:spec} and the structure of the twisting operator allow us to build an isomorphism between $\Phi(\mathcal{T})$
	and $\Phi(\mathcal{T})^{-1}$ in case $\mathcal{T}$
	has an infinite branch.
	\begin{lem}
		\label{lem:siso}Let $s\in\mathbb{N}$ and $g\in G_{s}^{m}$ for some
		$m\in\mathbb{N}$. Suppose that $g$ has odd parity. Then there is 
		a shift-equivariant isomorphism $\eta_{g}:\mathbb{K}^{\text{twist}}_{s}\to rev(\mathbb{K}^{\text{twist}}_{s})$ canonically associated
		to $g$ as in \eqref{eq:giso}. 	
		Moreover, if $s'>s$ and $g^{\prime}\in G_{s'}^{n}$ for some $n\geq m$
		with $\rho_{s',s}(g^{\prime})=g$, then $\pi_{s',s}\circ\eta_{g^{\prime}}=\eta_{g}\circ\pi_{s',s}$.
	\end{lem}
	
	\begin{proof}
		We recall from Remark \ref{rem:ConstrSeqKs} that $\left((\mathcal{W}_{m})_{s}^{\ast}\right)_{m\geq M(s)}$
		defines a construction sequence for $\mathbb{K}^{\text{twist}}_{s}$. Similarly,
		$\left(rev\left((\mathcal{W}_{m})_{s}^{\ast}\right)\right)_{m\geq M(s)}$
		is a construction sequence for $rev(\mathbb{K}^{\text{twist}}_{s})$. For
		$n>m$ we can write any element
		$[w]_{s}^{\ast}\in(\mathcal{W}_{n})_{s}^{\ast}$ as 
		\[
		[w]_{s}^{\ast}=\mathcal{C}^{\text{twist}}_{m,n}\left([w_{0}]_{s}^{\ast},[w_{1}]_{s}^{\ast},\dots,[w_{h_n/h_m-1}]_{s}^{\ast}\right)
		\]
		with $[w_{i}]_{s}^{\ast}\in(\mathcal{W}_{m})_{s}^{\ast}$. Lemma \ref{lem:groupTwistInd} based on the twisted skew diagonal action yields
		\begin{align*}
		g[w]_{s}^{\ast}=g\mathcal{C}^{\text{twist}}_{m,n}\left([w_{0}]_{s}^{\ast},\dots,[w_{h_n/h_m-1}]_{s}^{\ast}\right)=\mathcal{C}^{\text{twist}}_{m,n}\left(g[w_{h_n/h_m-1}]_{s}^{\ast},\dots,g[w_{0}]_{s}^{\ast}\right).	
		\end{align*}
		As pointed out in Remark \ref{rem:ClosedUnderTwistSkew}, $(\mathcal{W}_{n})_{s}^{\ast}$
		is closed under the twisted skew diagonal action. Thus 
		$
		\mathcal{C}^{\text{twist}}_{m,n}\left(g[w_{h_n/h_m-1}]_{s}^{\ast},\dots,g[w_{1}]_{s}^{\ast},g[w_{0}]_{s}^{\ast}\right)\in(\mathcal{W}_{n})_{s}^{\ast},
		$
		which implies 
		$
		rev\left(\mathcal{C}^{\text{twist}}_{m,n}\left(g[w_{h_n/h_m-1}]_{s}^{\ast},\dots,g[w_{1}]_{s}^{\ast},g[w_{0}]_{s}^{\ast}\right)\right)\in rev\left((\mathcal{W}_{n})_{s}^{\ast}\right).
		$
		With the aid of Lemma \ref{lem:RevTwistInd} we obtain that
		$
		\widetilde{\mathcal{C}}^{\text{twist}}_{m,n}\left(g[w_{0}]_{s}^{\ast},g[w_{1}]_{s}^{\ast},\dots,g[w_{h_n/h_m-1}]_{s}^{\ast}\right)\in rev\left((\mathcal{W}_{n})_{s}^{\ast}\right).
		$
		Hence, 
		\begin{equation} \label{eq:giso}
			\begin{split}
				\mathcal{C}^{\text{twist}}_{m,n}\left([w_{0}]_{s}^{\ast},\dots,[w_{h_n/h_m-1}]_{s}^{\ast}\right) 
			\mapsto  \widetilde{\mathcal{C}}^{\text{twist}}_{m,n}\left(g[w_{0}]_{s}^{\ast},\dots,g[w_{h_n/h_m-1}]_{s}^{\ast}\right)
			\end{split}  
		\end{equation}
		is an invertible map from the construction sequence for $\mathbb{K}^{\text{twist}}_{s}$ to the construction
		sequence for $rev(\mathbb{K}^{\text{twist}}_{s})$. It can also be interpreted as a shift-equivariant map from cylinder sets in $\mathbb{K}^{\text{twist}}_{s}$ to cylinder sets located in the same position in $rev(\mathbb{K}^{\text{twist}}_{s})$.  This yields the isomorphism in the 
		first assertion.
		
		The second assertion follows from specification \ref{item:A7}, which says that the
		action by $g^{\prime}$ is subordinate to the action by $g$ via the
		homomorphism $\rho_{s',s}$. 
	\end{proof}

In the following we call a sequence of isomorphisms $\zeta_{s}$ between
$\mathbb{K}^{\text{twist}}_{s}$ and $rev(\mathbb{K}^{\text{twist}}_{s})$ \emph{coherent} if $\pi^{\text{twist}}_{s+1,s}\circ\zeta_{s+1}=\zeta_{s}\circ\pi^{\text{twist}}_{s+1,s}$
for every $s\in\mathbb{N}$.
\begin{lem}
	\label{lem:iso}Let $\left(\zeta_{s}\right)_{s\in\mathbb{N}}$ be a
	coherent sequence of isomorphisms between $\mathbb{K}^{\text{twist}}_{s}$ and $rev(\mathbb{K}^{\text{twist}}_{s})$.
	Then there is an isomorphism $\zeta:\mathbb{K}^{\text{twist}}\to rev(\mathbb{K}^{\text{twist}})$
	such that $\pi^{\text{twist}}_{s}\circ\zeta=\zeta_{s}\circ\pi^{\text{twist}}_{s}$ for every $s\in\mathbb{N}$.
\end{lem}

\begin{proof}
	Since the sequence of isomorphisms $(\zeta_s)_{s\in \N}$ is coherent, their inverse limit defines
	a measure-preserving isomorphism between the subalgebra of $\mathcal{B}\left(\mathbb{K}^{\text{twist}}\right)$
	generated by $\bigcup_{s}\mathcal{H}^{\text{twist}}_{s}$ and the subalgebra of $\mathcal{B}\left(rev(\mathbb{K}^{\text{twist}})\right)$
	generated by $\bigcup_{s}rev(\mathcal{H}^{\text{twist}}_{s})$. By Lemma \ref{lem:algebra}
	this extends uniquely to a measure-preserving isomorphism $\tilde{\zeta}$
	between $\mathcal{B}\left(\mathbb{K}^{\text{twist}}\right)$ and $\mathcal{B}\left(rev(\mathbb{K}^{\text{twist}})\right)$.
	Then by part (1) of Lemma \ref{lem:subalgebra} we can find sets $D\subset\mathbb{K}^{\text{twist}}$,
	$D'\subset rev(\mathbb{K}^{\text{twist}})$ of measure zero such that $\tilde{\zeta}$
	determines a shift-equivariant isomorphism $\zeta$ between $\mathbb{K}^{\text{twist}}\setminus D$
	and $rev(\mathbb{K}^{\text{twist}})\setminus D'$.
\end{proof}
\begin{proof}[Proof of part (1) in Theorem \ref{thm:criterion}]
	Suppose that $\mathcal{T}\in\mathcal{T}\kern-.5mm rees$ has an infinite branch.
	Then $G_{\infty}(\mathcal{T})$ has an element $g$ of odd parity
	according to Fact \ref{fact:oddElement}. By Lemma \ref{lem:siso} we obtain
	a coherent sequence of isomorphisms $\zeta_s \coloneqq \eta_{\rho_{s}(g)}$ between
	$\mathbb{K}^{\text{twist}}_{s}(\mathcal{T})$ and $rev(\mathbb{K}^{\text{twist}}_{s}(\mathcal{T}))$. Hence, Lemma \ref{lem:iso}
	yields an isomorphism between $\mathbb{K}^{\text{twist}}(\mathcal{T})$ and $rev(\mathbb{K}^{\text{twist}}(\mathcal{T}))$.
	Since $rev(\mathbb{K}^{\text{twist}}(\mathcal{T}))$ is isomorphic to $(\mathbb{K}^{\text{twist}}(\mathcal{T}))^{-1}$, we conclude
	that $\mathbb{K}^{\text{twist}}(\mathcal{T})\cong(\mathbb{K}^{\text{twist}}(\mathcal{T}))^{-1}$. Since our smooth realization of twisted systems in Theorem \ref{theo:TwistRealSmooth} preserves isomorphism, we conclude that $\Phi(\mathcal{T}) \cong \Phi(\mathcal{T})^{-1}$ in case that the tree $\mathcal{T}$ has an infinite branch.
\end{proof}

    \subsection{\label{subsec:Non-Equiv}Proof of Non-Kakutani Equivalence}
    In \cite[sections 10.4 and 10.5]{GK3} it is shown that for a tree $\mathcal{T} \in \trees$ without an infinite branch the circular system $T_c = \mathcal{F}(\Psi(\mathcal{T}))$ and $T^{-1}_c = \mathcal{F}(\Psi(\mathcal{T}))^{-1}$ are not Kakutani equivalent. Within the required $\overline{f}$ estimates, the newly introduced spacers $b$ and $e$ are always ignored. As already observed in Section \ref{subsec:Construction}, the explicit form of the circular operator does not matter as long as each $0$-subsection $w^{l_n -1}_j$ comes with $q_n$ many newly introduced spacers, that is, we have a string of the form $b^{q_n - i}w^{l_n -1}_je^i$ for some $0\leq i < q_n$ as an $1$-subsection. Since our twisting operator is an operator of such type, we can apply the same analysis to conclude that $\mathbb{K}^{\text{twist}}(\mathcal{T})$ and $(\mathbb{K}^{\text{twist}}(\mathcal{T}))^{-1}$ are not Kakutani equivalent if $\mathcal{T}$ does not have an infinite branch. Since our Theorem \ref{theo:TwistRealSmooth} produces diffeomorphisms isomorphic to the symbolic systems, we conclude that $\Phi(\mathcal{T})$ and $\Phi(\mathcal{T})^{-1}$ are not Kakutani equivalent if $\mathcal{T}$ does not have an infinite branch. This proves part (2) of Theorem \ref{thm:criterion}.  
    Altogether we completed the proof of Theorem \ref{thm:criterion}.
    
    \begin{proof}[Proof of Theorem~\ref{thm:analytic}]
    	We substitute the real-analytic counterparts described in Remark~\ref{rem:analyticCounterpart} for Theorem~\ref{theo:TwistRealSmooth} and Lemma~\ref{lem:closeSmooth}.
    \end{proof}

	\subsection*{Acknowledgements}
	The author would like to thank Shilpak Banerjee, Matthew Foreman, Marlies Gerber, and Jean-Paul Thouvenot for helpful comments and discussion. He thanks the referees for careful advice.

\end{document}